\documentclass{amsart}
\usepackage{mathrsfs}
\usepackage{amssymb,latexsym,amsmath,amsthm,enumitem,mathrsfs}
\usepackage[margin=1in]{geometry}
\usepackage{color}
\usepackage{stmaryrd}
\usepackage{lineno}
\usepackage{bbm}   
\usepackage{scalerel,stackengine}
\usepackage{ esint }
\usepackage{graphicx}
\usepackage{amssymb}

\newcommand{\M}{\mathbb{M}}

\newcommand{\R}{\mathbb R}

\newcommand{\Z}{\mathbb Z}
\newcommand{\C}{\mathbb C}
\newcommand{\N}{\mathbb N}
\newcommand{\T}{\mathbb T}

\newcommand{\w}{\omega}

\newcommand{\e}{\varepsilon}
\newcommand{\g}{\gamma}
\newcommand{\p}{\varphi}
\newcommand{\s}{\psi}

\renewcommand{\a}{\alpha}
\renewcommand{\b}{\beta}
\renewcommand{\d}{\delta}

\newcommand{\mc}{\mathcal}
\newcommand{\mb}{\mathbb}

\def\set4{\mathcal I}
\def\tup14{(1,2,3,4)}

\newcommand\vwidehat[1]{\arraycolsep=0pt\relax%
\begin{array}{c}
\stretchto{
  \scaleto{
    \scalerel*[\widthof{\ensuremath{#1}}]{\kern-.5pt\bigwedge\kern-.5pt}
    {\rule[-\textheight/2]{1ex}{\textheight}} 
  }{\textheight} %
}{0.5ex}\\           
#1\\                 
\rule{-1ex}{0ex}
\end{array}
}
\newtheorem*{comm*}{Comment}
\newtheorem*{rmk}{Remark}
\newtheorem{definition}{Definition}
\newtheorem*{lemma*}{Lemma}
\newtheorem*{theorem*}{Theorem}

\newtheorem*{cor*}{Corollary}
\newtheorem{theorem}{Theorem}
\newtheorem{thm}{Theorem}[section]
\newtheorem{corollary}[thm]{Corollary}

\newtheorem{proposition}[thm]{Proposition}
\newtheorem*{proposition*}{Proposition}

\newtheorem{lemma}[thm]{Lemma}

\usepackage{ bbold }
\stackMath
\newcommand\widecheck[1]{%
\savestack{\tmpbox}{\stretchto{%
  \scaleto{%
    \scalerel*[\widthof{\ensuremath{#1}}]{\kern-.6pt\bigwedge\kern-.6pt}%
    {\rule[-\textheight/2]{1ex}{\textheight}}
  }{\textheight}%
}{0.5ex}}%
\stackon[1pt]{#1}{\scalebox{-1}{\tmpbox}}%
}
\usepackage{ esint }

\newcommand{\supp}{\mathrm{supp\,}}

\begin{document}

\author{Larry Guth}
\address{Department of Mathematics\\
Massachusetts Institute of Technology\\
Cambridge, MA 02142-4307, USA}
\email{lguth@math.mit.edu}

\author{Dominique Maldague}
\address{Department of Mathematics\\
Massachusetts Institute of Technology\\
Cambridge, MA 02142-4307, USA}
\email{dmal@mit.edu}


\date{}

\title{A sharp square function estimate for the moment curve in $\R^n$ }
\maketitle 

\begin{abstract}
    We use high-low frequency methods developed in the context of decoupling to prove sharp (up to $C_\e R^\e$) square function estimates for the moment curve $(t,t^2,\ldots,t^n)$ in $\R^n$. Our inductive scheme incorporates sharp square function estimates for auxiliary conical sets, 
    which allows us to fully exploit lower dimensional information. 
\end{abstract}

\section{Introduction}

Let $\g_n:[0,1]\to\R^n$ be the $n$-dimensional moment curve defined by $(t,t^2,\ldots,t^n)$ and consider functions whose Fourier transforms are supported in a neighborhood of the moment curve. Writing the neighborhood as a finitely-overlapping union $\cup\theta$, we study square function estimates of the form 
\begin{equation}\label{sqfnform}
\int_{\R^n}|\sum_\theta f_\theta|^p\le S(\cup\theta)\int_{\R^n}(\sum_\theta|f_\theta|^2)^{\frac{p}{2}} 
\end{equation}
for any Schwartz $f_\theta:\R^n\to\C$ with $\supp \widehat{f}_\theta\subset \theta$. We prove bounds of the form $S(\cup\theta)\le C_\e (\#\theta)^\e$ for any $\e>0$, for the optimal range of exponents $2\le p\le p_{n}$. When $p=2$, \eqref{sqfnform} follows directly from Plancherel's theorem. When $n=2$ and $p=4$, \eqref{sqfnform} follows from a counting argument \cite{feffL4,cordoba}. This approach was extended to higher dimensions in \cite{gpsqfn}, but does not apply for non-even integers and only gives the sharp range of exponents when $n=2$.

This manuscript builds on the work in \cite{locsmooth} for the cone in $\R^3$ and \cite{maldagueM3} for the moment curve in $\R^3$ by further developing tools from decoupling theory 
to prove sharp square function estimates. Although the critical exponent for the cone in $\R^3$ is $p=4$, counting arguments are not sufficient to prove the sharp square function estimate. Guth, Wang, and Zhang introduced tools like wave envelope estimates which gave very strong information about the cone in $\R^3$. However, there was and is the question whether these ideas apply to any other manifolds besides the $2$-dimensional cone in $\R^3$. This paper generalizes these ideas to all moment curves, cones over moment curves, and some other shapes. 

The proof is by induction on dimension.  In order to make the induction work, we need to work with not only moment curves and cones over moment curves but also more general shapes which we call $m${th} order Taylor cones.  Figuring out the right shapes to include to make the induction work is one of the main new ingredients in the paper.

For $R\in 2^{n\N}$, define the anisotropic neighborhood $\mc{M}^n(R)\subset\R^n$ of the $n$-dimensional moment curve by 
\begin{equation}\label{mmomcurve}
\mc{M}^n(R):=\big\{\g_n(t)+\sum_{j=1}^n\lambda_j\g_n^{(j)}(t):t\in[0,1],\quad|\lambda_j|\le R^{-\frac{j}{n}} \big\}. 
\end{equation}
Use $I_\theta$ to denote the intervals of length $R^{-\frac{1}{n}}$ with endpoints in $R^{-\frac{1}{n}}\Z\cap[0,1]$ and define $\Theta^n(R)$ to be the collection of $\theta:=\mc{M}^n(R)\cap\{t\in I_\theta\}$, so that $\mc{M}^n(R)=\bigcup_{\theta\in\Theta^n(R)}\theta$. Each $\theta$ is approximately the convex hull of the subarc of $\g_n$ corresponding to the interval $I_\theta$.

Let  $p_{n}=\frac{n(n+1)}{2}+1$. Our main theorem is the following.
\begin{theorem}\label{maintheorem} For any $\e>0$, there exists $C_\e\in(0,\infty)$ so that
\[ \int_{\R^n}|\sum_{\theta\in\Theta^n(R)} f_\theta|^p\le C_\e R^\e \int_{\R^n}|\sum_{\theta\in\Theta^n(R)}|f_\theta|^2|^{\frac{p}{2}} \]
for any $2\le p\le p_n$ and any Schwartz functions $f_\theta:\R^n\to\C$ with $\supp\widehat{f}_\theta\subset\theta$. 
\end{theorem}

Theorem \ref{maintheorem} is the $n$-dimensional generalization of the $3$-dimensional result in \cite{maldagueM3}. Adapting the sharp example discussed in \cite{maldagueM3} to $n$-dimensions is straightforward, and shows that the range of $p$ in Theorem \ref{maintheorem} is sharp. The argument for $\mc{M}^3(R)$ in \cite{maldagueM3} used sharp wave envelope estimates (more refined versions of square function estimates) for the parabola in $\R^2$ and for the cone in $\R^3$. Our inductive scheme for $\mc{M}^n(R)$ involves analyzing more complicated shapes which do not admit wave envelope estimates, but do nevertheless satisfy sharp square function estimates. We provide some descriptions for the geometries and intuition for the inductive scheme in \textsection \ref{geointro}-\ref{intro2}.

Define $\phi_n:[0,1]\to\R^{n+1}$ by $\phi_n(t)=(1,\frac{t}{1!},\frac{t^2}{2!},\ldots,\frac{t^n}{n!})$. For $0\le m\le n-1$, define the $m$th order Taylor cone $\Gamma^{n+1}_m(R)\subset\R^{n+1}$ to be  
\begin{equation}\label{mcone} 
\Gamma_m^{n+1}(R):=\Big\{\sum_{i=0}^n\lambda_i\phi_n^{(i)}(t):t\in[0,1],\quad|\lambda_i|\le R^{-\frac{i}{n}}\quad\forall\,i,\quad \max_{0\le j\le m}|\lambda_j|2^{m+1-j}R^{\frac{j}{n}}\ge 1 \Big\}. \end{equation}
Write $\Gamma_m^{n+1}(R)=\bigcup_{\theta\in\Xi_m^{n+1}(R)}\theta$ where $\Xi_m^{n+1}(R)$ is the collection of subsets $\theta=\Gamma_m^{n+1}(R)\cap\{t\in J_\theta\}$, with $J_\theta$ varying over intervals of length $2^{-100n}R^{-\frac{1}{n}}$ with endpoints $2^{-100n}R^{-\frac{1}{n}}\Z\cap[0,1]$.

We define multi-scale constants which are part of the inductive scheme. For the moment curve, let $\mb{M}^n(R)$ denote that smallest constant such that 
\[ \int_{\R^{n}}|\sum_{\theta\in\Theta^n(R)}f_\theta|^{p}\le \mb{M}^n(R) \int_{\R^{n}}|\sum_{\theta\in\Theta^n(R)}|f_\theta|^2|^{\frac{p}{2}}\]
for any $2\le p\le p_n$ and any Schwartz functions $f_\theta:\R^{n}\to\C$ satisfying $\supp \widehat{f}_\theta\subset\theta\in\Theta^n(R)$. Similarly, for $0\le m<n-1$, define $\C_m^{n+1}(R)$ be the smallest constant such that
\[ \int_{\R^{n+1}}|\sum_{\theta\in\Xi_m^{n+1}(R)}g_\theta|^{p}\le \mb{C}_m^{n+1}(R) \int_{\R^{n+1}}|\sum_{\theta\in\Xi_m^{n+1}(R)}|g_\theta|^2|^{\frac{p}{2}}\]
for any $2\le p\le p_{n-m}$ and any Schwartz functions $g_\theta:\R^{n+1}\to\C$ satisfying $\supp \widehat{g}_\theta\subset\theta\in\Xi^{n+1}_m(R)$. If $m=n-1$, then define $\C_{n-1}^{n+1}(R)$ to be the smallest constant such that 
\[  \int_{\R^{n+1}}|\sum_{\theta\in\Xi_{n-1}^{n+1}(R)}g_\theta|^{p}\le \mb{C}_m^{n+1}(R) \int_{\R^{n+1}}\sum_{\theta\in\Xi_{n-1}^{n+1}(R)}|g_\theta|^p \]
for any $1\le p\le 2$ and any Schwartz functions $g_\theta:\R^{n+1}\to\C$ satisfying $\supp \widehat{g}_\theta\subset\theta\in\Xi^{n+1}_{n-1}(R)$.

Write $a\lesssim b$ to denote $a\le C b$ for some absolute constant $C$, meaning $C$ only depends on the fixed parameters of the set-up. The notation $a\sim b$ means $a\lesssim b$ and $b\lesssim a$. We write $a\lesssim_\e b$ to emphasize that the implicit constant depends on some parameter $\e$. 

The following two propositions together with the $L^4$ square function estimate for the parabola \cite{cordoba,feffL4} as a base case, imply Theorem \ref{maintheorem}. 

\begin{proposition}\label{coneinduct} Let $n\ge 2$. Suppose that for all $\e>0$ and all $R>1$,
\begin{align}
\M^k(R)&\lesssim_\e  R^\e\quad\text{for all}\quad 1\le k\le n \label{conecon1}\\
\C^{k+1}_m(R)&\lesssim_\e R^\e\quad\text{for all}\quad 1\le k\le n-1\quad\text{and}\quad 0\le m\le k-1.     \label{conecon2}
\end{align}
Then for all $\e>0$ and $R>1$, $\C_m^{n+1}(R)\lesssim_\e R^\e$ for any $0\le m\le n-1$. 
\end{proposition}

\begin{proposition}\label{momcurveinduct} Let $n\ge 2$. Suppose that for all $\e>0$ and all $R>1$,
\begin{align*}
\M^k(R)&\lesssim_\e  R^\e\quad\text{for all}\quad 1\le k\le n-1  \\
\C^{k}_m(R)&\lesssim_\e R^\e\quad\text{for all}\quad 1\le k\le n\quad\text{and}\quad 0\le m\le k-1.     
\end{align*}
Then for all $\e>0$ and all $R>1$, $\M^n(R)\lesssim_\e R^\e$. 
\end{proposition}

The first half of the paper (\textsection\ref{prelimsec},\textsection\ref{conesec}) concerns Proposition \ref{coneinduct}. We define local, multi-scale constants and introduce a key iteration to control how different scales are related to one another. 
In the second half of the paper (\textsection\ref{tools}-\ref{mainsec}), we prove Proposition \ref{momcurveinduct} by following a high-low inductive scheme which also incorporates the key iteration. We describe relevant geometries in \textsection\ref{geointro} and features of a general multi-scale approach to proving square function estimates in \textsection\ref{intro1}. In \textsection\ref{intro2}, we describe the key iteration in a special case.


\textbf{Acknowledgements.} LG is supported by a Simons Investigator grant. DM is supported by the National Science Foundation under Award No. 2103249.

\subsection{Description of geometries \label{geointro}} 

Let $0\le m\le n-1$. The following are properties of the $m$th order Taylor cone $\Gamma_m^{n+1}(R)$:
\begin{enumerate}
    \item $\Gamma_m^{n+1}(R)$ is contained in a ball of radius $\sim_n1$ centered at the origin. 
    \item Each $\theta\in\Xi_m^{n+1}(R)$ is a subset of the convex set 
    \[ \overline{\theta}:=\{\sum_{i=0}^n\lambda_i\phi_n^{(i)}(a):|\lambda_i|\lesssim R^{-\frac{i}{n}}\quad\forall i\},\]
    where $\theta$ has associated $t$-interval $[a,a+2^{-100n}R^{-\frac{1}{n}}]$. 
    The displayed set is comparable to (meaning contains and is contained in a dilated copy of) a box of dimensions $1\times R^{-\frac{1}{n}}\times\cdots\times R^{-1}$. To see an element of $\Xi_0^{n+1}(R)$ from $\overline{\theta}$, imagine removing a box of dimensions $1/2\times CR^{-\frac{1}{n}}\times\cdots CR^{-1}$ from $\overline{\theta}$. If we remove a box of dimension $1/2\times R^{-\frac{1}{n}}/2\times\cdots\times R^{-\frac{m}{n}}/2\times CR^{\frac{m+1}{n}}\times\cdots\times CR^{-1}$ from $\overline{\theta}$, then we obtain an element of $\Xi_m^{n+1}(R)$.   
    \item $\Gamma_m^{n+1}(R)$ contains the following copies of $\Gamma_0^{n-m+1}(R^{\frac{n-m}{n}})$: 
    \begin{itemize}
        \item $\Gamma_0^{n-m+1}(R^{\frac{n-m}{n}})\times\{0\}^m$
        \item $\{0\}\times R^{-\frac{1}{n}}\Gamma_0^{n-m+1}(R^{\frac{n-m}{n}})\times \{0\}^{m-1}$ 
        \item $\{0\}^2\times R^{-\frac{2}{n}}\Gamma_0^{n-m+1}(R^{\frac{n-m}{n}})\times\{0\}^{m-2}$
    \newline $\vdots$
        \item $\{0\}^m\times R^{-\frac{m}{n}}\Gamma_{0}^{n-m+1}(R^{\frac{n-m}{n}})$ .
    \end{itemize}
    \item If $r<R$, then $\Gamma_0^{n+1}(R)\subset\Gamma_0^{n+1}(r)$. 
    \item If $r<R$ and $m>0$, then $\Gamma_m^{n+1}(R)$ and $\Gamma_m^{n+1}(r)$ have a nontrivial symmetric difference. See the beginning of \textsection\ref{conesec} for the technical significance of this observation. 

\end{enumerate}

\subsection{A basic inductive scheme for square function estimates \label{intro1}}

To demonstrate the usefulness of an inductive structure for square function estimates, we first review an earlier counting argument that applies for certain even exponents. Consider $\mc{M}^2(R)$ (the parabola), for which the critical exponent is $p=4$. If $f_\theta:\R^2\to\C$ has Fourier support in $\theta\in \Theta^2(R)$, we have by Plancherel's theorem 
\begin{equation}\label{L4exp} \int_{\R^2}|\sum_{\theta\in\Theta^2(R)}f_\theta|^4=\sum_{i=1}^4\sum_{\theta_i\in\Theta^2(R)}\int_{\R^2}\widehat{f}_{\theta_1}*\widehat{f}_{\theta_2}\overline{\widehat{f}_{\theta_1}*\widehat{f}_{\theta_2}}. \end{equation}
Since $\widehat{f}_{\theta_i}*\widehat{f}_{\theta_j}$ is supported in the sum-set $\theta_i+\theta_j$, bounding  the number of $\theta_l,\theta_k$ for which $(\theta_i+\theta_j)\cap(\theta_l+\theta_k)\not=\emptyset$ leads to a square function estimate. These observations were first made by Fefferman and later recorded by Cordoba \cite{feffL4,cordoba}. In \cite{gpsqfn}, the authors extended the counting approach to moment curves in all dimensions, for even exponents in a certain range. However, for exponents $p$ which are not even integers or which are too large, there is no counting approach available. It may also happen that the sets $\theta_i+\theta_j$ have large overlap, but there is other cancellation in \eqref{L4exp} which still leads to a square function estimate. This is the case for the cone in $\R^3$, which is partitioned into angular sectors $\theta$. The sharp square function estimate for the cone in $\R^3$ has critical exponent $p=4$ \cite{locsmooth}, but it does not follow from a counting argument. Now we describe the structure of an inductive set-up which applies to square function estimates for arbitrary $p$.

For $R>1$, consider an abstract subset $X(R)\subset\R^n$ which is partitioned into $\sqcup_{\theta\in \mc{S}(R)}\theta$. Suppose that if $r\le R$, $\theta\in\mc{S}(R)$, and $\tau\in\mc{S}(r)$, then either $\theta\subset \tau$ or $\theta\cap\tau=\emptyset$. If $g:\R^n\to\C$ is a function with Fourier support in $X(R)$, then for each $r\le R$, we have $g=\sum_{\tau\in\mc{S}(r)}g_\tau$, where $g_\tau=(\widehat{g}\chi_{\tau})^{\widecheck{\,\,\,}}$.

We are interested in bounding the smallest constant $D_{n,p}(R)$ satisfying 
\[ \int_{\R^n}|g|^p\le D_{n,p}(R) \int_{\R^n}(\sum_{\theta\in\mc{S}(R)}|g_\theta|^2)^{\frac{p}{2}} \]
for any $g$ with Fourier transform supported in $X(R)$. In particular, we would like to show that for any $\e>0$, there exists $C_\e>0$ such that $D_{n,p}(R)\le C_\e R^\e$ for all $R$. We will describe some ideas introduced by Guth, Wang, and Zhang to incorporate multi-scale analysis, though the following is not precisely how their argument is presented in \cite{locsmooth}. The idea is to study a two-parameter quantity $D_{n,p}(r,R)$ (for $1<r<R$), which is the smallest constant satisfying 
\[ \int_{\R^n}(\sum_{\tau\in\mc{S}(r)}|g_\tau|^2)^{\frac{p}{2}}\le D_{n,p}(r,R) 
 \int_{\R^n}(\sum_{\theta\in\mc{S}(R)}|g_\theta|^2)^{\frac{p}{2}}\]
for all $g$ with $\supp\widehat{g}\subset X(R)$. Note that $D_{n,p}(1,R)=D_{n,p}(R)$. Induction on scales says that to prove $D_{n,p}(R)\le C_\e R^\e$, we may assume that $D_{n,p}(r)\le C_\e r^\e$ for all $r<R/2$. Working instead with $D_{n,p}(r,R)$, it suffices to prove that $D_{n,p}(r,R)\le C_\e (R/r)^\e$ with the assumption that $D_{n,p}(s,S)\le C_\e (S/s)^\e$ for all $(S/s)<(R/r)/2$, which is a stronger inductive hypothesis.

One may also think of bounding $D_{n,p}(r,R)$ via an iterative argument. Let $K>0$ be a large parameter we will choose later. Let $\supp\widehat{g}\subset X(R)$. There are three key ingredients: 
\begin{enumerate}
    \item (Base case) For each $1\le r\le K^2$, $\int_{\R^n}|\sum_{\tau\in\mc{S}(r)}|g_\tau|^2|^{\frac{p}{2}}\lesssim A_\d K^\d \int_{\R^n}\big(\sum_{\tau'\in\mc{S}(K^2)}|g_{\tau'}|^2\big)^{\frac{p}{2}}$.  
    \item (Progress) If $K^3\le Kr\le R$, then one of the following holds:
    \begin{enumerate}
        \item $\int_{\R^n}(\sum_{\tau\in\mc{S}(r)}|g_\tau|^2)^{\frac{p}{2}}\le B_\d K^\d\int_{\R^n}(\sum_{\tau'\in\mc{S}(Kr)}|g_{\tau'}|^2)^{\frac{p}{2}}$, or 
        \item $\int_{\R^n}(\sum_{\tau\in\mc{S}(r)}|g_\tau|^2)^{\frac{p}{2}}\le B_\d K^\d \sum_{\tau'\in \mc{S}(r/K)}\int_{\R^n}(\sum_{\substack{\tau\in\mc{S}(r)\\ \tau\subset\tau'}}|g_{\tau}|^2)^{\frac{p}{2}}$. 
    \end{enumerate}
    \item (Rescaling) For any $\rho\le r\le R$, 
    \[ \sum_{\tau'\in\mc{S}(\rho)}\int_{\R^n}(\sum_{\substack{\tau\in\mc{S}(r)\\\tau\subset\tau'}}|g_\tau|^2)^{\frac{p}{2}} \le D_{n,p}(r/\rho,R/\rho)\int_{\R^n}(\sum_{\theta\in\mc{S}(R)}|g_\theta|^2)^{\frac{p}{2}}. \] 
\end{enumerate}
We can then show that $D_{n,p}(r,R)\le C_\e R^\e$. Indeed, by using the above ingredients in succession, we have for $K^3\le Kr\le R$ that
\[ D_{n,p}(r,R)\le A_{\d}B_{\d} K^{2\d} \left[ D_{n,p}(rK,R)+D_{n,p}(K^2,RK/r)\right]. \]
Iterating and choosing $\d>0$ and $K>0$ depending on $\e>0$ leads to the upper bound $D_{n,p}(r,R)\le C_\e R^\e$. 

Using lower dimensional results, a base case (ingredient (1)) is available for the cone-like regions $\Gamma_m^{n+1}(R)$, which enables analysis of multi-scale constants as described above. There is no similar base case for $\mc{M}^n(R)$, which is why the proof techniques bounding $\C_m^{n+1}(R)$ and $\M^n(R)$ are organized differently. See the beginning of \textsection\ref{tools} for more about the strategy to bound $\M^n(R)$. 

The main work of establishing Propositions \ref{coneinduct} and \ref{momcurveinduct} lies in justifying the second \emph{progress} ingredient above. We use algorithms adapted to the specific geometries being considered, in Lemmas \ref{multilem1}, \ref{ptwise}, and \ref{multilem3} for $\Gamma_m^n(R)$ and in Proposition \ref{algo} for $\mc{M}^n(R)$. In the following subsection, we describe a simplified example of the algorithm which also demonstrates how $m$th order Taylor cones naturally arise.

\subsection{A simplified explanation of the multi-scale algorithm and $m$th order Taylor cones \label{intro2} }

In addition to induction on the scale $R$, there is another important type of induction we employ: induction on the integers $m$ and $k$ in $\C_m^{k+1}(R)$ and $\M^k(R)$. One type of base case is when $m=k-1$, where we bound $\C_{k-1}^{k+1}(R)$ directly using $L^2$-based analysis. The other base case is the $L^4$ square function estimate for the parabola. Let $\mb{M}^k$ be the statement that $\mb{M}^k(R)\lesssim_\e R^\e$ for all $\e>0$ and $R>1$ (and similarly for $\C_m^{k+1}$). Then the sequence of implications is as follows, with the non-blue terms depending on all previous inputs:
\[ \textcolor{blue}{\M^2}\overset{\textcolor{blue}{\C_1^{2+1}}}{\implies} \C_0^{2+1}\implies {\M^3} \overset{\textcolor{blue}{\C_2^{3+1}}}{\implies} \C_1^{3+1}\implies \C_0^{3+1}\implies \M^4 \overset{\textcolor{blue}{\C_3^{4+1}}}{\implies}\C_2^{4+1}\implies\C_1^{4+1}\implies  \C_0^{4+1}\implies\M^5\overset{\textcolor{blue}{\C_4^{5+1}}}{\implies} \cdots  \]

For concreteness, let us consider how to prove $\M^4$. 
We focus just on the critical exponent $p_4=11$. Given an intermediate expression 
\begin{equation}\label{inteqn} 
\int_{\R^4}(\sum_{\tau\in\Theta^4(r)}|f_\tau|^2)^{\frac{11}{2}} \end{equation}
where $r<R$ and $\supp\widehat{f}\subset \mc{M}^4(R)$, we would like to make progress towards an upper bound of the form 
\begin{equation}\label{goalintro} 
C_\e R^\e\int_{\R^4}(\sum_{\theta\in\Theta^4(R)}|f_\theta|^2)^{\frac{11}{2}} , 
\end{equation}
where \emph{progress} means an inequality of type 2(a) or 2(b) from the key ingredients listed in \textsection\ref{intro1}. 
Consider the Fourier support of $\sum_{\tau\in\Theta^4(r)}|f_\tau|^2$ in \eqref{inteqn}. Each summand $|f_{\tau}|^2$ has Fourier transform $\widehat{f}_\tau*\widehat{\overline{f}}_{\tau}$, which is supported in $\tau-\tau$. If $\tau$ is associated to the $t$-interval $[a,a+r^{-\frac{1}{4}}]$, then $\tau-\tau$ is contained in 
\begin{align}\label{tau-tauintro} \{\sum_{i=1}^4\lambda_i\g_4^{(i)}(a):|\lambda_i|\lesssim r^{-\frac{i}{4}} \}. \end{align}
Note that the exponent $\frac{11}{2}$ lies in the range $2\le \frac{11}{2}\le 7$, which is the range of $p$ for which $\Gamma_0^{3+1}(r)$ satisfies a square function estimate. The subset of $\tau-\tau$ corresponding to $|\lambda_1|\sim r^{-\frac{1}{4}}$ is essentially an element of $\Xi_0^{3+1}(r^{\frac{3}{4}})$ dilated by a factor of $r^{-\frac{1}{4}}$. A high-low frequency decomposition and the triangle inequality bounds \eqref{inteqn} by
\begin{align}\label{hilo} C\int_{\R^4}|\sum_{\tau\in\Theta^4(r)}(\widehat{|f_{\tau}|^2}\eta_r)^{\widecheck{\,\,}}|^{\frac{11}{2}} +C\int_{\R^4}|\sum_{\tau\in\Theta^4(r)}(\widehat{|f_{\tau}|^2}(1-\eta_r))^{\widecheck{\,\,}}|^{\frac{11}{2}},\end{align}  
where $\eta_r$ is identically $1$ in the ball of radius $C_0^{-1}r^{-\frac{1}{4}}$ centered at the origin (we will be vague about the selection of a constant $C_0$ here). Label the case where \eqref{inteqn} is dominated by the first term in \eqref{hilo} as Case $(l_{11/2})$ to reflect that the integrand is dominated by low frequencies and the exponent is $11/2$. Case $(h_{11/2})$ is when \eqref{inteqn} is dominated by the second term in \eqref{hilo}. 

\noindent\underline{Case $(l_{11/2})$} The low-frequency integrand satisfies a favorable pointwise inequality by local $L^2$ orthogonality: 
\[ ||f_\tau|^2*\widecheck{\eta}_r|\lesssim \sum_{\substack{\tau'\in\Theta^4(C_0r)\\\tau'\subset\tau}} |f_{\tau'}|^2*|\widecheck{\eta}_r|\qquad\text{for each $\tau\in\Theta^4(r)$}. \]
Then by Young's convolution inequality, since $\|\widecheck{\eta}_r\|_1\sim 1$, conclude that \eqref{inteqn} is bounded by
\[ C\int_{\R^4}|\sum_{\tau'\in\Theta^4(C_0r)}|f_{\tau'}|^2|^{\frac{11}{2}}. \]
The refinement of $\tau$ into $\tau'$ with controlled implicit constants represents progress of type (2a) towards \eqref{goalintro}, and Case $(l_{11/2})$ concludes. 

\noindent\underline{Case $(h_{11/2})$} Assume that \eqref{inteqn} is bounded by the second, high-frequency term in \eqref{hilo}. Since the exponent $\frac{11}{2}$ is in the range $2\le \frac{11}{2}\le 7$ and the support of each $\widehat{|f_\tau|^2}(1-\eta_r)$ is identified with an element of $\Xi_0^{3+1}(r^{\frac{3}{4}})$, we have 
\[ \int_{\R^4}|\sum_{\tau\in\Theta^4(r)}(\widehat{|f_{\tau}|^2}(1-\eta_r))^{\widecheck{\,\,}}|^{\frac{11}{2}}\lesssim C_0^{3+1}(r^{\frac{3}{4}}) \int_{\R^4}|\sum_{\tau\in\Theta^4(r)}|(\widehat{|f_{\tau}|^2}(1-\eta_r))^{\widecheck{\,\,}}|^2|^{\frac{11}{4}}. \]
We are done using the high-frequency cutoff, and the integral on the right hand side is bounded using Cauchy-Schwarz and Young's inequality by 
\[ \int_{\R^4}(\sum_{\tau\in\Theta^4(r)}|f_{\tau}|^4)^{\frac{11}{4}}. \]
This is similar to the initial expression since the summands $|f_\tau|^4$ have approximately the same Fourier support \eqref{tau-tauintro} as $|f_\tau|^2$ (since $\widehat{|f_\tau|^4}$ may be written as a 4-fold convolution of two copies of $\widehat{f}_\tau$ and two copies of the reflection of $\widehat{f}_\tau$). However, we have made some progress by reducing the exponent from $\frac{11}{2}$ to $\frac{11}{4}$, which lies in the range $2\le \frac{11}{4}\le 4$. The next step is to iterate the previous argument by again performing a high-low decomposition of frequency space. Since we are working in $\R^4$ and in the range $[2,4]$, both the geometries $\Gamma_0^{3+1}$ and $\Gamma_1^{3+1}$ admit square function estimates. We will define the high frequency part of $\supp\widehat{|f_{\tau}|^4}$ so that it may be identified with an element of $\Xi_1^{3+1}(r^{\frac{3}{4}})$. See the comments about key decisions at the end of this subsection for an explanation of why we chose $\Xi_1^{3+1}(r^{\frac{3}{4}})$ instead of $\Xi_0^{3+1}(r^{\frac{3}{4}})$. If the $t$-interval associated to $\tau\in\Theta^4(r)$ has initial point $a$, we divide the Fourier support of $|f_\tau|^4$ into the high part, 
\begin{equation}\label{newhi} \tau_h:=\{\sum_{i=1}^4\lambda_i\g_4^{(i)}(a):|\lambda_i|\lesssim r^{-\frac{i}{n}}\,\forall i,\quad \max(|\lambda_1|C_1^{2}r^{\frac{1}{4}},|\lambda_2|C_1r^{\frac{2}{4}})\ge 1\}, \end{equation}
which is identified with an element of $\Xi_1^{3+1}(r^{\frac{3}{4}})$ dilated by $r^{-\frac{1}{4}}$, and the low part 
\begin{equation}\label{newlo} \tau_l:= \{\sum_{i=1}^4\lambda_i\g_4^{(i)}(a):|\lambda_i|\lesssim r^{-\frac{i}{n}}\,\forall i,\quad \max(|\lambda_1|C_1^2r^{\frac{1}{4}},|\lambda_2|C_1r^{\frac{2}{4}})\le 1\} , \end{equation}
where $C_1$ is a large constant that we do not specify. Let $\eta_{\tau_h}+\eta_{\tau_l}$ be smooth bump functions which are identically one on $\supp\widehat{|f_\tau|^4}$ and supported in $\tau_h$ and $\tau_l$ respectively. Consider the cases where the $\tau_h$ and $\tau_l$ frequencies dominate separately.


\noindent\underline{Case $(h_{11/2},l_{11/4})$} In this case, we have the inequality 
\[ \int_{\R^4}(\sum_{\tau\in\Theta^4(r)}|f_\tau|^2)^{\frac{11}{2}}   \le C \int_{\R^4}|\sum_{\tau\in\Theta^4(r)}|f_\tau|^4*\widecheck{\eta}_{\tau_l}|^{\frac{11}{4}}   . \]
Consider each summand one at a time. By Fubini's theorem, we have
\[  |f_{\tau}|^4*|\widecheck{\eta}_{\tau_l}|(x,y)= \int_{\R^2}\int_{\R^2}|f_{\tau}(x',y')|^4|\widecheck{\eta}_{\tau_l}|(x-x',y-y')dx' dy'\]
where $x,y\in\R^2$. 
We will consider upper bounds for the inner integral above. Because the Fourier transform on $\R^4$ factors into the composition of 2-dimensional Fourier transforms, we may regard $f_{\tau}$ as a function in the first two variables with Fourier support in $\mc{M}^2(R^{\frac{1}{2}})$. The weight function $|\widecheck{\eta}_{\tau_l}|$ is localized in the first two variables to a rectangle of dimensions $C_1^2r^{\frac{1}{4}}\times C_1r^{\frac{2}{4}}$. Invoking a local $L^4$ square function estimate for $\mc{M}^2(R^{\frac{1}{2}})$ then yields the pointwise inequality
\[ |f_{\tau}|^4*|\widecheck{\eta}_{\tau_l}|\lesssim \M^2(C_1) |\sum_{\substack{\tau'\in\Theta^4(C_1r)\\\tau'\subset\tau}}|f_{\tau'}|^2|^2*|\widecheck{\eta}_{\tau_l}|. \]
In summary, we have an upper bound for \eqref{inteqn} of the form 
\[ C_0 \C_0^{3+1}(r^{\frac{3}{4}})\M^2(C_1) \int_{\R^4}(\sum_{\tau\in\Theta^4(r)}|\sum_{\substack{\tau'\in\Theta^4(C_1r)\\\tau'\subset\tau}}|f_{\tau'}|^2|^2*|\widecheck{\eta}_{\tau_l}|)^{\frac{11}{4}}. \]
We remark that $C_1$ should be chosen to be significantly larger than $C_0$ in order to make progress of type 2(a). Also, since $\widecheck{\eta}_{\tau_l}$ depends on each $\tau$, we cannot simply invoke Young's convolution inequality to eliminate the auxiliary convolutions. This is a technical point, but by carefully defining the functions $\eta_{\tau_l}$, we show that it is safe to substitute the integral displayed above by 
\[ \int_{\R^4}(\sum_{\tau\in\Theta^4(r)}|\sum_{\substack{\tau'\in\Theta^4(C_1r)\\\tau'\subset\tau}}|f_{\tau'}|^2|^2)^{\frac{11}{4}}\le \int_{\R^4}( \sum_{\substack{\tau'\in\Theta^4(C_1r)}}|f_{\tau'}|^2)^{\frac{11}{2}}, \]
which completes the \emph{progress} step for this line of reasoning. 

\noindent\underline{Case $(h_{11/2},h_{11/4})$} The high-frequency part that we have not yet addressed is the case that 
\[  \eqref{inteqn} \lesssim \C_0^{3+1}(r^{\frac{3}{4}})\int_{\R^4}|\sum_{\tau\in\Theta^4(r)}|f_\tau|^4*\widecheck{\eta}_{\tau_h}|^{\frac{11}{4}}.  \]
Recall that each $\tau_h$ is identified with an (almost unique, depending on $C_1$) element of $\Xi_1^{3+1}(r^{\frac{3}{4}})$ dilated by a fixed factor, and so the right hand side above is 
\[ \lesssim \C_0^{3+1}(r^{\frac{3}{4}})\C_1^{3+1}(r^{\frac{3}{4}})\int_{\R^4}(\sum_{\tau\in\Theta^4(r)} ||f_\tau|^4*\widecheck{\eta}_{\tau_h}|^2)^{\frac{11}{8}} . \]
Although we can no longer cite standard inequalities like Cauchy-Schwarz or Young's inequality, we show that the integral above can be bounded by 
\begin{equation}\label{l8} \int_{\R^4}(\sum_{\tau\in\Theta^4(r)} |f_\tau|^8)^{\frac{11}{8}} . \end{equation}
At this stage, we see something different from the previous cases. One difference is that the exponent lies in the range $1\le \frac{11}{8}\le 2$, so we should consider using $\C_2^{3+1}(r^{\frac{3}{4}})$. Since $8$ is an even integer, the Fourier support of $|f_\tau|^8$ is roughly the same as the Fourier support of $|f_\tau|^2$, so we could perform one more high-low frequency decomposition 
\[ \sum_{\tau\in\Theta^4(r)}|f_\tau|^8*\widecheck{\eta}_{\tilde{\tau}_l}+\sum_{\tau\in\Theta^4(r)}|f_\tau|^8*\widecheck{\eta}_{\tilde{\tau}_h}. \]
The difficulty now is dealing with the low part. By regarding the convolution with $\widecheck{\eta}_{\tilde{\tau}_l}$ as an $L^8$ integral against a weight, we would like to invoke previously proven square function estimates for lower dimensional moment curves. The problem is that even for one lower dimension, $\mc{M}^3$ only admits a sharp square function estimate in the range $2\le p\le 7$, not $2\le p\le 8$. To address this issue, we use the fact that $\|\cdot\|_{\ell^p}\le \|\cdot\|_{\ell^q}$ whenever $p\ge q$ to bound \eqref{l8} by 
\[ \int_{\R^4}(\sum_{\tau\in\Theta^4(r)} |f_\tau|^6)^{\frac{11}{6}}. \]
If we had chosen to replace the $\ell^8$ expression by an $\ell^7$ expression, then the exponent for the low analysis is no longer an issue, but there is no simple formula for the Fourier transform of $|f_{\tau}|^7$ as there is for $|f_{\tau}|^6$. Furthermore, we remark that the exponent $\frac{11}{6}$ remains in the same range $1\le \frac{11}{8}\le \frac{11}{6}\le 2$ as the previous exponent $\frac{11}{8}$, so we should still identify high-frequency pieces which may be bounded using $\C_2^{3+1}(r^{\frac{3}{4}})$. We verify that similar arguments involving the exponents $p$ may be made in the general setting in Lemma \ref{pprops}. Define the high and low subsets of $\supp\widehat{|f_\tau|^6}$ as 
\[ \tilde{\tau}_h:=\{\sum_{i=1}^4\lambda_i\g_4^{(i)}(a):|\lambda_i|\lesssim r^{-\frac{i}{n}} \,\,\forall i,\quad \max(|\lambda_1|C_2^3r^{\frac{1}{4}},|\lambda_2|C_2^2r^{\frac{2}{4}},|\lambda_3|C_2r^{\frac{3}{4}}) \ge 1 \} \]
and 
\[ \tilde{\tau}_l:= \{\sum_{i=1}^4\lambda_i\g_4^{(i)}(a):|\lambda_i|\lesssim r^{-\frac{i}{n}} \,\,\forall i,\quad \max(|\lambda_1|C_2^3r^{\frac{1}{4}},|\lambda_2|C_2^2r^{\frac{2}{4}},|\lambda_3|C_2r^{\frac{3}{4}}) \le 1 \},\]
and let $\eta_{\tilde{\tau}_h}$ and $\eta_{\tilde{\tau}_l}$ be adapted to their respective sets. We consider this high-low frequency decomposition in the following final two cases.

\noindent\underline{Case $(h_{11/2},h_{11/4},l_{11/8})$} In the case that the $\tilde{\tau}_l$ frequencies dominate, we have the pointwise inequality 
\[ |\sum_{\tau\in\Theta^4(r)}|f_\tau|^6*\widecheck{\eta}_{\tilde{\tau}_l}|\lesssim \M^3(C_2)|\sum_{\tau\in\Theta^4(r)}|\sum_{\substack{\tau'\subset\tau \\\tau'\in\Theta^4(C_2r)}}|f_{\tau'}|^2|^3*|\widecheck{\eta}_{\tilde{\tau}_l}| . \]
A few more technical steps leads to the final inequality bounding \eqref{inteqn} by 
\[  \lesssim \C_0^{3+1}(r^{\frac{3}{4}})\C_1^{3+1}(r^{\frac{3}{4}})\M_3(C_2) \int_{\R^4}(\sum_{\tau'\in \Theta^4(C_2r)}|f_{\tau'}|^2)^{\frac{11}{2}}, \]
which is progress of type 2(a) towards \eqref{goalintro}. 

\noindent\underline{Case $(h_{11/2},h_{11/4},h_{11/8})$} Now we consider when the $\tilde{\tau}_h$ frequencies dominate. We have the inequality
\[ \eqref{inteqn}\lesssim \C_0^{3+1}(r^{\frac{3}{4}})\C_1^{3+1}(r^{\frac{3}{4}})\C_2^{3+1}(r^{\frac{3}{4}})\sum_{\tau\in\Theta^4(r)}\int_{\R^n}|f_\tau|^{11}. \]
This is the one case where we see progress of type 2(b), and the algorithm is concluded.

\vspace{1mm}
\noindent\underline{{\bf{Comments about key decisions}}} In each of the cases in our algorithm, we identify high and low subsets of the Fourier support of $\sum_\tau|f_\tau|^{2k}$ for some $k$, which is approximately contained in $\cup_\tau (\tau-\tau)$. The initial division is the simplest: 
\[ \tag{$\star$} \begin{cases}
    (\cup_\tau(\tau-\tau))_{\text{low}}:=(\cup_\tau(\tau-\tau))\cap\{|\xi|<C_0^{-1}r^{-\frac{1}{4}}\} \\
    (\cup_\tau(\tau-\tau))_{\text{high}}:=(\cup_\tau(\tau-\tau))\cap\{|\xi|\ge C_0^{-1}r^{-\frac{1}{4}}\}
\end{cases}. \]
While it would be convenient to use this definition in all of the cases, our criterion for a good high/low decomposition is changing at each stage, depending on the integer $k$ in the expression $(\sum_\tau|f_\tau|^{2k})^{\frac{11}{2k}}$. We have two requirements: 
\begin{enumerate}
    \item there is an $L^{\frac{11}{2k}}$ square function estimate for $(\sum_{\tau}|f_\tau|^{2k})_{\text{high}}$, and
    \item there is a pointwise $L^{2k}$ estimate for each $(|f_\tau|^{2k})_{\text{low}}$.
\end{enumerate}  
It is easier to satisfy (1) when the high set is smaller (for example, the high set from ($\star$)) and it is easier to satisfy (2) when the low set is smaller. Identifying copies of $\Gamma_1^{3+1}(r)$ and $\Gamma_2^{3+1}(r)$ within $\cup_\tau(\tau-\tau)$ as the high sets in the cases $k=2$ and $k=3$ respectively allows for both (1) and (2) to be satisfied. 

The set $\Gamma_1^{2+1}(r)$ was also studied in \cite{locsmooth}, but only an $L^2(R^3)$ estimate was required in that set-up. It is fairly easy to get a similar $L^2(\R^4)$ estimate for the set $\Gamma_1^{3+1}(r)$ by finite overlapping on the Fourier side, but this is no longer enough. At some point, we realized that $\Gamma_1^{3+1}(r)$ is closely related to $\Gamma_0^{2+1}(r)$ (the cone in $\R^3$) and there is an $L^p(\R^4)$ square function estimate for $2\le p\le 4$ that is similar to the one proven in \cite{locsmooth}. This allows us to use an $L^{\frac{11}{4}}(\R^4)$ square function estimate to bound the high part of $\sum_\tau|f_\tau|^4$, despite the rather large high-frequency set identified with $\Gamma_1^{3+1}(r)$.

\section{Preliminaries for the proof of Proposition \ref{coneinduct} \label{prelimsec}}

In this section, we prove preliminary technical results which are inputs to the proof of Proposition \ref{coneinduct}. In \textsection\ref{smallgen}, we describe some flexibility in the definition of the sets in $\Xi_m^{n+1}(R)$ which justifies later appeals to $\C_m^{n+1}(R)$ to bound functions which are only approximately supported in elements of $\Xi_m^{n+1}(R)$. In \textsection\ref{L2sec}, we directly prove key $L^2$ based inequalities which directly bound the constants $\C_{n-1}^{n+1}(R)$. We also introduce weight functions which will be used extensively to write local versions of inequalities. 

\subsection{Small generalizations of $\C_m^{n+1}(R)$ \label{smallgen}}  

We show that the definition we used for the elements of $\Xi_m^{n+1}(R)$ is not sensitive to changes by constant factors.

\begin{proposition}[A general $\C_m^{n+1}(R)$] \label{gencone} Let $n\ge 2$, $0\le m\le n-1$, and $L\ge 2$. Suppose that $\C_m^{n+1}(R)\lesssim_{n,\e}R^\e$. There is a constant $c_n>0$ so that the following holds. If for $l=0,\ldots,R^{1/n}-1$, $g_{R,L,l}:\R^n\to\C$ are Schwartz functions with  
\begin{align}\nonumber
\supp\widehat{g_{R,L,l}}\subset \{\sum_{i=0}^n\lambda_i&\phi_n^{(i)}(t):t\in[\frac{l}{100L2^nR^{\frac{1}{n}}},\frac{l+1}{100L2^nR^{\frac{1}{n}}}],\quad |\lambda_i|\le R^{-\frac{i}{n}}\,\forall \,i,\quad\max_{0\le j\le m}|\lambda_j|2^{m+1-j}R^{\frac{j}{n}}\ge\frac{1}{L}\}, 
\end{align} 
then 
\[ \int_{\R^{n+1}}|\sum_{l=0}^{R^{1/n}-1} g_{R,L,l}|^{p}\lesssim_{n,\e} L^{c_n} R^\e   \int_{\R^{n+1}}|\sum_{l=0}^{R^{1/n}-1}|g_{R,L,l}|^2|^{\frac{p} {2}}  \]
for any $2\le p\le p_{n-m}$. 
\end{proposition}

\begin{proof} 
It suffices to assume that $R\ge L^C$ for a constant $C$ that we may choose. By Taylor expansion, the Fourier support of $g_{R,L,l}$ is contained in 
\[ \{\sum_{i=0}^n\lambda_i\phi_n^{(i)}((2^{100n}L)^{-1}lR^{-\frac{1}{n}}):|\lambda_i|\le (1+\frac{1}{100L})R^{-\frac{i}{n}},\quad\max_{0\le i\le m}|\lambda_i|2^{m+1-i}R^{\frac{i}{n}}\ge \frac{99}{100L}\}. \]
Let $s:\R\to[0,1]$ be a smooth bump function satisfying $s(t)=1$ when $|t|\le \frac{3}{2}$ and $\supp s\subset(-2,2)$. The function $\s(x_0,\ldots,x_n):=s(x_0)\cdots s(R^{\frac{m}{n}}x_m) s(x_{m+1})\cdots s(x_n)$ is identically equal to $1$ on the box $[-1,1]\times[-R^{-\frac{1}{n}},R^{-\frac{1}{n}}]\times\cdots\times[-R^{-\frac{m}{n}},R^{-\frac{m}{n}}]\times[-1,1]^{n-m}$. Let $A_l:\R^{n+1}\to\R^{n+1}$ be a linear map which, in $x_0,\ldots,x_m$, takes $\phi_{m}^{(i)}(l(2^{100n}L)^{-1}R^{-\frac{1}{n}})$ to the standard basis vector ${\bf{e}}_{i+1}$ (with a $1$ in the $(i+1)$st coordinate) for $i=0,\ldots,m$, and which is the identity in the remaining variables $x_{m+1},\ldots,x_n$. The function $\s\circ A_l$ is supported in 
\begin{equation*}
\{\sum_{i=0}^m\mu_i\phi_m^{(i)}((2^{100n}L)^{-1}lR^{-\frac{1}{n}}): |\mu_i|\le 2R^{-\frac{i}{n}}\}\times[-2,2]^{n-m} \end{equation*}
and identically one in 
\[ \{\sum_{i=0}^m\mu_i\phi_m^{(i)}((2^{100n}L)^{-1}lR^{-\frac{1}{n}}): |\mu_i|\le (3/2)R^{-\frac{i}{n}}\}\times[-2,2]^{n-m}. \]
For $j\in\N$, 
define $\eta_{l,j}(x_0,\ldots,x_n)$ to be the difference
\begin{align}
\label{auxbump} \eta_{l,j}(x_0,\ldots,x_n):=&\s\circ A_l\Big(\Big(\frac{3}{4}\Big)^j2^{j(m+1)}x_0,\ldots,\Big(\frac{3}{4}\Big)^j2^jx_m,x_{m+1},\ldots,x_n\Big)\\
&\qquad -\s\circ A_l\Big(\Big(\frac{3}{4}\Big)^{j+1}2^{(j+1)(m+1)}x_0,\ldots,\Big(\frac{3}{4}\Big)^{j+1}2^{j+1}x_m,x_{m+1},\ldots,x_n\Big). \nonumber
\end{align}
Note that $\eta_{l,j}$ is supported in 
\begin{equation}\label{ksupp2} \{\sum_{i=0}^m\mu_i\phi_m^{(i)}((2^{100n}L)^{-1}lR^{-\frac{1}{n}}): \max_{0\le i\le m}|\mu_i|[2(4/3)^j2^{-j(m+1)}]^{-1}2^{m-j}(R/2^j)^{\frac{i}{n}}\in[\frac{1}{2},1] \}\times[-2,2]^{n-m}. \end{equation}
For the support of $\eta_{l,j}$ to have nonempty intersection with the support of $\widehat{g_{R,L,l}}$, we have $3^j2^{jm}\lesssim L$. Since there are fewer than $C\log L$ many $j$, by dyadic pigeonholing, there is a $j\lesssim\log L$ which satisfies 
\begin{equation}\label{RHS}  
\int_{\R^{n+1}}|\sum_l g_{R,L,l}|^{p }\le (\log L)^{O(1)} \int_{\R^{n+1}}|\sum_{l} g_{R,L,l}*\widecheck{\eta_{l,j}}|^{p }. \end{equation}
Each $\widehat{g_{R,L,l}}\eta_{l,j}$ is supported in the $2\cdot\big(\frac{4}{3}\big)^j2^{-j(m+1)}$-dilation of an element of $\Xi_m^{n+1}(2^{-j}R)$. Using the hypothesis that $\C_m^{n+1}(R)\lesssim_\e R^\e$ and Cauchy-Schwarz, it follows that 
\[ \int_{\R^{n+1}}|\sum_{l} g_{R,L,l}*\widecheck{\eta_{l,j}}|^{p }\lesssim_\e R^\e L^{O(1)} \int_{\R^{n+1}}|\sum_{l} |g_{R,L,l}*\widecheck{\eta_{j,l}}|^2|^{\frac{p }{2}}. \]
It remains to prove 
\begin{lemma}\label{remlem}
    \[\int_{\R^{n+1}}|\sum_{l} |g_{R,L,l}*\widecheck{\eta}_{l,j}|^2|^{\frac{p}{2}}\lesssim_\e R^\e \int_{\R^{n+1}}|\sum_{l} |g_{R,L,l}|^2|^{\frac{p}{2}}. \]
\end{lemma}
We delay this proof until \textsection\ref{truncsec}.

\end{proof}

\begin{lemma}[Localized $\C_m^{n+1}(R)$] \label{loclem} Let $\s:\R^{n+1}\to\C$ be a smooth function with Fourier transform supported in a ball of radius $R^{-1}$ centered at the origin. If $\C_m^{n+1}(R)\lesssim_\e R^\e$, then
\[ \int_{\R^{n+1}}|\sum_{\theta\in\Xi_m^{n+1}(R)}f_\theta \s|^{p}\lesssim_\e R^\e \int_{\R^{n+1}} |\sum_{\theta\in\Xi_m^{n+1}(R)}|f_\theta \s|^2|^{\frac{p}{2}} \]
for any $2\le p\le p_{n-m}$ and any Schwartz functions $f_\theta$ which have Fourier support in $\theta$. 
\end{lemma}

\begin{proof} Each summand $f_\theta \s$ has Fourier transform supported in 
\[ \{\sum_{i=0}^n\lambda_i \phi_n^{(i)}(l2^{-100n}R^{-\frac{1}{n}}):|\lambda_i|\le 2 R^{-\frac{i}{n}},\quad\max_{0\le j\le m}|\lambda_j|2^{m-j}R^{\frac{j}{n}}\ge \frac{1}{4}\}. \]
Therefore, Proposition \ref{gencone} applies.     
\end{proof}

\subsection{Weight functions and the locally constant property \label{wtsec}}

We introduce weight functions which will be used to localize the integrals. Let $\s_n:\R^{n}\to[0,\infty)$ be a radial, smooth bump function supported in $|\xi|\le 1$ and satisfying $|\widecheck{\s}_n(x)|>0$ for all $|x|\le 1$. For any $d\in\N^+$, define $W^{n,d}:\R^{n}\to[0,\infty)$ by 
\[ W^{n,d}(x)=\sum_{j=0}^\infty \frac{1}{2^{100n^2jd}}|\widecheck{\s}_n|^2(2^{-j}x). \]

Let $B_0$ be the unit ball centered at the origin in $\R^n$. For any set $U=T(B_0)$ where $T$ is an affine transformation $T:\R^n\to\R^n$, define
\[ W_{U}^{n,d}(x)=W^{n,d}(T^{-1}(x)). \]
Note that specifying either the spatial localization $U$ or its Fourier dual $U^*:=\{\xi\in\R^n:|\langle\xi,x\rangle|\le \frac{1}{2}\quad\forall x\in  U-U\}$ (assuming $W_U^{n,d}$ is centered at the origin) suffices to define the weight function $W_U^{n,d}$. We write $\tilde{W}^{n,d}_U$ to mean the $L^1$-normalized version of $W_U^{n,d}$, so $\tilde{W}_U^{n,d}(x):=\|W_U^{n,d}\|_1^{-1}W_U^{n,d}$. 

If $U\subset\R^n$ is a convex set, then we write $W_U^{n,d}$ to mean $W_{\tilde{U}}^{n,d}$, where $\tilde{U}=T(B_0)$ for some affine transformation $T$ and $\tilde{U}$ is comparable to $U$. This defines $W_U^{n,d}$ up to a bounded constant, which is sufficient for our arguments. 

The following are some important properties of $W^{n,d}$ that we will repeatedly use.
\begin{enumerate}
    \item The function $W^{n,d}(x)$ is $\sim 1$ if $|x|\le 1$ and $\sim |x|^{-100n^2d}$ if $|x|>1$. 
    \item The Fourier transform $\widehat{W^{n,d}}$ is supported in $|\xi|\le 2$. 
    \item If $A_1,A_2$ are affine transformations of the unit ball and $A_1\subset A_2$, then  the following hold:
    \begin{enumerate}
        \item $\tilde{W}_{A_1}^{n,d}*\tilde{W}_{A_1}^{n,d}\lesssim_{n,d} \tilde{W}_{A_1}^{n,d}$, 
        \item ${W}_{A_1}^{n,d}\le {W}_{A_2}^{n,d}$,
        \item $\tilde{W}_{A_1}^{n,d_1}*\tilde{W}_{A_2}^{n,d_2}\lesssim_{n,d_1,d_2} \tilde{W}_{A_2}^{n,\min(d_1,d_2-(100n)^{-1})}$.
    \end{enumerate} 
\end{enumerate}

Let $f:\R^n\to\C$ be Schwartz and suppose that $\widehat{f}$ is supported in $T(B_0)$ for some affine transformation $T$. By \emph{locally constant property}, we mean the pointwise inequality $|f(x)|\lesssim_{n,d}|f|*\tilde{W}_{(T(B_0))^*}^{n,d}(x)$ for each $d\ge 1$. This follows from writing $\widehat{f}=\widehat{f}(\s\circ T^{-1})$ where $\s$ is a fixed, appropriately chosen smooth bump function, and then using $|\widecheck{\s\circ T^{-1}}|\lesssim_{n,d}\tilde{W}_{(T(B_0))^*}^{n,d}$.

\subsubsection{A key local $L^2$ estimate \label{L2sec}}
Let $1\le r\le R$ with $r,R\in2^{n\N}$. For each $\theta\in\Xi_{n-1}^{n+1}(R)$, let $I(\theta)$ be the $t$-interval of length $2^{-100n}R^{-\frac{1}{n}}$ associated to $\theta$. We consider estimates of the form 
\[ \int_{\R^{n+1}}|\sum_{\theta\in\Xi_n^{n+1}(R)}f_\theta|^2W^{n+1,d}_{B_{\tilde{r}}} \lesssim_\e r^\e\int_{\R^{n+1}} \sum_{J\in\mc{J}(r^{-\frac{1}{n}})}|\sum_{\substack{I(\theta)\subset J\\ \theta\in\Xi_{n-1}^{n+1}(R)}}f_\theta|^2 W_{B_{\tilde{r}}}^{n+1,d},\]
for an appropriate radius $\tilde{r}$. We may use different radii $\tilde{r}$ depending on which part of the Fourier support within $\Gamma_{n-1}^{n+1}(R)$ is dominant. Consider the following subsets of $\Gamma_{n-1}^{n+1}(R)$: for a dyadic parameter $\lambda$, $R^{-\frac{n-1}{n}}\le\lambda\le 1$, let $\Xi_{n-1}^{n+1}(R,\lambda)$ be the collection of 
\[ \theta=\{\sum_{i=0}^n\lambda_i\phi_n^{(i)}(a):|\lambda_i|\le R^{-\frac{i}{n}}\quad\forall i,\quad\max_{0\le i\le n-1}|\lambda_i|\ge\lambda/2,\quad\max_{0\le i\le n-1}|\lambda_i|2^{n-i}R^{\frac{i}{n}}\ge 1 \}, \]
where $a\in 2^{-100n}R^{-\frac{1}{n}}\Z\cap[0,1]$ is the initial point of $I(\theta)$.

\begin{thm}[Local $L^2$ orthogonality] \label{locL2thm} For any $d\ge 1$ and ball $B\subset\R^{n+1}$of radius $\lambda^{-1}r$, we have 
\[ \int_{\R^{n+1}}|\sum_{\theta\in\Xi_{n-1}^{n+1}(R,\lambda)}f_\theta |^2W_B^{n+1,d}\lesssim  \int_{\R^{n+1}}\sum_{J\in\mc{J}(r^{-\frac{1}{n}})}|\sum_{\substack{I(\theta)\subset J\\\theta\in\Xi_{n-1}^{n+1}(R,\lambda)}}f_\theta |^2W_B^{n+1,d}.  \] 
\end{thm}

To prove Theorem \ref{locL2thm}, it suffices to consider functions with Fourier support contained in further specialized subsets of the elements of $\Xi_{n-1}^{n+1}(R,\lambda)$. By the triangle inequality, it suffices to work with $\theta\in\Xi_{n-1}^{n+1}(R,\lambda)$ with $I(\theta)$ all contained in a single $L^{-1}$-intervals (where $L$ is a large constant we are free to choose). For $0\not=(k_0,\ldots,k_{n-1})\in\{-1,0,1\}^n$, let $\Xi_{n-1}^{n+1}(R,\lambda,{\bf{k}})$ be the collection of 
\[ \theta=\{\sum_{i=0}^n\lambda_i\phi_n^{(i)}(a):|\lambda_i|\le \min(\lambda,R^{-\frac{i}{n}}),\quad k_i\lambda_i\ge |k_i|\lambda/2,\quad\max_{0\le i\le n-1}|\lambda_i|2^{n-i}R^{\frac{i}{n}}\ge 1 \}, \]
where $a\in 2^{-100n}R^{-\frac{1}{n}}\Z\cap[0,1]$ is the initial point of $I(\theta)$. There are bump functions $\s_{{\bf{k}}}$ with $\supp\s_{\bf{k}}\cap\theta$ equal to the displayed set for each $\theta$, and we may assume that for some $0\not={\bf{k}}$,
\[ \int_{\R^{n+1}}|\sum_{\theta\in\Xi_{n-1}^{n+1}(R,\lambda)}f_\theta |^2W_B^{n+1,d} \lesssim \int_{\R^{n+1}}|\sum_{\theta\in\Xi_{n-1}^{n+1}(R,\lambda)}f_\theta*\widecheck{\s}_{{\bf{k}}}  |^2W_B^{n+1,d} . \]
Suppose that the right hand side is bounded by 
\[ C\int_{\R^{n+1}}|\sum_{J\in\mc{J}(r^{-\frac{1}{n}})}|\sum_{\substack{I(\theta)\subset J\\ \theta\in\Xi_{n-1}^{n+1}(R,\lambda)}}f_\theta*\widecheck{\s}_{{\bf{k}}}  |^2W_B^{n+1,d}.  \]
Since the $\s_{{\bf{k}}}$ satisfies $|\widecheck{\s}_{{\bf{k}}}|\lesssim \tilde{W}_{\tilde{B}}^{n+1,d}$, where $\tilde{B}\subset\R^{n+1}$ is the ball centered at the origin of radius $\lambda^{-1}$. Then we have
\begin{align*} \int_{\R^{n+1}}|\sum_{J}|\sum_{\substack{I(\theta)\subset J\\ \theta\in\Xi_{n-1}^{n+1}(R,\lambda)}}f_\theta*\widecheck{\s}_{{\bf{k}}}  |^2W_B^{n+1,d}&\lesssim \int_{\R^{n+1}}|\sum_{J}|\sum_{\substack{I(\theta)\subset J\\ \theta\in\Xi_{n-1}^{n+1}(R,\lambda)}}f_\theta  |^2(\tilde{W}_{\tilde{B}}^{n+1,d} *W_B^{n+1,d})\\
    &\sim \int_{\R^{n+1}}|\sum_{J}|\sum_{\substack{I(\theta)\subset J\\ \theta\in\Xi_{n-1}^{n+1}(R,\lambda)}}f_\theta  |^2 W_B^{n+1,d}. 
\end{align*}

We use the following technical lemma in the proof of Theorem \ref{locL2thm}. 

\begin{lemma}\label{L2techlem} Let $C,c>0$. There is a constant $L=L(n,C,c)>0$ so that the following holds. Let $\lambda$ be a dyadic value in the range $R^{-\frac{n-1}{n}}\le \lambda\le 1$. Let $k_0\in\{0,\ldots,n-1\}$ and let $0\not={\bf{k}}\in\{-1,0,1\}^n$. Suppose that $|T|\le \frac{1}{L_n}$ and
\begin{equation}\label{assump2} 
|H_{k,k}^1-\sum_{i=0}^{k}H_{k,i}^2\frac{1}{(k-i)!}T^{k-i}|\le C \lambda r^{-1} \quad \text{for each $k=1,\ldots,n$} \end{equation}
where $|H_{k,i}^j|\le C\min(R^{-\frac{i}{n}},\lambda)$ for all $k,i,j$, $|H_{k,k_0}^j|\ge cR^{-\frac{k_0}{n}}$ for all $k,j$, and $H_{k,i}^jk_i\ge c|k_i|\lambda$ for each $i=0,\ldots,n-1$.
Then 
\[ |T|\lesssim_{n,C,c} r^{-\frac{1}{n}}. \]
\end{lemma}
\begin{proof}
We induct on $n$. If $n=1$, then $k_0=0$, $\lambda\sim 1$, and \eqref{assump2} implies that $|H_{1,0}^2 T|\le C\lambda r^{-1}+|H_{1,1}^1|+|H_{1,1}^2|$. By assumption, we have $|H_{1,0}^2|\gtrsim 1$ and $|H_{1,i}^j|\lesssim R^{-1}$.  It follows that $|T|\lesssim r^{-1}$.

Now assume that $n>1$ and that \eqref{assump2} holds. If $k_0<n-1$, then the result follows from the $(n-1)$-dimensional case. Therefore, we assume that $k_0=n-1$.

Let $l_0\in\{0,\ldots,n-1\}$ be the the smallest index with $k_{l_0}\not=0$. If $l_0=0$, then the lemma is proved using the $n=1$ case. Suppose that $l_0>0$. Rewrite the hypotheses \eqref{assump2} for $k=l_0+1,\ldots,n$ as 
\begin{equation}\label{revhyp} |H_{l_0+l,l_0+l}^1-\frac{1}{l!}\tilde{H}_{l_0+l,l_0}^2 T^l+\sum_{i=1}^{l}H_{l_0+l,l_0+i}^2\frac{1}{(l-i)!}T^{l-i}|\le C\lambda r^{-1}\qquad\text{for all $l=1,\ldots,n-l_0$},\end{equation} 
where $\tilde{H}_{l_0+l,l_0}^2=\sum_{i=0}^{l_0}H_{l_0+l,i}^2\frac{l!}{(l_0+l-i)!}T^{l_0-i}$. Note that since $|T|\le \frac{1}{L}$, $|H_{l_0+l,j}^2|\le \lambda$ for all $j$, and $|H_{l_0+l,l_0}^2|\ge c\lambda$, we have $|\tilde{H}_{l_0+l,l_0}^2|\sim \lambda$ if $L$ is sufficiently large. Since $l_0>0$, we may invoke the $n-l_0$ case to conclude that for $L$ sufficiently large, \eqref{revhyp} implies that
\[ |T|\lesssim r^{-\frac{1}{n-l_0}}\lesssim r^{-\frac{1}{n}}, \]
as desired.

\end{proof}

\begin{proof}[Proof of Theorem \ref{locL2thm}] As discussed after the statement of Theorem \ref{locL2thm}, it suffices to consider functions $\sum_{\theta\in\Xi^{n+1}_{n-1}(R,\lambda,{\bf{k}})}f_\theta$. 
For each $J\in\mc{J}(r^{-\frac{1}{n}})$, the Fourier support of $\sum_{I(\theta)\subset J}f_\theta$ is contained in 
\begin{align*} 
\tau_J=\{\sum_{i=0}^n\lambda_i\phi_n^{(i)}(a+t)&:|\lambda_j|\le \min(\lambda,R^{-\frac{j}{n}})\quad\forall 0\le j\le n, \\
    &\qquad\qquad\qquad\qquad|t|\le 2^{-100n}r^{-\frac{1}{n}},\quad k_i\lambda_i\ge |k_i|\lambda/2\quad\forall 0\le i\le n-1\}.  \end{align*}
where $a$ is the initial point of $J$. 
By Plancherel's theorem,  
\[\int_{\R^{n+1}}|\sum_{\theta\in\Xi_n^{n+1}(R,\lambda,{\bf{k}})}f_\theta|^2W^{n+1,d}_B=\int_{\R^{n+1}}\big(\sum_{\theta,\theta'\in\Xi_n^{n+1}(R,\lambda,{\bf{k}})}\widehat{f_\theta}*\widehat{\overline{f}}_\theta\big)\widecheck{W}^{n+1,d}_B.  \]
By the triangle inequality (absorbing some factors of $L$), it suffices to count the overlap of the $\lambda r^{-1}$ neighborhood of the $\theta$ and $\theta'$, where $I(\theta)$ and $I(\theta')$ are within the same $L^{-1}$-interval. Suppose that $|t_1-t_2|\le L^{-1}$, and for each $j=1,2$, $|\mu_i^j|\le \min(\lambda,R^{-\frac{i}{n}})$, $k_i\mu_{i}^j|\gtrsim |k_i|\lambda$, $\max_{0\le i\le n-1}|\mu_i^j|R^{\frac{i}{n}}\gtrsim 1$, and 
\[ |\sum_{i=0}^n\mu_i^1\phi_n^{(i)}(t_1)-\sum_{i=0}^n\mu_i^2\phi_n^{(i)}(t_2)|\le \lambda r^{-1}.    \]
By Taylor expansion, this is equivalent to 
\[ |\sum_{i=0}^n\big(\mu_i^1-\sum_{j=0}^{i}\frac{1}{(i-j)!}\mu_j^1(t_2-t_1)^{i-j}\big)\phi_n^{(i)}(t_1)|\le \lambda r^{-1}.    \]
Since $\det[\phi_n^{(0)}(t_1)\cdots\phi_n^{(n)}(t_1)]\sim_n1$ uniformly in $t_1\in[0,1]$, by Cramer's rule, the previous displayed inequality implies that
\[ |\mu_i^1-\sum_{j=0}^{i}\frac{1}{(i-j)!}\mu_j^1(t_2-t_1)^{i-j}|\lesssim \lambda r^{-1}\qquad\text{for each }i=1,\ldots,n. \]
It follows from Lemma \ref{L2techlem} that $|t_1-t_2|\lesssim r^{-\frac{1}{n}}$. Then 
\[ \int_{\R^{n+1}}|\sum_{\theta\in\Xi_n^{n+1}(R,\lambda,{\bf{k}})}f_\theta |^2W^{n+1,d}_B\lesssim_{L} \int_{\R^{n+1}}\sum_{J\in \mc{J}(r^{-\frac{1}{n}})} |\sum_{I(\theta)\subset J}f_\theta|^2 W^{n+1,d}_B, \]
which finishes the proof. 

\end{proof}

\subsubsection{Global estimates for $L^p$, $1\le p\le 2$ \label{globL2}}

\begin{theorem}\label{smallpmn-1} For each $1\le p\le 2$, 
\[ \int_{\R^{n+1}}|\sum_{\theta\in\Xi_{n-1}^{n+1}(R)}f_\theta|^p\lesssim_\e R^\e \sum_{\theta\in\Xi_{n-1}^{n+1}(R)} \int_{\R^{n+1}}|f_\theta|^p. \] 
\end{theorem}

\begin{proof}  For each $\theta\in\Xi_{n-1}^{n+1}(R)$ and $j=0,\ldots,n-1$, let $\theta^j=\{\sum_{i=0}^n\lambda_i\phi_n^{(i)}(a)\in\theta:|\lambda_j|2^{m-j}R^{\frac{j}{n}}\sim 1\}$ and note that $\theta=\cup_{j=0}^{n-1}\theta^j$. There are smooth bump functions $\s_{\theta^j}$ which satisfy $\sum_j\s_{\theta_j}\equiv 1$ on $\theta$ and $\supp\s_{\theta_j}\subset\theta_j$. Suppose that 
\[ \int_{\R^{n+1}}|\sum_{\theta}f_\theta|^{p}\lesssim_n \int_{\R^{n+1}}|\sum_{\theta}f_{\theta}*\widecheck{\s}_{\theta^j}  |^p. \]
If $j<n-1$, then since $\theta^j\subset\tau \times\R^{n-j-1}$ for an element $\tau\in\Xi_{j+1}^{j+2}(R^{\frac{j+1}{n}})\times\R^{n-j-1}$, the theorem follows from a cylindrical version of the $n=j+1$ dimensional theorem (see Theorem 4.1 from \cite{maldagueM3} for a precise cylindrical version of square function estimate; the proof also applies to general geometries). Therefore, it suffices to consider the case where $\theta^{n-1}$ dominates for each $\theta$. 
It also suffices to show that 
\[ \int_{\R^{n+1}}|\sum_{\theta}f_{\theta}*\widecheck{\s}_{\theta^{n-1}}  |^p\lesssim_\e R^\e \int_{\R^{n+1}}\sum_{\theta}|f_{\theta}*\widecheck{\s}_{\theta^{n-1}}  |^p\]
since for each $\theta$, by Young's inequality,
\[ \int_{\R^{n+1}}|f_\theta*\widecheck{\s}_{\theta^{n-1}}|^p\lesssim \int_{\R^{n+1}}|f_\theta|^p. \]
We will slightly abuse notation by using $f_\theta$ to denote $f_\theta*\widecheck{\s}_{\theta^{n-1}}$ for the remainder of the proof. Each function $f_\theta$ then has a wave packet decomposition adapted to translates $T$ of a dual set $\theta^*$ which tile $\R^{n+1}$.  By a 
standard pigeonholing procedures outlined, for example, in \cite{maldagueM3}, it suffices to prove that for $\a>0$, $A>0$, and each ball $B_R\subset\R^{n+1}$ of radius $R$,
\begin{equation}\label{p=2case}
\a^p|\{x\in B_R:|f(x)|\sim\a\}|\lesssim_\e R^\e \sum_{\theta\in\Xi_{n-1}^{n+1}(R)}\int_{\R^{n+1}}|f_\theta|^p,
\end{equation}
where $f_\theta=\sum_{T\in\T_\theta}c_T\s_T$ is a wave packet decomposition with the properties that $\T_\theta$ is a finite set, $|c_T|\sim A$ for all $T$, and $\s_T$ are $L^\infty$-normalized wave packets subordinate to $T$ from a tiling of $\R^{n+1}$ by translates of $\theta^*$. The $p=2$ case of \eqref{p=2case} follows directly from Plancherel's theorem. It follows that  
\begin{align*} 
\a^p|\{x\in\R^{n+1}:|f(x)|\sim\a\}|&\lesssim_\e R^\e \sum_{\theta\in\Xi_{n-1}^{n+1}(R)}\a^{p-2}\int_{\R^{n+1}}|f_\theta|^2 \\
&\lesssim_\e R^\e \sum_{\theta\in\Xi_{n-1}^{n+1}(R)}\frac{A^{2-p}}{\a^{2-p}}A^p\#\T_\theta|\theta^*|\lesssim_\e R^\e \sum_{\theta\in\Xi_{n-1}^{n+1}(R)}\frac{A^{2-p}}{\a^{2-p}}\int_{\R^{n+1}}|f_\theta|^p. \end{align*}
If $\frac{A}{\a}\le 1$, then we are done. Suppose now that $\a\le A$. Then 
\begin{align*} 
\a^p|\{x\in\R^{n+1}:|f(x)|\sim\a\}|&\le A^{p-1} \sum_{\theta\in\Xi_{n-1}^{n+1}(R)} \int_{\R^{n+1}}|f_\theta| \\
&\lesssim_\e R^\e A^{p-1}\sum_{\theta\in\Xi_{n-1}^{n+1}(R)}A\#\T_\theta|\theta^*|\lesssim_\e R^\e \sum_{\theta\in\Xi_{n-1}^{n+1}(R)}\int_{\R^{n+1}}|f_\theta|^p.  
\end{align*}

\end{proof}

\section{Proof of Proposition \ref{coneinduct} \label{conesec}}

We induct on the difference $n-m$. The base case of $m=n-1$ is treated in \textsection\ref{globL2}. The main work of the proof is to analyze a multi-scale constant for $0\le m\le n-2$. A significant component of the analysis uses local versions of the square function estimates defining $\C_m^{n+1}(R)$. Local estimates follow from global estimates for $\Gamma_0^{n+1}(R)$, but not for $\Gamma_m^{n+1}(R)$ for $m>0$. This is because the sets $\Gamma_m^{n+1}(R)$, indexed by $R$, are not nested. For comparison, $\Gamma_0^{n+1}(R)\subset \Gamma_0^{n+1}(r)$ and also the $r^{-1}$-neighborhood of each $\theta\in\Xi_0^{n+1}(R)$ is approximately contained in one $\tau\in\Xi_0^{n+1}(r)$. If we localize the spatial side to a ball $B_r$ of radius $r$ using a weight function $W_{B_r}^{n+1,d}$ (see \textsection\ref{wtsec} for a precise definition) with $\supp\widehat{W_{B_r}^{n+1,d}}$ contained in a ball of radius $r^{-1}$, then for $2\le p\le p_n$,
\[ \int_{\R^{n+1}}|\sum_{\theta\in\Xi_0^{n+1}(R)}f_\theta W_{B_r}^{n+1,d}|^p\lesssim \C_0^{n+1}(r)\int_{\R^{n+1}}|\sum_{J\in\mc{J}(r^{-\frac{1}{n}})}|\sum_{\substack{I(\theta)\subset J\\ \theta\in\Xi_0^{n+1}(R)}}f_\theta W_{B_r}^{n+1,d}|^2)^{\frac{p}{2}},  \]
so $\C_0^{n+1}(\cdot)$ also bounds local square function estimates. This is no longer true for $\Gamma_m^{n+1}(R)$ when $m>0$ since $\Gamma_m^{n+1}(R)$ and $\Gamma_m^{n+1}(r)$ have nontrivial symmetric difference. Indeed, we have $\Gamma_m^{n+1}(R)\cap\Big(\{0\}^m\times\R^{n-m+1}\Big)=\{0\}^m\times R^{-\frac{m}{n}}\Gamma_0^{n-m+1}(R^{\frac{n-m}{n}})$ and it follows that 
    \begin{itemize}
        \item $\Gamma_m^{n+1}(R)\setminus\Gamma_m^{n+1}(r)\supset \{0\}^m\times R^{-\frac{m}{n}}\Gamma_0^{n-m+1}(R^{\frac{n-m}{n}})$,  
        \item $\Gamma_m^{n+1}(r)\setminus\Gamma_m^{n+1}(R)\supset \{0\}^m\times r^{-\frac{m}{n}}\Gamma_0^{n-m+1}(R^{\frac{n-m}{n}})$. 
    \end{itemize}
Since global estimates do not imply local estimates when $m>0$, we explicitly analyze multi-scale constants which are localized to spatial balls of some scale. In order to show that these constants are bounded, we first treat a special case. The special case deals with functions with refined, truncated Fourier support that is described in the following section.

\subsection{\label{multidef}Truncated Fourier support and a local multi-scale quantity}

To prove that $\C_m^{n+1}(R)\lesssim_\e R^\e$, we work with localized multi-scale quantities and various truncated versions of $\Gamma_m^{n+1}(R)$. We use the weight functions $W^{n+1,d}_B$ and $\tilde{W}_B^{n+1,d}$ introduced in \textsection\ref{wtsec} to localize the integrals.


We introduce the first truncated version of $\Gamma_m^{n+1}(R)$. Define $\mc{H}_m^n(R,K)$ to be the tuples ${\bf{h}}=(h_0,\ldots,h_m)\in (2^{-\N})^{m+1}$ satisfying the following properties: 
\begin{enumerate}
    \item $2^{-m-1}R^{-\frac{m}{n}}\le |h_m|\le R^{-\frac{m}{n}}$.
    \item $K^{-1}R^{-\frac{m}{n}}\le |h_i|\le R^{-\frac{m}{n}}$ for all $i$.
\end{enumerate}
Define the truncated $m$-cone $\Gamma_m^{n+1}(R,K,{\bf{h}})\subset\R^{n+1}$ to be the union of 
\begin{align}\label{Gammah} 
\theta_a=\big\{\sum_{i=0}^{n}\lambda_i\phi_n^{(i)}(a):&\,\,|\lambda_i-h_i|\le  K^{-\b}|h_i|\quad\text{for all }0\le i\le m,\,\, |\lambda_j|\le R^{-\frac{j}{n}}\,\, \text{if}\,\, m+1\le j\le n \}
\end{align}
where $a\in 2^{-200n}R^{-\frac{1}{n}}\Z\cap[0,1]$ and we choose $\b$ in Proposition \ref{initscale}. 
Let $\Xi_m^{n+1}(R,K,{\bf{h}})$ denote the collection of these $\theta_a$. For a scale $s\in2^{-\N}$, let $\mc{J}(s)=\{[l2^{-100n}s,(l+1)2^{-100n}s):l=0,\ldots,2^{100n}s^{-1}-1\}$. 
Define $S_{K}^{n,m}(r,\rho,R)$ to be the smallest constant such that for any vector ${\bf{h}}\in\mc{H}_m^n(R,K)$, ball $B\subset\R^{n+1}$ of radius $R^{\frac{m}{n}}\rho$, $d\ge 1$, and any $2\le p\le p_{n-m}$,
\begin{align*} 
\int_{\R^{n+1}}|\sum_{J\in\mc{J}(r^{-\frac{1}{n}})}|\sum_{\substack{I(\theta)\subset J\\\theta\in\Xi_m^{n+1}(R,K,{\bf{h}})}}&
f_\theta|^2|^{\frac{p}{2}} W_B^{n+1,d}   \\
\le S_{K}^{n,m}(r,\rho,R)&\times\int_{\R^{n+1}}|\sum_{I\in\mc{J}(\rho^{-\frac{1}{n}})}|\sum_{\substack{I(\theta)\subset I\\ \theta\in\Xi_m^{n+1}(R,K,{\bf{h}})}} f_\theta|^2|^{\frac{p}{2}} W_B^{n+1,d} 
\end{align*}
for all Schwartz functions $f_\theta:\R^{n}\to\C$ with $\supp \widehat{f}_\theta\subset \theta\in\Xi_m^{n+1}(R,K,{\bf{h}})$.

\begin{proposition} \label{SnmmKbd} Let $\e>0$. Suppose that $R\ge C_K$ and $K$ is sufficiently large depending on $\e$. Then $S_{K}^{n,m}(r,\rho,R)\lesssim_\e \rho^\e$. 
\end{proposition}

We prove Proposition \ref{SnmmKbd} in \textsection\ref{SnmmKbdsec}.  
After bounding $S_K^{n,m}(r,\rho,R)$, we will work with a more general local multi-scale constant. We do not prove a local estimate in the full generality of functions with Fourier support in $\Gamma_m^{n+1}(R)$ because we require some truncation to justify a base case for the induction on scales. Since our truncations are somewhat complex (compared to truncating to balls or annuli), it is not easy to prove that local square function estimates for truncated functions imply local square function estimates for non-truncated functions with Fourier support in all of $\Gamma_m^{n+1}(R)$. However, we show that local estimates for truncated functions imply global estimates for non-truncated functions, which is sufficient to prove Proposition \ref{coneinduct}.  

We define a more general truncated version of $\Gamma_m^{n+1}(R)$, for $1\le m\le n-2$. Let $\sigma=(\sigma_0,\ldots,\sigma_m)\in2^{\Z}$ with $\sigma_m=1$ and $\sigma_k\in[\sigma_{k+1}^{-1}\cdots\sigma_m^{-1}R^{-\frac{1}{n}},1]$ for all $k<m$. It is also notationally convenient to define $\sigma_{m+1}:=1$. Define $\Xi_m^{n+1}(R,\sigma)$ to be the union of sets
\begin{align*}
\theta_a(\sigma)&=\Big\{\sum_{i=0}^n\lambda_i\phi_n^{(i)}(t):a\le t\le a+2^{-100n} R^{-\frac{1}{n}},\quad|\lambda_i|\le \sigma_0^{\max(0,1-i)}\cdots\sigma_m^{\max(0,m+1-i)}R^{-\frac{i}{n}}\text{ for all $i$},
\\
    &\quad\text{for each $0\le k\le m$,} \begin{cases}
        \underset{0\le i\le k}{\max}2^{-i}|\lambda_i|\sigma_k^{-k-1+i}\cdots\sigma_m^{-m-1+i}R^{\frac{i}{n}}\ge 2^{-k-1} &\,\, \text{if}\,\,\sigma_k>\sigma_{k+1}^{-1}\cdots\sigma_{m+1}^{-1}R^{-\frac{1}{n}} \\
        \underset{0\le i\le k}{\max}|\lambda_i|\le \sigma_{k+1}\cdots\sigma_{m}^{m-k}R^{-\frac{k+1}{n}} &\,\, \text{if}\,\,\sigma_k=\sigma_{k+1}^{-1}\cdots\sigma_{m+1}^{-1}R^{-\frac{1}{n}}
    \end{cases} \Big\}, 
\end{align*}
where $a\in 2^{-100n} R^{-\frac{1}{n}}\Z\cap[0,1]$. Although the above definition looks complicated, we describe how to obtain a truncation from functions with Fourier support in $\Gamma_m^{n+1}(R)$ in the proof of Proposition \ref{coneinduct} in \textsection\ref{truncsec}. 
For each $2\le r\le \rho\le R$ and $\sigma$, let $S_{\sigma}^{n,m}(r,\rho,R)$ be the corresponding local multi-scale quantity defined to be the smallest constant such that for any ball $B$ of radius $R^{\frac{m}{n}}\rho$, $d\ge 1$, and any $2\le p\le p_{n-m}$,
\begin{align*} 
\int_{\R^{n+1}}|\sum_{J\in\mc{J}(r^{-\frac{1}{n}})}|\sum_{\substack{I(\theta)\subset J\\\theta\in\Xi_m^{n+1}(R,\sigma)}}
f_\theta|^2|^{\frac{p}{2}} W_B^{n+1,d}  \le S_{\sigma}^{n,m}(r,\rho,R)&\int_{\R^{n+1}}|\sum_{I\in\mc{J}(\rho^{-\frac{1}{n}})}|\sum_{\substack{I(\theta)\subset I\\ \theta\in\Xi_m^{n+1}(R,\sigma)}} f_\theta|^2|^{\frac{p}{2}} W_B^{n+1,d} 
\end{align*}
for all Schwartz functions $f_\theta:\R^{n}\to\C$ with $\supp \widehat{f}_\theta\subset \theta\in\Xi_m^{n+1}(R,\sigma)$.

\subsection{Estimate for an initial scale}

The following proposition is based on the Pramanik-Seeger approximation of cones by neighborhoods of cylinders over curves, which was also used by Bourgain-Demeter in the proof of decoupling for the cone.

\begin{proposition}[Initial scale for $0\le m\le n-2$] \label{initscale} Suppose that the hypotheses of Proposition \ref{coneinduct} hold. Suppose that $\b$ is a large enough dimensional constant. Let $\s:\R^{n+1}\to\R^{n+1}$ be a smooth function with Fourier transform supported in a ball of radius $R^{\frac{m}{n}} K^{-\b}$. For any $\d>0$, $R\ge C_K$, and $2\le p\le p_{n-m}$,
   \[ \int_{\R^{n+1}}|\sum_{\theta\in\Xi_m^{n+1}(R,K,{\bf{h}})}g_\theta\s|^{p}\lesssim_\d K^\d \int_{\R^{n+1}}|\sum_{J\in\mc{J}(K^{-1})}|\sum_{\substack{I(\theta)\subset J\\ \theta\in\Xi_m^{n+1}(R,K,{\bf{h}})}}g_\theta\s|^2|^{\frac{p}{2}}. \] 
\end{proposition}

\begin{proof} Each element of $\Xi_m^{n+1}(R,K,{\bf{h}})$ is within a distance $\lesssim K^{-B}R^{-\frac{m}{n}}$ of a point $\sum_{i=0}^m\tilde{h}_i\phi_n^{(i)}(t)$, for some $t\in[0,1]$ and $\tilde{h}_i=h_i$ if $|h_i|> K^{-1}R^{-\frac{m}{n}}$ and $\tilde{h}_i=0$ if $|h_i|\le K^{-1}R^{-\frac{m}{n}}$. We initiate an algorithm that allows us to use lower dimensional square function estimates to certain neighborhoods of one of curves $\g_n^{(0)}(t),\ldots,\g_n^{(m)}(t)$. The function $\s$ is harmless since it has the effect of taking a $K^{-\b}R^{-\frac{m}{n}}$-neighborhood on the Fourier side. We will choose constants $B_0=\b>(n+1)B_1>(n+1)^2B_2>\cdots>(n+1)^{k} B_{k}$ at the end of the proof. 

Suppose that $|\tilde{h}_0|\ge K^{-B_1}|h_m|$. Let $A_0$ be inverse of a lower triangular $(n+1)\times (n+1)$ matrix with $1$'s along the diagonal satisfying 
\[ \sum_{i=0}^m\tilde{h}_iA_0\phi_n^{(i)}(t)=\tilde{h}_0\phi_n(t)\quad\text{for all}\quad t\in[0,1]. \]
Note that by Cramer's rule, the entries of $A_0$ have norm bounded by $\lesssim K^{nB_1}$.
Therefore, 
\[ A_0\big(\Gamma_m^{n+1}(R,K,{\bf{h}})\big)\subset\mc{N}_{C|h_m|K^{nB_1-B_0}}\{\tilde{h}_0\phi_n(t):t\in[0,1]\}, \]
where $|\tilde{h}_0|\ge K^{-B_1}|h_m|$. After rescaling the set on the right hand side by $\frac{1}{\tilde{h}_0}$, it is contained in 
\[ \R\times \mc{N}_{CK^{(n+1)B_1-B_0}}\{\g_{n-m}(t):t\in[0,1]\}\times \R^{m}. \]
We will choose $B_1$ and $B_0$ so that $CK^{(n+1)B_1-B_0)}\le K^{-(n-m)}$, so that the proposition follows from $\M^{n-m}(K^{n-m})\lesssim_\d K^\d$.

Now suppose that $|\tilde{h}_0|\le K^{-B_1}|h_m|$ and $|\tilde{h}_1|\ge K^{-B_2}|h_m|$. Then 
\[\Gamma_m^{n+1}(R,K,{\bf{h}})\subset\mc{N}_{CK^{-B_1}|h_m|}\{\sum_{i=1}^m\tilde{h}_iR^{-\frac{i}{n}}\phi_n^{(i)}(t):t\in[0,1]\}.\]
Similar to the previous case, we can define an $(n+1)\times (n+1)$ matrix $A_1$ satisfying  
\[ A_1\Big(\sum_{i=1}^m\tilde{h}_i\phi_n^{(i)}(t)\Big)= \tilde{h}_1\phi_n^{(1)}(t)\quad\text{for all}\quad t\in[0,1] \]
and $A_1\big(\Gamma_m^{n+1}(R,K,{\bf{h}})\big)\subset\R\times \mc{N}_{CK^{nB_2-B_1}|h_m|}\{\tilde{h}_1\g_{n-m+1}^{(1)}(t):t\in[0,1]\}\times \R^{m-1}$. After dilating by $\frac{1}{\tilde{h}_1}$, we have
\[ \frac{1}{\tilde{h}_1}A_1\big(\Gamma_m^{n+1}(R,K,{\bf{h}})\big)\subset \mc{N}_{CK^{(n+1)B_2-B_1}}\{\phi_{n}^{(1)}(t):t\in[0,1]\}. \]
The curve $\phi_{n}^{(1)}(t)$ is contained in the product of $\R^{m+1}$ with an anisotropic rescaling of $\g_{n-m}(t)$, so if $(n+1)B_1-B_2\le -(n-m)$, then the proposition again follows from $\M^{n-m}(K^{n-m})\lesssim_\d K^\d$.

We iterate this procedure until either the proposition is proved, or we are in the final case that $|\tilde{h}_i|\ge K^{-B_i}|h_m|$ for each $i=0,\ldots,m-1$. Then 
\[ \Gamma^{n+1}(R,K,{\bf{h}})\subset \mc{N}_{K^{-B_{k-1}}|h_m|}\{{h}_m\phi_n^{(m)}(t):t\in[0,1]\}.\] 
Again, use an affine transformation $A_m$ which transforms 
\[ A_m\big(\Gamma^{n+1}(R,K,{\bf{h}})\big)\subset\mc{N}_{CK^{-B_{k-1}}|h_m|}\{h_m\phi_n^{(m)}(t):t\in[0,1]\}. \]
Since $\phi_n^{(k)}$ contains an anisotropically rescaled version of $\p_{n-m}$, we are done as long as $B_k\ge n-m$. 

The choice $B_0=(n-m)(n+2)^n$ and $B_i=(n+2)^{n-i}(n-m)$ satisfies the necessary properties.

\end{proof}

\subsubsection{The multi-scale inequality}

We prove a key multi-scale inequality in Proposition \ref{conemulti}. Note that when $m=n-1$, we have a bound for $S_{K}^{n,n-1}(r,\rho,R)$ from the $L^2$-based analysis in \textsection\ref{L2sec}. 

\begin{proposition}\label{L2conekakloc} Let $1\le n$. Then for $K$ larger than a dimensional constant, $S_{K}^{n,n-1}(r,\rho,R)\lesssim_\e \rho^\e$. 
\end{proposition}

\begin{proof} This follows immediately from Theorem \ref{locL2thm}. 
\end{proof}

\begin{proposition}\label{conemulti} Suppose that the hypotheses of Proposition \ref{coneinduct} hold and that $0\le m\le n-2$. Suppose also that $S_K^{n,k}(r,\rho,R)\lesssim_\d \rho^\d$ for each $m+1\le k\le n-1$. Then is an absolute constant $C>0$ such that for any $\d_0\in(0,2^{-100n})$, 
\[ S_{K}^{n,m}(r,\rho,R)\lesssim_{\d_0} r^{\d_0} S_{K}^{n,m}(1,\rho/r,R/r)+\sum_{k=1}^{n-m} r^{C\d_0^{n-k+1}}\Big[S_{K}^{n,m}(\min(\rho,r^{1+\d_0^{n-k}}),\rho,R) \Big]. \]
\end{proposition}
The implicit constant is uniform in ${\bf{h}}\in\mc{H}_m^{n}(R,K)$. 
We work with a sequence of exponents $\tilde{p}_n$ which is an even version of the sequence $p_n$. Let $\tilde{p}_1=2$ and for $n\ge2$, let
\[ \tilde{p}_n:=\begin{cases}\sum_{i=1}^ni+1\quad&\text{if} \quad n\equiv1\text{ or }2\pmod{4}\\
\sum_{i=1}^ni\quad&\text{if} \quad n\equiv0\text{ or }3\pmod{4}
\end{cases} . \]
Before proving Proposition \ref{conemulti}, we need to verify some properties about the sequence of exponents $\tilde{p}_n$, which are recorded in the following lemma. 
\begin{lemma}\label{pprops}
\begin{enumerate} For each $n\ge 2$, we have
    \item $2\le\tilde{p}_n\le p_n$,
    \item $\tilde{p}_n\le \tilde{p}_{n+1}$,
    \item $\frac{\tilde{p}_{n}}{\tilde{p}_{n-1}}\le 2$, 
    \item for $n\ge k\ge 1$, if $\frac{p_n}{\tilde{p}_k}>2$, then $\frac{p_n}{\tilde{p}_k}\le p_{n-k}$. 
\end{enumerate}    
\end{lemma}
\begin{proof} The first two properties are clear. To verify property (3), first observe that that $\frac{2}{1}$, $\frac{4}{2}$, and $\frac{6}{4}$ are all $\le 2$. Therefore, it suffices to check that $\frac{k^2+k+2}{k^2-k}\le 2$ for $k\ge 4$. This is equivalent to $0\le k^2-3k-2$, which is true when $k\ge 4$.

Now we verify the last property. The cases $k=1,2$ are easy to check, so we assume that $k\ge 3$. We need to show that if $p_n-2\tilde{p}_k>0$, then $p_{n-k}\tilde{p}_k-p_n\ge 0$. The inequality $p_n-2\tilde{p}_k>0$ is a quadratic inequality in $n$ with a positive and a negative root. Since $n\ge 0$, the inequality is equivalent to the condition that
\begin{align}
    \label{equiv1} n> -\frac{1}{2}+\frac{1}{2}\sqrt{8 k^2 + 8 k - 7}.
\end{align}
The inequality we need to verify is also quadratic in $n$: $0\le \frac{1}{2}(k^2+k-2)n^2+[\frac{1}{2}(k^2+k)(-2k+1)-1]n+\frac{1}{2}k^4+\frac{1}{2}k^2+k-2$. The roots of this inequality are 
\[ \frac{2 k^3 + k^2 - k + 2\pm \sqrt{k^4 - 6 k^3 + 21 k^2 + 28 k - 28}}{2 (k^2 + k - 2)}. 
\]
It suffices to verify that the larger root is less than or equal to the right hand side of \eqref{equiv1}. Indeed, we have for $k\ge 3$, 
\begin{align*}
-(k^2+k-2)+(k^2+k-2)&\sqrt{8k^2+8k-8}-\big[2k^3+k^2-k+2+\sqrt{k^4-6k^3+21k^2+28k-28}\big] \\
    &= (k^2+k-2)\sqrt{8k^2+8k-8}-2k^3-2k^2-\sqrt{k^4-6k^3+21k^2+28k-28} \\
    &\ge(2\sqrt{2}-2)k^3+(k-2)\sqrt{8k^2+8k-8}-2k^2-\sqrt{k^4+3k^2+28k} \\
    &\ge(2\sqrt{2}-2)k^3+(k-2)2\sqrt{2}k-(2+\frac{8}{3\sqrt{3}})k^2\\
    &\ge k\left[(2\sqrt{2}-1)k^2-(2+\frac{8}{3\sqrt{3}}-2\sqrt{2})k-2\sqrt{2} \right]\\
    &\ge k\left[\frac{8}{5}k-2\sqrt{2} \right]\\
    &\ge 0.    
\end{align*}
    
\end{proof}

Next, we prove some key lemmas that we will use in the proof of Proposition \ref{conemulti}. Let $2\le p\le p_{n-m}$. We describe the set-up needed to state the lemmas. Consider the left hand side of the defining inequality for $S_{K}^{n,m}(r,\rho,R)$:
\begin{equation}\label{now2} 
\int_{\R^{n+1}} |\sum_{J\in\mc{J}(r^{-\frac{1}{n}})}|\sum_{\substack{I(\theta)\subset J\\ \theta\in\Xi_m^{n+1}(R,K,{\bf{h}})}} f_\theta|^2|^{\frac{p}{2}}W, \end{equation}
where $W=W_{B}^{n+1,d}$ for a ball $B\subset\R^{n+1}$ a ball of radius $R^{\frac{m}{n}}\rho$.

Let $n_{p}=\min\{l\ge 2:1\le \frac{p}{\tilde{p}_{l-1}}\le 2\}$. Our iteration will be indexed by a parameter $k$ in the range $2\le k\le n_p$.

Consider $J\in \mc{J}(r^{-\frac{1}{n}})$ with initial point $a\in[0,1]$. The Fourier support of $|\sum_{I(\theta)\subset J}f_\theta|^2$ is contained in 
\begin{equation}\label{claim} 
\{\sum_{i=0}^n\lambda_i\phi_n^{(i)}(a):|\lambda_i|\le R^{-\frac{m}{n}}r^{\frac{m-i}{n}} \}.\end{equation}
We define auxiliary functions which we will use to decompose the above Fourier support. Let $s:\R\to[0,1]$ be the bump function used in \eqref{auxbump} and note that the function 
\[ \s^k(x_0,\ldots,x_n):=s(R^{\frac{m}{n}}r^{-\frac{m}{n}}x_0)s(R^{\frac{m}{n}}r^{\frac{1-m}{n}} x_1)\cdots s(R^{\frac{m}{n}}r^{\frac{k-1}{n}}x_{m+k-1})s(x_{m+k})\cdots s(x_n)\] 
is identically $1$ on the rectangle $[-R^{-\frac{m}{n}}r^{\frac{m}{n}},R^{-\frac{m}{n}}r^{\frac{m}{n}}]\times\cdots\times[-R^{-\frac{m}{n}}r^{-\frac{k-1}{n}},R^{-\frac{m}{n}}r^{-\frac{k-1}{n}}]\times[-1,1]^{n+1-m-k}$. Let $A^k_{J}:\R^{n+1}\to \R^{n+1}$ be a linear map satisfying 
\begin{align} 
A_{J}^k\big([-R^{-\frac{m}{n}}r^{\frac{m}{n}},R^{-\frac{m}{n}}r^{\frac{m}{n}}]&\times\cdots\times[-R^{-\frac{m}{n}}r^{-\frac{k-1}{n}},R^{-\frac{m}{n}}r^{\frac{m}{n}}r^{-\frac{k-1}{n}}]\times[-1,1]^{n+1-m-k}\big) \nonumber\\
\label{etaFsupp} &=\{\sum_{i=0}^{m+k-1}\lambda_i\phi_{m+k-1}^{(i)}(a):|\lambda_i|\le R^{-\frac{m}{n}}r^{\frac{m-i}{n}}\}\times[-1,1]^{n+1-m-k} . 
\end{align}
Define $\s_{J_{m+1}}^k(x_0,\ldots,x_n)$ by $\s^k\circ (A_{J}^k)^{-1}$. For each dyadic $\sigma\in[r^{-1/n},1]$, define 
\[ \s_{J,\sigma}^k(x_0,\ldots,x_n)=\s_{J_{m+1}}^k(\sigma^{-m-k}x_0,\sigma^{-m-k+1}x_1,\ldots,\sigma^{-1}x_{m+k-1},x_{m+k},\ldots,x_n) \]
and note that 
\[ \s_{J}^k=\s_{J,r^{-\d_k}}^k+\sum_{r^{-\d_k}<\sigma\le 1}(\s_{J,\sigma}^k-\s_{J,\sigma/2}^k)=:\eta_{J,r^{-\d_k}}^k+\sum_{r^{-\d_k}<\sigma\le 1}\eta_{J,\sigma}^k.  \]

In the following lemmas, $\theta$ always indexes the sets in $\Xi_m^{n+1}(R,K,{\bf{h}})$. Our iteration involves convolving weight functions, which loses a factor in the decay rate. There are at most $n$ many steps in the iteration, so we let $d_k=4(n-k)d$. For each $J=[a,a+2^{-100n}r^{-\frac{1}{n}})$, let $\tilde{W}_{J,m+k-1}^{n+1,d_k}$ be an $L^1$-normalized weight function (as was introduced in \textsection\ref{multidef}) which is Fourier supported in 
\[ \{\sum_{i=0}^{m+k-2}\lambda_i\phi_{m+k-2}^{(i)}(a):|\lambda_i|\le R^{-\frac{m}{n}} r^{\frac{m-i}{n}}\}\times [-2,2]^{n+2-m-k}. \]
Let $\tilde{W}_{J,m+k,r^{-\d_k}}^{n+1,d_k}$ be a weight function that is centered at the origin and that has the same Fourier support as $\widecheck{\eta}_{J,r^{-\d_k}}^k$. Let $0<\d_0$ and define $\d_k=\d_0^{n-k}$.
\begin{lemma}\label{multilem1} The integral \eqref{now2} is bounded by 
\begin{align*} 
\sum_{2\le k\le n_p} &C_{\d_{k-1}}r^{C\d_{k-1}}\int_{\R^{n+1}}|\sum_{J_{m+1}}|\sum_{\substack{I(\theta)\subset J
}}f_\theta|^{\tilde{p}_{k-1}}*\tilde{W}_{J,m+k-1}^{n+1,d_k}*\widecheck{\eta}_{J,r^{-\d_k}}^k|^{\frac{p}{\tilde{p}_{k-1}}}W\\
    &+C_{\d_{n_p}}r^{\d_{n_p}}\int_{\R^{n+1}}\sum_{J_{m+1}}|\sum_{\substack{I(\theta)\subset J_{m+1}
    }}f_\theta|^p W. \end{align*}    
\end{lemma}


\begin{lemma}\label{ptwise} For each $1\le k\le n_p$ and $J\in\mc{J}(r^{-\frac{1}{n}})$, 
we have
\[ ||\sum_{\substack{I(\theta)\subset J_{m+1}
}}f_\theta|^{\tilde{p}_k}*\widecheck{\eta}_{J,r^{-\d_k}}^k|\lesssim_\d r^\d |\sum_{I_k \subset J_{m+1}}|\sum_{I(\theta)\subset I_k}f_\theta|^2|^{\frac{\tilde{p}_{k-1}}{2}}*\tilde{W}_{J,{m+k},r^{-\d_k}}^{n+1,d_k}, \]
where $I_k\in\mc{J}\big(\max(\rho^{-\frac{1}{n}},r^{-\frac{(1+\d_k)}{n}})\big)$. 
\end{lemma}

\begin{lemma}\label{multilem3} 
Let $I_k$ vary over intervals in $\mc{J}\big(\max(\rho^{-\frac{1}{n}},r^{-\frac{1+\d_k}{n}})\big)$. 
Then 
\begin{align*} 
\int_{\R^{n+1}}|\sum_{J}|\sum_{I_k\subset J}|\sum_{\substack{I(\theta)\subset I_{k}
}}f_\theta|^2|^{\frac{\tilde{p}_{k-1}}{2}} *\tilde{W}_{J,{m+k},r^{-\d_k}}^{n+1,d_k}|^{\frac{p}{\tilde{p}_{k-1}}}W\lesssim_\d r^\d \int_{\R^{n+1}}|\sum_{I_{k}}|\sum_{\substack{I(\theta)\subset I_{k}
}}f_\theta|^2|^{\frac{p}{2}}W. 
\end{align*}

\end{lemma}

We remark that we are free to choose any $\d>0$ in the notation $\lesssim_\d r^\d$ from Lemmas \ref{ptwise} and \ref{multilem3}, and this parameter is independent of $\d_0>0$. Now we show how Lemmas \ref{multilem1}, \ref{ptwise}, and \ref{multilem3} imply Proposition \ref{conemulti}. 

\begin{proof}[Proof of Proposition \ref{conemulti}]  
Our objective is to bound \eqref{now2}. Begin by using Lemma \ref{multilem1}. There are two possible outcomes. Suppose first that \eqref{now2} is bounded by 
\[ C_{\d_{n_p}}r^{\d_{n_p}}\int_{\R^{n+1}}\sum_{J_{m+1}}|\sum_{\substack{I(\theta)\subset J
    }}f_\theta|^p W^p.  \]
Let $J$ be the interval with initial point $a$ and of length $2^{-100n}r^{-\frac{1}{n}}$. Let $A$ be the affine map mapping $\phi_{n}^{(i)}(a)\mapsto r^{\frac{i}{n}}\phi_n^{(i)}(0)$ for all $i$. Then the function $f_\theta\circ (A^{-1})^t$ has Fourier transform supported in $A(\theta)\in\Xi_m^{n+1}(R/r,K,{\bf{\tilde{h}}})$, where $\tilde{h}=(\tilde{h}_0,\ldots,\tilde{h}_m)$ is defined by $\tilde{h}_i=r^{\frac{i}{n}}h_i$ if $r^{\frac{i}{n}}|h_i|\ge K^{-1}r^{\frac{m}{n}}R^{-\frac{m}{n}}$ and $\tilde{h}_i=K^{-1}r^{\frac{m}{n}},R^{-\frac{m}{n}}$ otherwise. Note that ${\bf{\tilde{h}}}$ is an element of $\mc{H}_m^n(R/r,K)$. The weight function is localized to an ellipsoid which may be approximated with finitely overlapping balls of radius $R^{\frac{m}{n}}r^{-\frac{m}{n}}(\rho/r)$, so we have 
\begin{align*} 
\int_{\R^{n+1}}|\sum_{\substack{I(\theta)\subset J
    }}&f_\theta\circ(A^{-1})^t|^p W\circ(A^{-1})^t\\
    &\lesssim S^{n,m}_{K}(1,\rho/r,R/r)\int_{\R^{n+1}}|\sum_{\substack{I\in\mc{J}(\rho^{-\frac{1}{n}})\\ I\subset J}} |\sum_{\substack{I(\theta)\subset I
    }}f_\theta\circ(A^{-1})^t|^2|^{\frac{p}{2}} W\circ(A^{-1})^t. 
\end{align*}
Summing over $J$ and using $\|\cdot\|_{\ell^{p/2}}\le\|\cdot\|_{\ell^1}$, conclude in this case that \eqref{now2} is bounded by 
\[C_{\d_{n_p}}r^{\d_{n_p}} S^{n,m}_{K}(1,\rho/r,R/r)\int_{\R^{n+1}}|\sum_{\substack{I\in\mc{J}(\rho^{-\frac{1}{n}})}} |\sum_{\substack{I(\theta)\subset I
    }}f_\theta|^2|^{\frac{p}{2}} W. \]

The remaining case is if for some $2\le k\le n_p$, \eqref{now2} is bounded by 
\[ C_{\d_{k-1}}r^{C\d_{k-1}}\int_{\R^{n+1}}|\sum_{J}|\sum_{\substack{I(\theta)\subset J}}f_\theta|^{\tilde{p}_{k-1}}*\tilde{W}_{J,m+k-1}^{n+1,d_{k-1}}*\widecheck{\eta}_{J,r^{-\d_k}}^k|^{\frac{p}{\tilde{p}_{k-1}}}W . \]
Then apply Lemma \ref{ptwise} pointwise to bound the above expression by 
\[ C_{\d_{k-1}}r^{C\d_{k-1}}\int_{\R^{n+1}}|\sum_{J_{m+1}}|\sum_{I_k\subset J_{m+1}}|\sum_{\substack{I(\theta)\subset I_k}}f_\theta|^2|^{\frac{\tilde{p}_{k-1}}{2}}*\tilde{W}_{J,m+k,r^{-\d_k}}^{n+1,d_k}|^{\frac{p}{\tilde{p}_{k-1}}}W  \]
where $I_k\in\mc{J}\big(\max(\rho^{-\frac{1}{n}},r^{-\frac{(1+\d_k)}{n}})\big)$ and $\tilde{W}_{J,{m+k},r^{-\d_k}}^{n+1,d_k}$ has the same Fourier support as $\widecheck{\eta}_{J_{m+1},r^{-\d_k}}^k$. Finally, by Lemma \ref{multilem3}, we have bounded \eqref{now2} by
\[ C_{\d_{k-1}}r^{C\d_{k-1}}\int_{\R^{n+1}}|\sum_{I_k}|\sum_{\substack{I(\theta)\subset I_k}}f_\theta|^2|^{\frac{p}{2}}W.  \]
Conclude in this case that 
\[ S_{K}^{n,m}(r,\rho,R)\le C_{\d_{k-1}}r^{C\d_{k-1}} S_K^{n,m}(\min(\rho,r^{1+\d_k}),\rho,R). \]

\end{proof}

It remains to prove the lemmas. 

\begin{proof}[Proof of Lemma \ref{multilem1}]
We initiate an iteration to bound \eqref{now2}. 

\noindent\fbox{Initial step:} 
For each $J$ with $J=[a,a+2^{-100n}r^{-\frac{1}{n}})$, the Fourier support of $|\sum_{\substack{I(\theta)\subset J\\ \theta\in\Xi_m^{n+1}(R,K,{\bf{h}})}} f_\theta|^2$ is contained in 
\begin{align} 
& \label{disabo}   \big\{\sum_{i=0}^{n}\lambda_i\phi_n^{(i)}(a):|\lambda_i|\le R^{-\frac{m}{n}}r^{\frac{m-i}{n}} \big\}\subseteq \{\sum_{i=0}^m\lambda_i\phi_{m}^{(i)}(a):|\lambda_i|\le R^{-\frac{m}{n}} r^{\frac{m-i}{n}}\}\times [-2,2]^{n-m} . 
\end{align}
It follows from the locally constant property, \eqref{now2} is bounded by 
\[\int_{\R^{n+1}} |\sum_{J\in\mc{J}(r^{-\frac{1}{n}})}|\sum_{\substack{I(\theta)\subset J\\ \theta\in\Xi_m^{n+1}(R,K,{\bf{h}})}} f_\theta|^2*\tilde{W}_{J,{m+1}}^{n+1,d_1}|^{\frac{p}{2}}W, 
 \]
where $\tilde{W}_{J,{m+1}}^{n+1,d}$ is an $L^1$-normalized weight function that is centered at the origin and Fourier supported in the right hand side of \eqref{disabo}. Let $n_{p}=\min\{l\ge 2:1\le \frac{p}{\tilde{p}_{l-1}}\le 2\}$. Let $2\le k\le n_p$ and define $\d_k=\d_0^{n-k}$. The input for step $k$ is a bound for \eqref{now2} of the form 
\begin{equation}\label{stepkeqn} C_{\d_{k-1}} r^{C\d_{k-1}} \int_{\R^{n+1}}|\sum_{J\in\mc{J}(r^{-\frac{1}{n}})}|\sum_{I(\theta)\subset J}f_\theta|^{\tilde{p}_{k-1}}*\tilde{W}_{J,{m+k-1}}^{n+1,d_{k-1}}|^{\frac{p}{\tilde{p}_{k-1}}} W, \end{equation}
where $\tilde{W}_{J,m+k-1}^{n+1,d_{k-1}}$ is Fourier supported in 
\[\{\sum_{i=0}^{m+k-2}\lambda_i\phi_{m+k-2}^{(i)}(a):|\lambda_i|\le R^{-\frac{m}{n}} r^{\frac{m-i}{n}}\}\times [-2,2]^{n+2-m-k}.  \]

\vspace{2mm}
\noindent\fbox{Step $k$, $k<n_p$:} Suppose that $2\le k< n_p$, so $2\le \frac{p}{\tilde{p}_{k-1}}$, and that \eqref{now2} is bounded by \eqref{stepkeqn}. 
We will use the auxiliary functions described before Lemma \ref{multilem1} to decompose the Fourier support of each summand $|\sum_{I(\theta)\subset J} f_\theta|^{\tilde{p}_{k-1}}*\tilde{W}_{J,{m+k-1}}^{n+1,d_{k-1}}$. Since $\tilde{p}_{k-1}$ is even, the Fourier support of $|\sum_{I(\theta)\subset J}f_\theta|^{\tilde{p}_{k-1}}$ is contained in the left hand side of \eqref{disabo} (dilated by a harmless constant factor). If \eqref{stepkeqn} is bounded by 
\begin{equation}
CC_{\d_{k-1}}r^{C\d_{k-1}}\int_{\R^{n+1}}|\sum_{J \in\mc{J}(r^{-\frac{1}{n}})}|\sum_{I(\theta)\subset J}f_\theta|^{\tilde{p}_{k-1}}*\tilde{W}_{J,{m+k-1}}^{n+1,d_{k-1}}*\widecheck{\eta}_{J,r^{-\d_k}}^k|^{\frac{p}{\tilde{p}_{k-1}}} W, \end{equation}
then the lemma is proved and the iteration halts. Otherwise, assume that \eqref{stepkeqn} is bounded by 
\begin{equation}\label{displayedabovek} 
(\log r)^CC_{k-1}r^{C\d_{k-1}} \int_{\R^{n+1}}|\sum_{J\in\mc{J}(r^{-\frac{1}{n}})}|\sum_{I(\theta)\subset J}f_\theta|^{\tilde{p}_{k-1}}*\tilde{W}_{J,{m+k-1}}^{n+1,d_{k-1}}*\widecheck{\eta}_{J,\sigma}^k|^{\frac{p}{\tilde{p}_{k-1}}} W  \end{equation}
for some dyadic $\sigma$, $r^{-\d_k}\le \sigma\le 1$. The function $\eta_{J,\sigma}^k$ is supported in 
\begin{align} 
\label{J-Jsigma}
\big\{\sum_{i=0}^{m+k-1}\lambda_i\phi_{m+k-1}^{(i)}(a):\max_{0\le i\le m+k-1}&|\lambda_i|2^{-i}\sigma^{i-m-k}R^{\frac{m}{n}}r^{\frac{i-m}{n}}\in[2^{-m-k},1]\big\}\times[-2,2]^{n+1-m-k}.  \end{align}
Therefore, the Fourier support of each $|\sum_{I(\theta)\subset J }f_\theta|^{\tilde{p}_{k-1}}*\tilde{W}_{J,{m+k-1}}^{n+1,d_{k-1}}*\widecheck{\eta}_{J,\sigma}^k$ is contained in 
\[ \{\sum_{i=0}^n\lambda_i\phi_n^{(i)}(a):|\lambda_i|\le \sigma^{m+k-i}R^{-\frac{m}{n}}r^{\frac{m-i}{n}},\quad\max_{0\le i\le m+k}|\lambda_i|2^{-i}\sigma^{i-m-k}R^{\frac{m}{n}}r^{\frac{i-m}{n}}\in[2^{-m-k},1]\},\]
which is contained in a unique element $\tau\in\Xi_{m+k-1}^{n+1}(\sigma^{n}r)$ dilated by a factor of $\sigma^{m+k}R^{-\frac{m}{n}}r^{\frac{m}{n}}$. There are $\sim\sigma^{-1}$ many adjacent intervals $J_{m+1}$ whose corresponding sets \eqref{J-Jsigma} are identified with a single $\tau$. Note that after rescaling the spatial side by a factor of $\sigma^{m+k}R^{-\frac{m}{n}}r^{\frac{m}{n}}$, $W$ becomes a weight function localized to a ball of radius $\sigma^{m+k}r^{\frac{m}{n}}\rho\ge \sigma^n r$. We assumed in this case that $2\le \frac{p}{\tilde{p}_{k-1}}$ and by (4) of Lemma \ref{pprops}, we have $\frac{p}{\tilde{p}_{k-1}}\le p_{n-m-k+1}$. Therefore, we may apply the hypothesis that $\C_{m+k-1}^{n+1}(r)\lesssim_\d r^\d$ and Lemma \ref{loclem} to bound \eqref{displayedabovek} by 
\[ C_{\d_k} r^{C\d_k}  \int_{\R^{n+1}} |\sum_{J } ||\sum_{I(\theta)\subset J }f_\theta|^{\tilde{p}_{k-1}}*\tilde{W}_{J,{m+k-1}}^{n+1,d_{k-1}}*\widecheck{\eta}_{J,\sigma}^k|^2|^{\frac{p}{2\tilde{p}_{k-1}}}
W.
\]
Since $2\ge \frac{\tilde{p}_k}{\tilde{p}_{k-1}}$, use $\|\cdot\|_{\ell^2}\le \|\cdot\|_{\ell^{\tilde{p}_k/\tilde{p}_{k-1}}}$ to bound the above displayed expression by 
\[ C_{\d_k} r^{C\d_k}  \int_{\R^{n+1}} |\sum_{J} ||\sum_{I(\theta)\subset J}f_\theta|^{\tilde{p}_{k-1}}*\tilde{W}_{J,{m+k-1}}^{n+1,d_{k-1}}*\widecheck{\eta}_{J,\sigma}^k|^{\frac{\tilde{p}_k}{\tilde{p}_{k-1}}} |^{\frac{p}{\tilde{p}_k}}
W.
\]
Since $\sigma\ge r^{-\d_k}$, for each $J$, we have $|\widecheck{\eta}_{J,\sigma}^k|\lesssim r^{C\d_k}\tilde{W}_{J,{m+k}}^{n+1,d_k}$. The weight functions $\tilde{W}_{J,{m+k}}^{n+1,d_k}$ are $L^1$ normalized and satisfy  $\tilde{W}_{J,{m+k-1}}^{n+1,d_{k-1}}*\tilde{W}_{J,{m+k}}^{n+1,d_k}\lesssim \tilde{W}_{J,{m+k}}^{n+1,d_k}$, so by Cauchy-Schwarz (applied pointwise to the integral from the convolution), the previous displayed expression is bounded by 
\[ C_{\d_k} r^{C\d_k}  \int_{\R^{n+1}} |\sum_{J} |\sum_{I(\theta)\subset J}f_\theta|^{\tilde{p}_{k}}*\tilde{W}_{J,{m+k}}^{n+1,d_k}|^{\frac{p}{\tilde{p}_k}}
W.\]
This concludes Step $k$.

\vspace{2mm}
\noindent\fbox{Step $n_p$:} The iteration reaches step $n_p$ if it has produced a bound of \eqref{now2} by \eqref{stepkeqn} with $k=n_p$. Note that by the definition of $n_p$, we have $1\le \frac{p}{\tilde{p}_{n_p-1}}\le 2$. As in Step $k$, we decompose the Fourier supports of each summand $|\sum_{I(\theta)\subset J}f_\theta|^{\tilde{p}_{n_p-1}}*\tilde{W}_{J,{m+n_p-1}}^{n+1,d_{n_p-1}}$ into $\lesssim (\log r)$ many pieces. However, instead of performing this decomposition in the first $(m+n_p)$ many coordinates, we decompose the first $n$ many coordinates according to the dyadic parameter $\sigma$, $r^{-\d_{n_p}}\le \sigma\le 1$, into sets
\begin{equation}\label{sets} 
\{\sum_{i=0}^n\lambda_i\phi_n^{(i)}(a):\max_{0\le i\le n-1}|\lambda_i|\sigma^{i-n}R^{\frac{m}{n}}r^{\frac{i-m}{n}}\in[2^{-n},1],\quad|\lambda_n|\le R^{-\frac{m}{n}}r^{\frac{m-n}{n}}\} \end{equation}
(and $\{\sum_{i=0}^n\lambda_i\phi_n^{(i)}(a):|\lambda_i|\le R^{-\frac{m}{n}}r^{\frac{m-n}{n}}\}$). Suppose that \eqref{stepkeqn} is bounded by 
\begin{equation}\label{fromhere} 
(\log r)^CC_{n_p-1}r^{C\d_{n_p-1}}\int_{\R^{n+1}}|\sum_{J}|\sum_{I(\theta)\subset J}f_\theta|^{\tilde{p}_{n_p-1}}*\tilde{W}_{J,{m+n_p-1}}^{n+1,d_{n_p-1}}*\widecheck{\eta}_{J,\sigma}^{n-m}|^{\frac{p}{\tilde{p}_{n_p-1}}} W. \end{equation}
If $\sigma=r^{-\d_{n_p}}$, then the lemma is proved and the iteration halts. The alternative is that $\sigma>r^{-\d_{n_p}}$. The sets \eqref{sets} (for a fixed $\sigma$ and $a$ varying over $r^{-\frac{1}{n}}\Z\cap[0,1]$) may be organized into elements $\tau\in\Xi_{n-1}^{n+1}(\sigma^nr)$ dilated by $\sigma^nR^{-\frac{m}{n}}r^{\frac{m}{n}}$. Then invoke Theorem \ref{locL2thm} to bound \eqref{fromhere} by 
\begin{equation}\label{almostdone} 
(\log r)^CC_{\d_{n_p-1}}r^{C\d_{n_p-1}}\sum_{I\in\mc{J}(\sigma^{-1}r^{-\frac{1}{n}})}\int_{\R^{n+1}}|\sum_{J\subset I}|\sum_{I(\theta)\subset J}f_\theta|^{\tilde{p}_{n_p-1}}*\tilde{W}_{J,{m+n_p-1}}^{n+1,d_{n_p-1}}*\widecheck{\eta}_{J,\sigma}^{n-m}|^{\frac{p}{\tilde{p}_{n_p-1}}}W . \end{equation}
For each $J$, we have $\tilde{W}_{J,m+n_p-1}^{n+1,d_{n_p-1}}*|\widecheck{\eta}_{J,\sigma}^{n-m}|*W\lesssim r^{C\d_{n_p}}W$. Therefore, by H\"{o}lder's inequality (applied to the integral from the convolution) and using the fact that $\sigma>r^{-\d_{n_p}}$, the previous displayed expression is bounded by 
\begin{align*} 
(\log R)^C&C_{\d_{n_p}} r^{C\d_{n_p}}\sum_{J}  \int_{\R^{n+1}} |\sum_{I(\theta)\subset J}f_\theta|^{p}W. 
\end{align*}

\end{proof}

\begin{proof}[Proof of Lemma \ref{ptwise}] 
Begin by using $|\widecheck{\eta}_{J,r^{-\d_k}}^k|\lesssim \tilde{W}_{J,{m+k},r^{-\d_k}}^{n+1,d_k}$, where $\tilde{W}_{J,{m+k},r^{-\d_k}}^{n+1,d_k}$ is Fourier supported in 
\begin{equation}\label{a&e} 
\big\{\sum_{i=0}^{m+k-1}\lambda_i\phi_{m+k-1}^{(i)}(a):|\lambda_i|\le R^{-\frac{m}{n}}r^{\frac{m-i}{n}}(r^{-\d_k})^{m+k-i}\big\}\times[-2,2]^{n+1-m-k}.\end{equation}
We have 
\[ ||\sum_{\substack{I(\theta)\subset J\\ \theta\in\Xi_m^{n+1}(R,K,{\bf{h}})}}f_\theta|^{\tilde{p}_{k-1}}*\widecheck{\eta}_{J,r^{-\d_k}}^k|\lesssim |\sum_{\substack{I(\theta)\subset J\\ \theta\in\Xi_m^{n+1}(R,K,{\bf{h}})}}f_\theta|^{\tilde{p}_{k-1}}*\tilde{W}_{J,{m+k},r^{-\d_k}}^{n+1,d_k}. \]
Write the integral on the right hand side evaluated at a point $x\in\R^{n+1}$:
\begin{equation}\label{conv} \int_{\R^{n+1}}|\sum_{I(\theta)\subset J}f_\theta(y)|^{\tilde{p}_{k-1}}\tilde{W}_{J,m+k,r^{-\d_k}}^{n+1,d_k}(x-y)dy. \end{equation}
This integral looks like a local $L^{\tilde{p}_{k-1}}$ expression of a function with Fourier support in $\cup_{I(\theta)\subset J}\theta$. The next step is to rescale this integral. 

    Let $J$ be the interval with initial point $a$ and of length $2^{-100n}r^{-\frac{1}{n}}$. Let $A:\R^{n+1}\to\R^{n+1}$ be the linear map acting as $\phi_{m+k-1}^{(i)}(a)\mapsto r^{\frac{i}{n}}\phi_{m+k-1}^{(i)}(0)$ in the first $m+k$ coordinates and as the identity in the remaining coordinates. Applying $A$ to the Fourier support of $\tilde{W}_{J,m+k,r^{-\d_k}}^{n+1,d_k}$ \eqref{a&e} yields
\begin{equation}\label{AJ-J}
\big\{\sum_{i=0}^{m+k-1}\lambda_i\phi_{m+k-1}^{(i)}(0): |\lambda_i|\le R^{-\frac{m}{n}}r^{\frac{m}{n}}(r^{-\d_k})^{m+k-i} \quad\forall i\big\}\times[-2,2]^{n+1-m-k}.
\end{equation}
This set is contained in $B\times \R^{n+1-m-k}$, where $B\subset\R^{m+k}$ is a ball of radius $R^{-\frac{m}{n}}r^{\frac{m}{n}}r^{-\d_k}$ centered at the origin. It follows that in the first $m+k$ coordinates, $\tilde{W}_{J,{m+k},r^{-\d_k}}^{n+1,d_k}$ is localized to a finitely overlapping dual balls $B^*$ of radius $R^{\frac{m}{n}}r^{-\frac{m}{n}}r^{\d_k}$. For each $\theta$ with $I(\theta)\subset J$, $A(\theta)\subset \tau_\theta\times[-2,2]^{n+1-m-k}$ where $\tau_\theta$ is a unique element of $\Xi_m^{m+k}((R/r)^{\frac{m+k-1}{n}},K,{\bf{\tilde{h}}})$, where $\tilde{h}_i=h_ir^{\frac{i}{n}}$ if $|h_i|r^{\frac{i}{n}}\ge K^{-1}R^{-\frac{m}{n}}r^{\frac{m}{n}}$ and $\tilde{h}_i=R^{-\frac{m}{n}}r^{\frac{m}{n}}$ if $|h_i|r^{\frac{i}{n}}<K^{-1}R^{-\frac{m}{n}}r^{\frac{m}{n}}$. Thus $A$ maps $\cup_{I(\theta)\subset J}\theta$  into a cylindrical neighborhood of $\Gamma_m^{m+k}((R/r)^{\frac{m+k-1}{n}},K,{\bf{\tilde{h}}})$. 
The rescaled integral we wish to bound is 
\begin{equation*}
\int_{\R^{n+1}}|\sum_{I(\theta)\subset J_{m+1}}f_\theta(A^ty)|^{\tilde{p}_{k-1}}\tilde{W}_{J,m+k,r^{-\d_k}}^{n+1,d_k}(x-A^ty)dy.
\end{equation*}
For each element $y\in\R^{n+1}$, write $y=(y_{m+k},y')$ where $y_{m+k}\in\R^{m+k}$ and $y'\in\R^{n+1-m-k}$. Rewrite the rescaled integral above as 
\begin{equation}\label{rescaled}
\int_{\R^{n+1-m-k}}\int_{\R^{m+k}}|\sum_{I(\theta)\subset J}f_\theta(A^t(y_{m+k},y'))|^{\tilde{p}_{k-1}}W_{J,{m+k},r^{-\d_k}}^{n+1,d_k}(x-A^t(y_{m+k},y'))dy_{m+k}dy'.
\end{equation}
For each $\theta$, we have by Fourier inversion and Fubini's theorem that
\begin{align*}    f_\theta(A^ty)&=\int_{\R^{n+1}}\widehat{f}_\theta(\xi) e(y\cdot A\xi)d\xi = |\det A|^{-1}\int_{\R^{n+1}}\widehat{f}_\theta(A^{-1}\xi) e(y\cdot \xi)d\xi  \\
    &= \int_{\tau_\theta} \left[ |\det A|^{-1}\int_{\R^{n+1-k}}\widehat{f}_\theta(A^{-1}(\xi_{m+k},\xi')) e(y'\cdot \xi')d\xi' \right] e(y_{m+k}\cdot\xi_{m+k})d\xi_{m+k} 
\end{align*}
where $\tau_\theta$ is the associated element of $\Xi_m^{m+k}((R/r)^{\frac{m+k-1}{n}},K,{\bf{\tilde{h}}})$. Using the displayed equation, for each $y'$, we may view $f_\theta(A^t(y_{m+k},y'))$ as a function with Fourier transform supported in $\tau_\theta$.

Recall that for each $y'\in \R^{n+1-m-k}$ and $x\in\R^{n+1}$, the weight function $\tilde{W}_{J,m+k,r^{-\d_k}}^{n+1,d_k}(x-A^t(y_{m+k},y'))$ is localized to a union of translates of $B^*$, which has radius $R^{\frac{m}{n}}r^{-\frac{m}{n}}r^{\d_k}$. Since $2\le \tilde{p}_{k-1}\le p_{k-1}$, we may use 
\[ S_K^{m+k-1,m}(1,\min(r^{\d_k(m+k-1)/n},(\rho/r)^{(m+k-1)/n}),(R/r)^{(m+k-1)/n})\lesssim_\d r^\d \] 
to bound \eqref{rescaled} by 
\[ C_\d r^\d \int_{\R^{n+1-m-k}}\int_{\R^{m+k}}|\sum_{\substack{I_k\subset J}}|\sum_{I(\theta)\subset I_k} f_\theta(A^t(y_{m+k},y'))|^2|^{\frac{\tilde{p}_{k-1}}{2}}\tilde{W}_{J,{m+k},r^{-\d_k}}^{n+1,d_k}(x-A^t(y_{m+k},y'))dy_{m+k}dy',\]
where $I_k$ varies over intervals in $\mc{J}(\max(\rho^{-1/n},r^{-(1+\d_k)/n}))$ contained in $J$. 
Undoing the change of variables finishes the proof. 

\end{proof}

\begin{proof}[Proof of Lemma \ref{multilem3}] Recall that $I_k$ denotes intervals in $\mc{J}\big(\max(\rho^{-\frac{1}{n}},r^{-\frac{(1+\d_k)}{n}})\big)$. Our goal is to show that
\begin{align} \label{goal3}
\int_{\R^{n+1}}|\sum_{J}|\sum_{I_k\subset J}|\sum_{\substack{I(\theta)\subset I_{k}
}}f_\theta|^2|^{\frac{\tilde{p}_{k-1}}{2}} *\tilde{W}_{J,{m+k},r^{-\d_k}}^{n+1,d_k}|^{\frac{p}{\tilde{p}_{k-1}}}W\lesssim_\d r^\d \int_{\R^{n+1}}|\sum_{I_{k}}|\sum_{\substack{I(\theta)\subset I_{k}
}}f_\theta|^2|^{\frac{p}{2}}W. 
\end{align}
Our strategy is to use an iteration that is similar to the proof of Lemma \ref{multilem1}. Begin by understanding the Fourier support of each $|\sum_{I_k\subset J} |\sum_{\substack{I(\theta)\subset I_{k}
}}f_\theta|^2|^{\frac{\tilde{p}_{k-1}}{2}} *\tilde{W}_{J,{m+k},r^{-\d_k}}^{n+1,d_k}$. For each $J$ with initial point $a$, $\sum_{I_k\subset J}|\sum_{I(\theta)\subset I_k  } f_\theta|^2$ is Fourier supported in 
\[ \{\sum_{i=0}^n\lambda_i\phi_n^{(i)}(a):|\lambda_i|\le R^{-\frac{m}{n}}r^{\frac{m-i}{n}} \}. \]
Since $\tilde{p}_{k-1}$ is even, $|\sum_{I_k\subset J_{m+1}}|\sum_{\substack{I(\theta)\subset I_{k}
}}f_\theta|^2|^{\frac{\tilde{p}_{k-1}}{2}}$ is Fourier supported in 
\[ \{\sum_{i=0}^n\lambda_i\phi_n^{(i)}(a):|\lambda_i|\le (\tilde{p}_{k-1}/2)R^{-\frac{m}{n}}r^{\frac{m-i}{n}} \}. \]
The intersection of this set with the Fourier support of $\tilde{W}_{J,{m+k},r^{-\d_k}}^{n+1,d_k}$ is contained in 
\[ \{\sum_{i=0}^n\lambda_i\phi_n^{(i)}(a):|\lambda_i|\lesssim R^{-\frac{m}{n}}r^{\frac{m-i}{n}}(r^{\d_k})^{m+k-i} \}. \]
\noindent\fbox{$2<\frac{p}{\tilde{p}_{k-1}}$ case:} Use the auxiliary functions $\sum_{r^{-\frac{1}{n}}\le \sigma_k\le r^{-\d_k}}\eta^k_{J,\sigma_k}$ to decompose the above Fourier support and suppose that the left hand side of \eqref{goal3} is bounded by 
\begin{align}\label{case1it} 
C(\log r)^C \int_{\R^{n+1}}|\sum_{J}|\sum_{I_k\subset J}|\sum_{\substack{I(\theta)\subset I_{k}
}}f_\theta|^2|^{\frac{\tilde{p}_{k-1}}{2}} *\tilde{W}_{J,{m+k},r^{-\d_k}}^{n+1,d_k}*\widecheck{\eta}_{J,\sigma_k}^k|^{\frac{p}{\tilde{p}_{k-1}}}W. \end{align}
If $\sigma_k=r^{-\frac{1}{n}}$, then $\eta_{J,r^{-\frac{1}{n}}}^k$ is a smooth bump function supported in $\{\sum_{i=0}^{m+k-1}\lambda_i\phi_{m+k-1}^{(i)}(a):|\lambda_i|\lesssim R^{-\frac{m}{n}}r^{-\frac{k}{n}} \}\times[-1,1]^{n+1-m-k}$, which is contained in the Fourier support of each $\tilde{W}_{J,{m+k},r^{-\d_k}}^{n+1,d_k}$. If \[B_k=\{(R^{\frac{m}{n}}r^{\frac{k}{n}}x_0)^2+\cdots+(R^{\frac{m}{n}}r^{\frac{k}{n}}x_{m+k-1})^2+x_{m+k}^2+\cdots+x_n^2\le 1\},\]
then the weight function $\tilde{W}_{B_k}^{n+1,d_{k+1}}$ satisfies $|\tilde{W}_{J,{m+k},r^{-\d_k}}^{n+1,d_k}*\widecheck{\eta}_{J,\sigma}^k|\lesssim \tilde{W}_{B_k}^{n+1,d_{k+1}}$ for each $J$. Finally, by Cauchy-Schwarz, \eqref{case1it} is bounded by 
\[ C(\log r)^C \int_{\R^{n+1}}|\sum_{J}|\sum_{I_k\subset J}|\sum_{\substack{I(\theta)\subset I_{k}
}}f_\theta|^2|^{\frac{\tilde{p}_{k-1}}{2}}|^{\frac{p}{\tilde{p}_{k-1}}}\tilde{W}_{B_k}^{n+1,d_{k+1}}*W.\]
Since $R^{\frac{m}{n}}\rho\ge R^{\frac{m}{n}}r^{\frac{k}{n}}$, we have $\tilde{W}_{B_k}^{n+1,d_{k+1}}*W\lesssim W$. It remains to note that since $\|\cdot\|_{\ell^{\tilde{p}_{k-1}/2}}\le\|\cdot\|_{\ell^1}$, the previous displayed expression is bounded by 
\[ C(\log r)^C \int_{\R^{n+1}}|\sum_{I_k} |\sum_{\substack{I(\theta)\subset I_{k}}} f_\theta|^2|^{\frac{p}{2}}W, \]
which proves the lemma and halts the iteration. 

Next, consider the case that \eqref{case1it} holds with $\sigma_k>r^{-\frac{1}{n}}$. As in the proof of Lemma \ref{multilem1}, there are $\sim \sigma_k^{-1}$ many neighboring $J$ for which the Fourier support of 
\[ |\sum_{I_k\subset J}|\sum_{\substack{I(\theta)\subset I_{k}
}}f_\theta|^2|^{\frac{\tilde{p}_{k-1}}{2}} *\tilde{W}_{J_{m+1},r^{-\d_k}}^{n+1,d_k}*\widecheck{\eta}_{J_{m+1},\sigma_k}^k \] 
is contained in an element $\tau\in\Xi_{m+k}^{n+1}(\sigma_k^nr)$ dilated by a factor of $R^{-\frac{m}{n}}r^{\frac{m}{n}}\sigma_k^{m+k}$. If we dilate frequency side by $R^{\frac{m}{n}}r^{-\frac{m}{n}}\sigma_k^{-(m+k)}$, then the weight function $W$ gets dilated by a factor of $R^{-\frac{m}{n}}r^{\frac{m}{n}}\sigma_k^{m+k}$, and so is localized to a ball of radius $r^{\frac{m}{n}}\sigma_k^{m+k}\rho\ge \sigma_k^nr $. Using the hypothesis that $\C_{m+k-1}^{n+1}(\sigma_k^n r)\lesssim_\d (\sigma_k^n r)^\d$ and Lemma \ref{loclem}, \eqref{case1it} is bounded above by 
\[  C_{\d}r^{\d} \int_{\R^{n+1}}|\sum_{J_k}|\sum_{J\subset J_k}|\sum_{I_k\subset J}|\sum_{\substack{I(\theta)\subset I_{k}
}}f_\theta|^2|^{\frac{\tilde{p}_{k-1}}{2}} *\tilde{W}_{J,{m+k},r^{-\d_k}}^{n+1,d_k}*\widecheck{\eta}_{J,\sigma_k}^k|^2|^{\frac{p}{2\tilde{p}_{k-1}}}W   , \]
where $J_k$ are intervals in $\mc{J}(\sigma_k^{-1}r^{-\frac{1}{n}})$. There are $L^1$-normalized weight functions $\tilde{W}_{J_k,\sigma_k}^{n+1,d_k-1}$ with Fourier transform supported in 
\[ \{\sum_{i=0}^{m+k-1}\lambda_i\phi_{m+k-1}^{(i)}(a):|\lambda_i|\lesssim R^{-\frac{m}{n}}r^{\frac{m-i}{n}}\sigma_k^{m+k-i} \}\times[-1,1]^{n+1-m-k} \]
(where $a$ is the initial point of $J_k$) which satisfy $|\tilde{W}_{J,{m+k},r^{-\d_k}}^{n+1,d_k}*\widecheck{\eta}_{J_{m+1},\sigma_k}^k|\lesssim \tilde{W}_{J_k,\sigma_k}^{n+1,d_k-1}$. Thus \eqref{case1it} is bounded by 
\[  C_{\d}r^{\d} \int_{\R^{n+1}}|\sum_{J_k}|\sum_{J\subset J_k}|\sum_{I_k\subset J}|\sum_{\substack{I(\theta)\subset I_{k}
}}f_\theta|^2|^{\frac{\tilde{p}_{k-1}}{2}}*\tilde{W}_{J_{k},\sigma_k}^{n+1,d_k-1} |^2|^{\frac{p}{2\tilde{p}_{k-1}}}W . \]
The Fourier support of each summand $|\sum_{J\subset J_k}|\sum_{I_k\subset J}|\sum_{\substack{I(\theta)\subset I_{k}
}}f_\theta|^2|^{\frac{\tilde{p}_{k-1}}{2}} *\tilde{W}_{J_{k},\sigma_k}^{n+1,d_k-1}|^2$ is contained in 
\[ \{\sum_{i=0}^{n}\lambda_i\phi_{n}^{(i)}(a):|\lambda_i|\lesssim R^{-\frac{m}{n}}r^{\frac{m-i}{n}}\sigma_k^{m+k-i} \},\]
where $a$ is the initial point of $J_k$. If $1\le \frac{p}{2\tilde{p}_{k-1}}\le 2$, then use $\eta_{J_k}^{n-m}=\sum_{\sigma_k^{-1}r^{-\frac{1}{n}}\le \sigma\le 1}\eta_{J_k,\sigma}^{n-m}$ to decompose the Fourier support into subsets
\[ \{\sum_{i=0}^{n}\lambda_i\phi_{n}^{(i)}(a):|\lambda_i|\lesssim R^{-\frac{m}{n}}r^{\frac{m-i}{n}}\sigma_k^{m+k-i}\sigma^{n-i} \} \]
and $\{\sum_{i=0}^{n}\lambda_i\phi_{n}^{(i)}(a):|\lambda_i|\lesssim R^{-\frac{m}{n}}r^{\frac{m-n}{n}}\sigma_k^{m+k-n}\}$. Then since $\C_{n-1}^{n+1}(\sigma_k^n\sigma^nr)\lesssim_\d (\sigma_k^n\sigma^nr)^\d$, \eqref{case1it} is bounded by 
\[ C_{\d}r^{\d}\sum_{J_{n-m}}\int_{\R^{n+1}}|\sum_{J_k\subset J_{n-m}}|\sum_{J\subset J_k}|\sum_{I_k\subset J}|\sum_{\substack{I(\theta)\subset I_{k}
}}f_\theta|^2|^{\frac{\tilde{p}_{k-1}}{2}}*\tilde{W}_{J_{k},\sigma_k}^{n+1,d_k-1} |^2*\widecheck{\eta}_{J_k,\sigma}^{n-m}|^{\frac{p}{2\tilde{p}_{k-1}}}W   \]
where the initial sum is over $J_{n-m}\in\mc{J}(\sigma^{-1}\sigma_k^{-1}r^{-\frac{1}{n}})$. It remains to observe that there are weight functions $\tilde{W}_{J_{n-m}}^{n+1,d_{k+1}}$ satisfying $\tilde{W}_{J_k,\sigma_k}^{n+1,d_k-1}*|\widecheck{\eta}_{J_k,\sigma}^{n-m}|\lesssim \tilde{W}_{J_{n-m}}^{n+1,d_{k+1}}$ for all $J_k\subset J_{n-m}$ and 
\begin{align*}
\sum_{J_{n-m}}\int_{\R^{n+1}}|\sum_{J_k\subset J_{n-m}}|&\sum_{J\subset J_k}|\sum_{I_k\subset J}|\sum_{\substack{I(\theta)\subset I_{k}
}}f_\theta|^2|^{\frac{\tilde{p}_{k-1}}{2}}*\tilde{W}_{J_{k},\sigma_k}^{n+1,d_k-1} |^2*\widecheck{\eta}_{J_k,\sigma}^{n-m}|^{\frac{p}{2\tilde{p}_{k-1}}}W  \\
    &\le \sum_{J_{n-m}}\int_{\R^{n+1}}|\sum_{J_k\subset J_{n-m}}|\sum_{J\subset J_k}|\sum_{I_k\subset J}|\sum_{\substack{I(\theta)\subset I_{k}}} f_\theta|^2|^{\frac{\tilde{p}_{k-1}}{2}} |^2*\tilde{W}_{J_{k},\sigma_k}^{n+1,d_k-1}*|\widecheck{\eta}_{J_k,\sigma}^{n-m}||^{\frac{p}{2\tilde{p}_{k-1}}}W \\
    &\lesssim \sum_{J_{n-m}}\int_{\R^{n+1}}|\sum_{J_k\subset J_{n-m}}|\sum_{J\subset J_k}|\sum_{I_k\subset J}|\sum_{\substack{I(\theta)\subset I_{k}}} f_\theta|^2|^{\frac{\tilde{p}_{k-1}}{2}} |^2*\tilde{W}_{J_{n-m}}^{n+1,d_{k+1}}|^{\frac{p}{2\tilde{p}_{k-1}}}W\\
    &\lesssim \sum_{J_{n-m}}\int_{\R^{n+1}}|\sum_{J_k\subset J_{n-m}}|\sum_{J\subset J_k}|\sum_{I_k\subset J}|\sum_{\substack{I(\theta)\subset I_{k}}} f_\theta|^2|^{\frac{\tilde{p}_{k-1}}{2}} |^2|^{\frac{p}{2\tilde{p}_{k-1}}}(\tilde{W}_{J_{n-m}}^{n+1,d_{k+1}}*W)\\
    &\lesssim \sum_{J_{n-m}}\int_{\R^{n+1}}|\sum_{J_k\subset J_{n-m}}|\sum_{J\subset J_k}|\sum_{I_k\subset J}|\sum_{\substack{I(\theta)\subset I_{k}}} f_\theta|^2|^{\frac{\tilde{p}_{k-1}}{2}} |^2|^{\frac{p}{2\tilde{p}_{k-1}}}W\\
    &\lesssim \int_{\R^{n+1}}|\sum_{I_k}|\sum_{\substack{I(\theta)\subset I_{k}}} f_\theta|^2|^{\frac{p}{2}}W,
\end{align*}
which proves the lemma. 

Now suppose that $\frac{p}{2\tilde{p}_{k-1}}> 2$. Let $p_{n-m-k_2}$ be the minimal exponent such that $2\le\frac{p}{2\tilde{p}_{k-1}}\le p_{n-m-k_2}$ and $n-m\ge k_2> k$. Use $\eta^{k_2}_{J_k,\sigma_{k_2}}$ to decompose the Fourier supports of $|\sum_{J\subset J_k}|\sum_{I_k\subset J}|\sum_{\substack{I(\theta)\subset I_{k}
}}f_\theta|^2|^{\frac{\tilde{p}_{k-1}}{2}} *\tilde{W}_{J_k,\sigma_k}^{n+1,d_k-1}|^2$ into 
\begin{align*} 
\Big\{\sum_{i=0}^n\lambda_i\phi_n^{(i)}(a):&|\lambda_i|\le R^{-\frac{m}{n}}r^{\frac{m-i}{n}}\sigma_k^{m+k-i}\sigma^{m+k_2-i}\quad\forall i,\\
&\qquad |\lambda_j|\ge 2^{-m-k_2+j}R^{-\frac{m}{n}}r^{\frac{m-j}{n}}\sigma_k^{m+k-j}\sigma^{m+k_2-j}\quad\text{for some }j \text{ with } 0\le j\le m+k_2-1 \Big\} \end{align*}
and $\{\sum_{i=0}^n\lambda_i\phi_n^{(i)}(a):|\lambda_i|\le R^{-\frac{m}{n}}r^{-\frac{k_2}{n}} \sigma_k^{k-k_2}\quad\forall i\}$. Using that $\C_{m+l}^{n+1}(\sigma_{k_2}^n\sigma_k^nr)\lesssim_\d (\sigma_{k_2}^n\sigma_k^nr)^\d$, \eqref{case1it} is then bounded by 
\begin{align*} 
C_{\d}r^{\d} \int_{\R^{n+1}}&|\sum_{J_{k_2}}|\sum_{J_k\subset J_{k_2}}|\sum_{J\subset J_k}|\sum_{I_k\subset J}|\sum_{\substack{I(\theta)\subset I_{k}}} f_\theta|^2|^{\frac{\tilde{p}_{k-1}}{2}} *\tilde{W}_{J_{k},\sigma_k}^{n+1,d_k-1}|^2*\widecheck{\eta}_{J_k,\sigma_{k_2}}^{k_2}|^2|^{\frac{p}{2^2\tilde{p}_{k-1}}}W\\
&\lesssim C_{\d}r^{\d} \int_{\R^{n+1}}|\sum_{J_{k_2}}|\sum_{J_k\subset J_{k_2}}|\sum_{J\subset J_k}|\sum_{I_k\subset J}|\sum_{\substack{I(\theta)\subset I_{k}}} f_\theta|^2|^{\frac{\tilde{p}_{k-1}}{2}} *\tilde{W}_{J_{k},\sigma_k}^{n+1,d_k-1}|^2*\tilde{W}_{J_{k_2},\sigma_{k_2}}^{n+1,d_{k_2}+1}|^2|^{\frac{p}{2^2\tilde{p}_{k-1}}}W
\end{align*}
where the initial sum is over $J_{k_2}\in\mc{J}(\sigma_k^{-1}\sigma_{k_2}^{-1}r^{-\frac{1}{n}})$. Iterate this process $j-2$ more times until either the lemma has been proved (corresponding to a low frequency case) or the outermost exponent $\frac{p}{2^j\tilde{p}_{k-1}}$ satisfies $1\le \frac{p}{2^j\tilde{p}_{k-1}}\le 2$  and, after repeating the case $1\le \frac{p}{2\tilde{p}_{k-1}}\le 2$ argument, \eqref{case1it} is bounded by 
\begin{equation}\label{ben} 
C_{\d_k}r^{C\d_k} \sum_{J_{k_{j+1}}}\int_{\R^{n+1}}|\sum_{J_{k_j}\subset J_{k_{j+1}}}|F_{J_{k_j}}|^2*\tilde{W}_{J_{k_{j+1}},\sigma_{j+1}}^{n+1,d_{k_{j+1}+1}}|^{\frac{p}{2^j\tilde{p}_{k-1}}}W\end{equation}
where $k=k_1$, $\sigma_{k_i}\in[(\sigma_{k_1}\cdots\sigma_{k_{i-1}})^{-1}r^{-\frac{1}{n}},1]$, $J_{k_i}\in\mc{J}((\sigma_{k_1}\cdots\sigma_{k_{i-1}})^{-1}r^{-\frac{1}{n}})$, and 
\[F_{J_i}=\sum_{J_{i-1}\subset J_i}|F_{J_{i-1}}|^2*W_{J_i,\sigma_i}^{n+1,d_{k_i}+1}\qquad\text{with}\qquad F_{J_{k_1}}=\sum_{J\subset J_{k_1}} |\sum_{I_{k_1}\subset J}|\sum_{\substack{I(\theta)\subset I_{k_1}}} f_\theta|^2|^{\frac{\tilde{p}_{k_1-1}}{2}} *\tilde{W}_{J_{k_1},\sigma_{k_1}}^{n+1,d_{k_1}-1}. \] 
The weight functions satisfy the property that if $J_{k_i}\subset J_{k_{i+l}}$, then $\tilde{W}_{J_{k_i},\sigma_i}^{n+1,d_{k_i}-1}*\tilde{W}_{J_{k_{i+l}},\sigma_{i+l}}^{n+1,d_{k_{i+l}}+1}\sim \tilde{W}_{J_{k_{i+l}},\sigma_{i+l}}^{n+1,d_{k_{i+l}}}$. Therefore, by applying Cauchy-Schwarz pointwise in the convolutions, we have that
\begin{align*}
F_{J_{k_2}}=\sum_{J_{k_1}\subset J_{k_2}}|F_{J_{k_1}}|^2*\tilde{W}_{J_{k_2},\sigma_{k_2}}^{n+1,d_{k_2}+1} &\lesssim \sum_{J_{k_1}\subset J_{k_2}}|\sum_{J\subset J_{k_1}} |\sum_{I_{k_1}\subset J}|\sum_{\substack{I(\theta)\subset I_{k_1}}} f_\theta|^2|^{\frac{\tilde{p}_{k_1-1}}{2}} |^2*\tilde{W}_{J_{k_2},\sigma_{k_2}}^{n+1,d_{k_2}}\\
    &\lesssim |\sum_{I_k\subset J_{k_2}}|\sum_{\substack{I(\theta)\subset I_{k}}} f_\theta|^2|^{\tilde{p}_{k-1}} *\tilde{W}_{J_{k_2},\sigma_{k_2}}^{n+1,d_{k_2}}
\end{align*}
and by induction, 
\begin{align*}
\sum_{ J_{k_{i-1}}\subset J_{k_i} }|F_{J_{k_{i-1}}}|^2*\tilde{W}_{J_{k_i},\sigma_{k_i}}^{n+1,d_{k_i}+1} &\lesssim \sum_{ J_{k_{i-1}}\subset J_{k_i}}||\sum_{I_k\subset J_{k_{i-1}}} |\sum_{\substack{I(\theta)\subset I_{k}}} f_\theta|^2|^{2^{i-3}\tilde{p}_{k-1}} *\tilde{W}_{J_{k_{i-1}},\sigma_{i-1}}^{n+1,d_{k_{i-1}}}|^2*\tilde{W}_{J_{k_i},\sigma_{k_i}}^{n+1,d_{k_i}+1}\\
    &\lesssim \sum_{ J_{k_{i-1}}\subset J_{k_i}}|\sum_{I_k\subset J_{k_{i-1}}} |\sum_{\substack{I(\theta)\subset I_{k}}} f_\theta|^2|^{2^{i-2}\tilde{p}_{k-1}}*\tilde{W}_{J_{k_i},\sigma_{k_i}}^{n+1,d_{k_i}}\\
    &\lesssim |\sum_{I_k\subset J_{k_{i}}} |\sum_{\substack{I(\theta)\subset I_{k}}} f_\theta|^2|^{2^{i-2}\tilde{p}_{k-1}}*\tilde{W}_{J_{k_i},\sigma_{k_i}}^{n+1,d_{k_i}}.   
\end{align*}
Conclude that \eqref{ben} is therefore bounded by 
\begin{align*}
C_{\d_k}r^{C\d_k} \sum_{J_{k_{j+1}}}\int_{\R^{n+1}}&||\sum_{I_k\subset J_{k_{j+1}}}|\sum_{I(\theta)\subset I_k}|f_\theta|^2|^{2^{j-1}\tilde{p}_{k-1}}*\tilde{W}_{J_{k_{j+1}},\sigma_{j+1}}^{n+1,d_{j+1}}|^{\frac{p}{2^j\tilde{p}_{k-1}}}W\\
    &\lesssim C_{\d_k}r^{C\d_k} \sum_{J_{k_{j+1}}}\int_{\R^{n+1}}|\sum_{I_k\subset J_{k_{j+1}}}|\sum_{I(\theta)\subset I_k}|f_\theta|^2|^{\frac{p}{2}}W\\
    &\lesssim C_{\d_k}r^{C\d_k} \int_{\R^{n+1}}|\sum_{I_k}|\sum_{I(\theta)\subset I_k}|f_\theta|^2|^{\frac{p}{2}}W,
\end{align*}
as desired. 

\vspace{2mm}
\noindent\fbox{$1\le \frac{p}{\tilde{p}_{k-1}}\le 2$ case:} 
This final case is analogous to the beginning of the $2< \frac{p}{2\tilde{p}_{k-1}}$ case. We use auxiliary functions $\sum_{r^{-\frac{1}{n}}\le\sigma\le r^{-\d_k}}\eta_{J,\sigma_k}^{k}$ to decompose the Fourier supports. By using a cylindrical version of the bound $\C_{m+k-2}^{m+k}(\sigma_k^{m+k-1}r^{\frac{m+k-1}{n}})\lesssim_\d (\sigma_k^{m+k-1}r^{\frac{m+k-1}{n}})^\d$, the left hand side of \eqref{goal3} is bounded by
\begin{align*}
    C_\d r^\d \sum_{J_k\in\mc{J}(\sigma_k^{-1}r^{-\frac{1}{n}})} \int_{\R^{n+1}}|\sum_{J\subset J_k}|\sum_{I_k\subset J}|\sum_{\substack{I(\theta)\subset I_{k}
}}f_\theta|^2|^{\frac{\tilde{p}_{k-1}}{2}} *\tilde{W}_{J,{m+k},r^{-\d_k}}^{n+1,d_k}*\widecheck{\eta}_{J,\sigma_k}^k|^{\frac{p}{\tilde{p}_{k-1}}}W.
\end{align*}
Since $\tilde{W}_{J,m+k,r^{-\d_k}}^{n+1,d_k-1}*|\widecheck{\eta}_{J,\sigma_k}^k|\lesssim \tilde{W}_{J_k,\sigma_k}^{n+1,d_{k+1}}$, where $\tilde{W}_{J_k,\sigma_k}^{n+1,d_{k+1}}$ is $L^1$-normalized, and $\tilde{W}_{J_k,\sigma_k}^{n+1,d_{k+1}}*W\lesssim W$, the previous displayed math is bounded by 
\[ C_\d r^\d \sum_{J_k\in\mc{J}(\sigma_k^{-1}r^{-\frac{1}{n}})} \int_{\R^{n+1}}|\sum_{J\subset J_k}|\sum_{I_k\subset J}|\sum_{\substack{I(\theta)\subset I_{k}
}}f_\theta|^2|^{\frac{\tilde{p}_{k-1}}{2}} |^{\frac{p}{\tilde{p}_{k-1}}}W.\]
Finally, observing using that $\|\cdot\|_{\ell^{\tilde{p}_{k-1}/2}}\le \|\cdot\|_{\ell^1}$ and that $\|\cdot\|_{\ell^{p/2}}\le \|\cdot\|_{\ell^1}$ finishes the proof. 

\end{proof}

The argument used to prove Lemma \ref{multilem3} actually justifies a more general proposition, which we state here for future use. Let $\tau\in\Xi_m^{n+1}(r)$ and consider $r^{-\frac{1}{n}}\le \sigma\le 1$. For each $ k\ge 1$ and $d\ge 1$, let $\tilde{W}^k_{\tau,\sigma}$ be the $L^\infty$-normalized weight function with decay factor $d$ (this dependence is suppressed because it is not important) which is centered at the origin and Fourier supported in 
\[ \{\sum_{i=0}^{k-1}\lambda_i\phi_{k-1}^{(i)}(a):|\lambda_i|\le \sigma^{k-i} r^{\frac{-i}{n}} \}\times[-2,2]^{n+1-k} \]
where $a$ is the initial point of $I(\tau)$. 

\begin{lemma}\label{multilem3pf} Suppose that $\C^{n'+1}_{m'}(r)\lesssim_\e r^\e$ for all $0\le m'\le n'-1<n-1$ and $\C^{n+1}_{m'}(r)\lesssim_\e r^\e$ for all $m<m'\le n-1$. Let $p_{n,m}(k)=p_{n-k}$ for $k=1,\ldots,m+1$ and $p_{n,m}(k)=p_{n+1-k}$ for $k=m+2,\ldots,n$. For any $2\le p\le p_{n,m}(k)$ and any $r^{-\frac{1}{n}}\le \sigma\le 1$, 
\[ \int_{\R^{n+1}}|\sum_{\tau\in\Xi_m^{n+1}(r)} G_\tau*\tilde{W}^{k}_{\tau,\sigma}|^{p}\lesssim_\d r^\d \int_{\R^{n+1}}|\sum_{\tau\in\Xi_m^{n+1}(r)} G_\tau |^{p}\]    
whenever $G_{\tau}:\R^n\to[0,\infty)$ has Fourier transform supported in $\tau-\tau$.  
\end{lemma}

\subsection{Inductive proof of Proposition \ref{SnmmKbd} \label{SnmmKbdsec}}

\begin{proof}[Inductive proof of Proposition \ref{SnmmKbd} ]
We will show that for any $\e>0$, there exists $\tilde{C}_\e$ such that $S_{K}^{n,m}(r,\rho,R)\le \tilde{C}_\e \rho^\e$ for all $1\le r\le\rho\le R$. We are permitted to choose $K$ to be a constant depending on $\e$. 

\noindent\fbox{$R\le C_K$} By Cauchy-Schwarz, if $R\le C_K$, then $S_{K}^{n,m}(r,\rho,R)\lesssim_K 1$. 
\vspace{2mm}
\newline\noindent\fbox{Inductive hypothesis} Suppose that for any $1\le r\le \rho\le R$ satisfying $R<R_0/K$, we have $S_{K}^{n,m}(r,\rho,R)\le \tilde{C}_\e \rho^\e$. We will show that for any $1\le r\le \rho\le R$ satisfying $R<R_0$, $S_{K}^{n,m}(r,\rho,R)\le \tilde{C}_\e \rho^\e$. It suffices to consider $\rho\ge C_K$. 

If $r\le K$, then by Proposition \ref{initscale}, $S_{K}^{n,m}(r,\rho,R)\le C_{\d_1} K^{\d_1}  S_{K}^{n,m}(K,\rho,R)$, where we may choose any $\d_1>0$. Assume from now on that $r\ge K$. By Proposition \ref{conemulti}, 
\[S_{K}^{n,m}(r,\rho,R)\le C_{\d_1}K^{\d_1}C_{\d_2}\left[r^{\d_2}S_{K}^{n,m}(1,\rho/r,R/r)+\sum_{k=1}^{n-m} r^{\d_2^{n-k+1}}\Big[S_{K}^{n,m}(\min(\rho,r^{1+\d_2^{n-k}}),\rho,R) \Big]\right],\]
where we may choose any $\d_2>0$. We will perform an iteration that depends on which term dominates the right hand side. 

Suppose that 
\[ S_{K}^{n,m}(r,R)\le C_{\d_1}K^{\d_1}nC_{\d_2}r^{\d_2}S_{K}^{n,m}(1,\rho/r,R/r). \]
Note that by assumption, $\frac{R}{r}< \frac{R_0}{K}$, so the inductive hypothesis applies and we have
\begin{equation} 
S_{K}^{n,m}(r,R)\le \big(nC_{\d_1}K^{\d_1}C_{\d_2}r^{\d_2}r^{-{\e}}\big)\tilde{C}_\e \rho^\e ,
\end{equation}
which halts our iteration.

The other case is that for some $1\le k\le n-m$,
\[ S_{K}^{n,m}(r,\rho,R)\le C_{\d_1}K^{\d_1}nC_{\d_2}r^{\d_2^{n-k+1}}S_{K}^{n,m}(\min(\rho,r^{1+\d_2^{n-k}}),\rho,R).\] 
If $r^{1+\d_2^{n-k}}\ge \rho$, then $S_{K}^{n,m}(\rho,\rho,R)=1$ and we have
\begin{equation}
S_{K}^{n,m}(r,\rho,R)\le C_{\d_1}K^{\d_1}nC_{\d_2}r^{\d_2^{n-k+1}}\le \big(C_{\d_1}K^{\d_1}C_{\d_2}r^{\d_2^{n-k+1}}\rho^{-\e}\big)\tilde{C}_\e \rho^\e \end{equation}
and again, our iteration halts. Finally, suppose that $r^{1+\d_0^{n-k}}<\rho$. Write $\eta_1=\d_2^{n-k}$ so that we have 
\[ S_{K}^{n,m}(r,\rho,R)\le C_{\d_1}K^{\d_1}nC_{\d_2}r^{\d_2\eta_1}S_{K}^{n,m}(r^{(1+\eta_1)},\rho,R). \]
Now proceed to the second step of the iteration. 

The input for step $k$ of the iteration (supposing that it has not halted at a previous step) is the inequality
\[ S_{K}^{n,m}(r,\rho,R)\le C_{\d_1}K^{\d_1}(nC_{\d_2})^{k-1}r^{\d_2[\eta_1+\cdots+\eta_{k-1}\prod_{i=1}^{k-2}(1+\eta_i)]} S_{K}^{n,m}(r^{\prod_{i=1}^{k-1}(1+\eta_i)},\rho,R). \]
Apply Proposition \ref{coneinduct} again to bound $S_{K}^{n,m}(r^{\prod_{i=1}^{k-1}(1+\eta_i)},\rho,R)$ on the right hand side above by  
\begin{align*} 
C_{\d_2}&\left[r^{\d_2\prod_{i=1}^{k-1}(1+\eta_i)}S_{K}^{n,m}(1,\rho/r^{\prod_{i=1}^{k-1}(1+\eta_i)},R/r^{\prod_{i=1}^{k-1}(1+\eta_i)})\right.\\
    &\qquad\qquad\qquad\left.+\sum_{l=1}^{n-m} r^{\d_2^{n-l+1}\prod_{i=1}^{k-1}(1+\eta_i)}S_{K}^{n,m}(\min(\rho,r^{(1+\d_2^{n-l})\prod_{i=1}^{k-1}(1+\eta_i)}),\rho,R)\right]. \end{align*}
If the first term dominates, then by the inductive hypothesis, 
\[ S_{K}^{n,m}(r,\rho,R)\le (C_{\d_1}K^{\d_1}(nC_{\d_2})^{k-1}r^{\d_2[\eta_1+\cdots+\eta_{k-1}\prod_{i=1}^{k-2}(1+\eta_i)]} C_{\d_2}r^{\d_2\prod_{i=1}^{k-1}(1+\eta_i)}r^{-\e\prod_{i=1}^{k-1}(1+\eta_i)})\tilde{C}_\e \rho^\e ,\]
and the iteration halts. If one of the other terms dominates, then for some $1\le l\le n-m$, $\eta_k:=\d_2^{n-l}$ satisfies 
\[ S_{K}^{n,m}(r,\rho,R)\le C_{\d_1}K^{\d_1}(nC_{\d_2})^{k}r^{\d_2[\eta_1+\cdots+\eta_{k}\prod_{i=1}^{k-1}(1+\eta_i)]}S_{K}^{n,m}(\min(\rho,r^{\prod_{i=1}^k(1+\eta_i)}),\rho,R). \]
If $\rho\le r^{\prod_{i=1}^k(1+\eta_i)}$, then the iteration halts with the inequality
\[   S_{K}^{n,m}(r,\rho,R)\le C_{\d_1}K^{\d_1}(nC_{\d_2})^{k}r^{\d_2[\eta_1+\cdots+\eta_{k}\prod_{i=1}^{k-1}(1+\eta_i)]}. \]
If $\rho>r^{\prod_{i=1}^k(1+\eta_i)}$, then 
\[ S_{K}^{n,m}(r,\rho,R)\le C_{\d_1}K^{\d_1}(nC_{\d_2})^{k}r^{\d_2[\eta_1+\cdots+\eta_{k}\prod_{i=1}^{k-1}(1+\eta_i)]}S_{K}^{n,m}(r^{\prod_{i=1}^k(1+\eta_i)},\rho,R) \]
and we proceed to step $k+1$. 
\vspace{1mm}

\noindent\fbox{Analysis of halting criterion} There are two possible outcomes of the iteration. The first is that for some $k\ge 1$ with $r^{\prod_{i=1}^{k-1}(1+\eta_i)}\le \rho$, we have
\[ S_{K}^{n,m}(r,\rho,R)\le C_{\d_1}K^{\d_1}(nC_{\d_2})^{k}r^{\d_2[\eta_1+\cdots+\eta_{k-1} \prod_{i=1}^{k-2}(1+\eta_i)]}r^{\d_2\prod_{i=1}^{k-1}(1+\eta_i)-\e\prod_{i=1}^{k}(1+\eta_i)}\tilde{C}_\e \rho^\e. \]
It suffices to verify that 
\[ C_{\d_1}K^{\d_1}(nC_{\d_2})^{k}r^{(2\d_2-\e)\prod_{i=1}^{k-1}(1+\eta_i)}\le 1. \]
Choose $\d_1=\frac{\e}{8}$ and $\d_2=\frac{\e}{4}$ so that it suffices to verify that
\[ C_{\e/8}K^{\e/8}(nC_{\e/4})^k r^{-\e/2(1+(\e/4)^n)^k}\le 1 .\]
This follows from $K\le r$ and choosing $K=K(\e)$ large enough so that the above inequality holds for all $k\ge 1$. 

The second possible outcome of the iteration is that for $r^{\prod_{i=1}^{k-1}(1+\eta_i)}\le \rho\le r^{\prod_{i=1}^k(1+\eta_i)}$, we have the inequality
\[ S_{K}^{n,m}(r,\rho,R)\le C_{\e/8}K^{\e/8}(nC_{\e/4})^{k}r^{(\e/4)[\eta_1+\cdots+\eta_k\prod_{i=1}^{k-1}(1+\eta_i)]}, \]
where we inputted the definitions of $\d_1$ and $\d_2$. The exponent of $r$ satisfies $\eta_1+\cdots+\eta_k\prod_{i=1}^{k-1}(1+\eta_i)\le (1+\max_{1\le j\le k}\eta_j)\prod_{i=1}^{k-1}(1+\eta_i)$, so it suffices to verify that
\[C_{\e/8}K^{\e/8}(nC_{\e/4})^{k}\rho^{(\e/4)(1+\e/4)} \le \tilde{C}_\e \rho^\e.\] 
Using that $r^{(1+(\e/4)^n)^k}\le \rho$ and $K\le \rho$, it further suffices to verify that 
\[ C_{\e/8}(\log \rho)^{4^n\e^{-n}\log (nC_{\e/4})}\rho^{(5\e/8)(1+\e/4)} \le \tilde{C}_\e \rho^\e. \]
It is no loss of generality to assume that $\tilde{C}_\e$ and $K$ are large enough so that the above inequality is true, which concludes the proof. 

\end{proof}


\subsection{Bounding $S_{\sigma}^{n,m}(r,\rho,R)$ for $m>0$}

We use the boundedness of $S_K^{n,m}(r,\rho,R)$ proved in Proposition \ref{SnmmKbd} to bound $S_\sigma^{n,m}(r,\rho,R)$. Call the $(m+1)$-tuples $\sigma$ considered in the definition of $\Gamma_m^{n+1}(R,\sigma)$ \emph{admissible} and write $\mc{S}_R^{n,m}$ for the collection of admissible $\sigma$.

\begin{proposition}\label{L2conekaklocgen} Let $1\le n$. Then $S_{\sigma}^{n,n-1}(r,\rho,R)\lesssim_\e \rho^\e$ uniformly in $\sigma\in\mc{S}^{n,m}_R$. 
\end{proposition}

\begin{proof} This follows immediately from Theorem \ref{locL2thm}. 
\end{proof}

\begin{proposition}\label{multiconek} Assume that the hypotheses of Proposition \ref{coneinduct} hold. Also assume that $S^{n',m'}(r,\rho,R)\lesssim_\e \rho^\e$ for each $0\le m'\le n'-1$ when $n'<n$, and $S^{n,m'}(r,\rho,R)\lesssim_\e \rho^\e$ for each $m< m'\le n-1$. Then $S^{n,m}_{\sigma}(r,\rho,R)\lesssim_\e \rho^\e$ uniformly in $\sigma\in\mc{S}^{n,m}_R$.
\end{proposition}

If $m=0$, then it is easy to show that $S^{n,m}_\sigma(r,\rho,R)\le K^{O(1)}S_K^{n,m}(r,\rho,R)$, so Proposition \ref{multiconek} follows from Proposition \ref{SnmmKbd}. For the rest of this section, assume that $0<m$. 

We describe how to modify the iteration used to bound $S_{K}^{n,m}(r,\rho,R)$ in Proposition \ref{conemulti} to treat the $S_{\sigma}^{n,m}(r,\rho,R)$. To prove Proposition \ref{multiconek}, consider a Schwartz function $f=\sum_{\theta\in\Xi_{m}^{n+1}(R,\sigma)} f_\theta$ with $\supp\widehat{f}_\theta\subset \theta$, $W=W_B^{n+1,d}$ for a ball of radius $R^{\frac{m}{n}}\rho$, and $2\le p\le p_{n-m}$. 

The following lemma uses a cylindrical version of a lower-dimensional inequality to obtain an initial square function estimate into intervals of length $|J|=2^{-100n} \sigma_{m-1}^{-1}R^{-\frac{1}{n}}$. 
\begin{lemma} \label{lem1k} Assume that $0<m$. For each $1\le r\le \min(\rho,\sigma^nR)$ We have
\begin{align*}
\int_{\R^{n+1}}(\sum_{J\in\mc{J}(r^{-\frac{1}{n}})} |\sum_{\substack{\theta\in\Xi_{m}^{n+1}(R,\sigma)\\I(\theta)\subset J}}  f_\theta&|^2)^{\frac{p}{2}}W\lesssim S_{\sigma'}^{n-1,m-1}(1,\min(\rho,\sigma^nR)^{\frac{n-1}{n}},\min(\rho,\sigma^n R)^{\frac{n-1}{n}}) \\
    &\times \int_{\R^{n+1}}|\sum_{J\in\mc{J}(\max(\rho^{-\frac{1}{n}},\sigma^{-1}R^{-\frac{1}{n}}))} |\sum_{\substack{I(\theta)\subset J}\theta\in\Xi_{m}^{n+1}(R,\sigma)}f_\theta|^2|^{\frac{p}{2}}W. 
\end{align*}
where $\sigma'$ is an $m$-tuple with $\sigma_i'=\sigma_i$ for $i<m-1$ and $\sigma_{m-1}'=1$    
\end{lemma}

\begin{lemma} \label{lem2k} Assume the hypotheses of Proposition \ref{multiconek} and that $0<m$. For each $r^{-\frac{1}{n}}\le \sigma_{m-1}^{-1}R^{-\frac{1}{n}}$, let $J$ denote intervals in $\mc{J}( r^{-\frac{1}{n}})$ and let $J'$ denote intervals in $\mc{J}( \rho^{-\frac{1}{n}})$. Then for any $\d_0\in(0,1)$, 
\begin{align*} 
\int_{\R^{n+1}}|\sum_{J} |\sum_{\substack{I(\theta)\subset J\\\theta\in\Xi_{m}^{n+1}(R,\sigma)}}&f_\theta|^2|^{\frac{p}{2}}W\lesssim_{\d_0} r^{\d_0}\sum_{J\in\mc{J}(r^{-\frac{1}{n}})} \int_{\R^{n+1}} |\sum_{\substack{I(\theta)\subset J\\\theta\in\Xi_{m}^{n+1}(R,\sigma)}}f_\theta|^pW \\
    & + \sum_{k=1}^{n-m} r^{\d_0^{n-k+1}}\Big[S_{\sigma}^{n,m}(\min(\rho,r^{1+\d_0^{n-k}}),\rho,R) \Big]\int_{\R^{n+1}}|\sum_{J'} |\sum_{\substack{I(\theta)\subset J'\\\theta\in\Xi_{m}^{n+1}(R,\sigma)}}f_\theta|^2|^{\frac{p}{2}}W. 
\end{align*}
\end{lemma}

\begin{lemma} \label{lem3k} Assume 
that $r^{-\frac{1}{n}}\le \sigma_{m-1}^{-1}R^{-\frac{1}{n}}$. For each $J\in\mc{J}(r^{-\frac{1}{n}})$, 
\begin{align*} 
\int_{\R^{n+1}} |\sum_{\substack{I(\theta)\subset J\\\theta\in\Xi_{m}^{n+1}(R,\sigma)}}f_\theta|^pW&\lesssim  S_{\sigma(r)}^{n,m}(1,\rho/r,R/r) \int_{\R^{n+1}} |\sum_{J'\subset J}|\sum_{\substack{I(\theta)\subset J\\\theta\in\Xi_{m}^{n+1}(R,\sigma)}}f_\theta|^2|^{\frac{p}{2}}W
\end{align*}
where $\sigma(r)$ is an $(m+1)$-tuple with $\sigma_m(r)=1$ and $\sigma_i(r)=\max(\sigma_i,\sigma_{i+1}^{-1}\cdots\sigma_m^{-1}(R/r)^{-1/n})$ for $i=0,\ldots,m-1$. 
\end{lemma}

First we show how the previous lemmas imply Proposition \ref{multiconek}. 
\begin{proof}[Proof of Proposition \ref{multiconek}] Let $\sigma\in\mc{S}^{n,m}_R$. Assume that $\sigma_{m-1}^{-1}R^{-\frac{1}{n}}>K^{-1}$. In this case, $S_\sigma^{n,m}(r,\rho,R)\lesssim K^CS_K^{n,m}(r,\rho,R)$. This is because after using the triangle inequality and rescaling (for the cost of $K^C$), it suffices to suppose that each $f_\theta$ is supported in 
\[ \{\sum_{i=0}^n\lambda_i\phi_n^{(i)}(a_\theta):R^{-\frac{m}{n}}\ge |\lambda_m|\ge R^{-\frac{m}{n}}/2,\quad |\lambda_i|\le \min(K^{-1}R^{-\frac{m}{n}},R^{-\frac{i}{n}})\quad\forall i\not=m\}. \]
By dyadic pigeonholing, there is a bump function $\eta_{h_m}$ supported in an annulus $||\xi|-h_m|<K^{-\b}R^{-\frac{m}{n}}$, for some $R^{-\frac{m}{n}}/2\le |h_m|\le R^{-\frac{m}{n}}$ which satisfies 
\[ \int_{\R^{n+1}}|\sum_{\theta\in\Xi_m^{n+1}(R,\sigma)}f_\theta*\widecheck{\eta}_{h_m}|^pW. \]
The Fourier support of each $f_\theta*\widecheck{\eta}_{h_m}$ is contained in an element of $\Xi_m^{n+1}(R,K,{\bf{h}})$ for ${\bf{h}}=(K^{-1}R^{-\frac{m}{n}},\ldots,K^{-1}R^{-\frac{m}{n}},h_m)$. Then by Proposition \ref{SnmmKbd}, 
\[ \int_{\R^{n+1}}|\sum_{\substack{\theta\in\Xi_m^{n+1}(R,\sigma)}} f_{\theta}|^{p}W\lesssim_\e  K^{O(1)}\rho^\e \int_{\R^{n+1}}(\sum_{J\in\mc{J}(\rho^{-\frac{1}{n}})}|\sum_{\substack{\theta\in\Xi_m^{n+1}(R,\sigma)\\I(\theta)\subset J}} f_{\theta}*\widecheck{\eta}_{h_m}|^2)^{\frac{p}{2}}W. \]
By Cauchy-Schwarz and properties of weight functions, the integrand is $\lesssim (\sum_{\theta\in\Xi_m^{n+1}(R,\sigma)}|f_\theta|^2)^{\frac{p}{2}}W$. Since $K$ is a constant depending on $\e$, the proposition is proved in this case. 

It remains to consider the $\sigma\in\mc{S}^{n,m}_R$ satisfying $\sigma_{m-1}^{-1}R^{-\frac{1}{n}}<K^{-1}$. By Lemma \ref{lem1k}, it suffices to consider $r^{-\frac{1}{n}}\le \sigma_{m-1}^{-1}R^{-\frac{1}{n}}$. Then the combination of Lemmas \ref{lem2k} and \ref{lem3k} implies that for each $0<\d_0<1$, 
\begin{align*} 
S_{\sigma}^{n,m}(r,\rho,R)&\lesssim_{\d_0} r^{\d_0} S_{\sigma(r)}^{n,m}(1,\rho/r,R/r) +\sum_{k=1}^{n-m} r^{\d_0^{n-k+1}}\Big[S_{\sigma}^{n,m}(\min(\rho,r^{1+\d_0^{n-k}}),\rho,R) \Big]. 
\end{align*}
This is similar to the multi-scale inequality satisfied by $S_{K}^{n,m}(r,\rho,R)$. In this set-up, we may bound the first term immediately since $(\sigma(r))_{m-1}\le K(R/r)^{-\frac{1}{n}}$. It thus suffices to suppose that 
\[ S_{\sigma}^{n,m}(r,\rho,R)\lesssim_{\d_0} \sum_{k=1}^{n-m} r^{\d_0^{n-k+1}}\Big[S_{\sigma}^{n,m}(\min(\rho,r^{1+\d_0^{n-k}}),\rho,R) \Big]. \]
This implies the desired bound for $S_\sigma^{n,m}(r,\rho,R)$ by the same argument as in the proof of Proposition \ref{SnmmKbd}. 
\end{proof}

It remains to verify Lemmas \ref{lem1k}, \ref{lem2k}, and \ref{lem3k}. 

\begin{proof}[Proof of Lemma \ref{lem1k}] By Khintchine's inequality, it suffices to bound 
\[ \int_{\R^{n+1}}|\sum_{\theta\in\Xi_m^{n+1}(R,\sigma)}f_\theta|^pW. \]
Organize the $f_\theta$ into subcollections based on $I(\theta)$ being in the same $2^{-100n}\sigma_{m-1}^{-1}R^{-\frac{1}{n}}$-interval. If $J\in\mc{J}(\sigma^{-1}R^{-\frac{1}{n}})$, then $\sum_{I(\theta)\subset J}f_\theta$ is Fourier supported in a set $\sigma^m\tau\times \R$ for some $\tau\in\Xi_{m-1}^{n}(\sigma^{n-1}R^{\frac{n-1}{n}},\sigma')$. Here, $\sigma'$ is an $m$-tuple with $\sigma_i'=\sigma_i$ for $i<m-1$ and $\sigma_{m-1}'=1$. After dilating the spatial side by a factor of $\sigma^m$, the weight function $W$ becomes localized to a ball of radius $\sigma^mR^{\frac{m}{n}}\rho\ge \sigma^{m-1}R^{\frac{m-1}{n}}\min(\rho,\sigma^nR)^{\frac{n-1}{n}}$. Therefore, the lemma follows from a cylindrical version of the defining inequality for \newline $S_{\sigma'}^{n-1,m-1}(1,\min(\rho,\sigma^{n}R)^{\frac{n-1}{n}},\min(\rho,\sigma^{n}R))^{\frac{n-1}{n}}$. 

\end{proof}

\begin{proof}[Proof of Lemma \ref{lem2k}] Lemma \ref{lem2k} is a small variation of Lemma \ref{multilem1}, so we only describe how to adapt the proof of Lemma \ref{multilem1}. We consider $\sigma^nR\le r\le \rho\le R$. For each $J$, the summand
$|\sum_{\substack{I(\theta)\subset J\\\theta\in\Xi_{m}^{n+1}(R,\sigma)}}f_\theta|^2$ is Fourier supported in a set of the form
\begin{equation}\label{typeof} \{\sum_{i=0}^n\lambda_i\phi_n^{(i)}(a):|\lambda_i|\le R^{-\frac{m}{n}}r^{\frac{m-i}{n}}\} \end{equation}
for some $a\in 2^{-100n}r^{-\frac{1}{n}}\Z\cap[0,1]$. 
As before, let $\d_k=\d_0^{n-k}$ and $d_k=4(n-k)d$. As in the proof of Lemma \ref{multilem1}, we dyadically decompose the Fourier supports into subsets
\[ \{\sum_{i=0}^n\lambda_i\phi_n^{(i)}(a):|\lambda_i|\le\mu^{m+2-i}R^{-\frac{m}{n}}r^{\frac{m-i}{n}}\}\setminus \{\sum_{i=0}^n\lambda_i\phi_n^{(i)}(a):|\lambda_i|\le(\mu/2)^{m+2-i}R^{-\frac{m}{n}}r^{\frac{m-i}{n}} \} , \]
where $r^{\d_2}\le \mu\le 1$ and we take the smallest set to be $\{\sum_{i=0}^n\lambda_i\phi_n^{(i)}(a):|\lambda_i|\le (r^{-\d_2})^{m+2-i}R^{-\frac{m}{n}}r^{\frac{m-i}{n}}\}$. Suppose that $\eta^{2}_{J,\mu}$ are bump functions associated to the above sets and which satisfy 
\[ \int_{\R^{n+1}}|\sum_{J} |\sum_{\substack{I(\theta)\subset J\\\theta\in\Xi_{m}^{n+1}(R,\sigma)}}f_\theta|^2|^{\frac{p}{2}}W\lesssim (\log r)^c \int_{\R^{n+1}}|\sum_{J} |\sum_{\substack{I(\theta)\subset J\\\theta\in\Xi_{m}^{n+1}(R,\sigma)}}f_\theta|^2*\widecheck{\eta}^{2}_{J,\mu}|^{\frac{p}{2}}W . \]
If $\mu>r^{-\d_2}$, then since $r^{\frac{m}{n}}\rho\ge  \mu^nr$, we have by $\C^{n+1}_{m+1}(r)\lesssim_\d r^\d$ and Lemma \ref{loclem} that 
\begin{align*} 
\int_{\R^{n+1}}|\sum_{J} |\sum_{\substack{I(\theta)\subset J\\\theta\in\Xi_{m}^{n+1}(R,\sigma)}}f_\theta|^2*\widecheck{\eta}^{2}_{J,\mu}|^{\frac{p}{2}}W &\lesssim_{\d_0}r^{C\d_2} \int_{\R^{n+1}}|\sum_{J} ||\sum_{\substack{I(\theta)\subset J\\\theta\in\Xi_{m}^{n+1}(R,\sigma)}}f_\theta|^2*\widecheck{\eta}^{2}_{J,\mu}|^2|^{\frac{p}{4}}W \\
    &\lesssim_{\d_0} r^{C\d_2} \int_{\R^{n+1}}|\sum_{J} |\sum_{\substack{I(\theta)\subset J\\\theta\in\Xi_{m}^{n+1}(R,\sigma)}}f_\theta|^4*\tilde{W}_{J,{m+2}}^{n+1,d_2}|^{\frac{p}{4}}W . 
\end{align*}
Since the Fourier support of each summand $|\sum_{\substack{I(\theta)\subset J\\\theta\in\Xi_{m}^{n+1}(R,\sigma)}}f_\theta|^4*\tilde{W}_{J,{m+2}}^{n+1,d_2}$ is contained in the same type of set as \eqref{typeof}, we may iterate this procedure. If $1\le \frac{p}{4}\le 2$, then decompose the first $n$ many coordinates. If $2\le \frac{p}{4}\le \tilde{p}_{n-m-2}$, then decompose the first $m+3$ many coordinates. If each step of the iteration produces high-frequency dominating cases, we will obtain the inequality of the form 
\[ \int_{\R^{n+1}}|\sum_J |\sum_{\substack{I(\theta)\subset J\\\theta\in\Xi_{m}^{n+1}(R,\sigma)}}f_\theta|^2|^{\frac{p}{2}}W\lesssim_{\d_0}r^{\d_0}\sum_J\int_{\R^{n+1}}|\sum_{\substack{I(\theta)\subset J\\\theta\in\Xi_{m}^{n+1}(R,\sigma)}}f_\theta|^p W ,\]
which proves the lemma. If one of the steps of the iteration is low-frequency dominating, then we use a pointwise bound. We now describe one of those cases. 

In the case that $\mu=r^{-\d_2}$, which we call a low-frequency case, we use an argument analogous to Lemma \ref{ptwise}. We bound each summand 
\[ ||\sum_{\substack{I(\theta)\subset J\\\theta\in\Xi_{m}^{n+1}(R,\sigma)}}f_\theta|^2*\widecheck{\eta}^{2}_{J,\mu}(x)|\lesssim |\sum_{\substack{I(\theta)\subset J\\\theta\in\Xi_{m}^{n+1}(R,\sigma)}}f_\theta|^2*\tilde{W}^{n+1,d_2}_{J,m+2,\mu}(x)\]
pointwise. After rescaling, the expression on the right hand side becomes 
\[ |\sum_{\substack{I(\theta)\subset J\\\theta\in\Xi_{m}^{n+1}(R,\sigma)}}\underline{f_\theta}|^2*\underline{\tilde{W}}^{n+1,d_2}_{J,m+2,\mu}(x) .\]
If we write $x=(x_2,x')\in\R^{m+2}\times\R^{n-m-1}$, then for fixed $x'$, $\underline{\tilde{W}}^{n+1,d_2}_{J,\mu}(\cdot,x')$ is a weight function localized to a union of balls of radius $r^{\d_2}R^{\frac{m}{n}}r^{-\frac{m}{n}}$ and each $\underline{f}_\theta(\cdot,x')$ is Fourier supported in an element of $\Xi_m^{m+2}((R/r)^{\frac{m+1}{n}},\sigma(r))$, where $\sigma(r)$ is an $(m+1)$-tuple with $\sigma_m(r)=1$ and $\sigma_i(r)=\max(\sigma_i,\sigma_{i+1}^{-1}\cdots\sigma_m^{-1}(R/r)^{-1/n})$ for $i=0,\ldots,m-1$. Then by Proposition \ref{L2conekaklocgen}, the previous displayed expression is bounded by 
\[ C_\d r^\d \sum_{J''\subset J}|\sum_{\substack{I(\theta)\subset J'\\\theta\in\Xi_{m}^{n+1}(R,\sigma)}}\underline{f_\theta}|^2*\underline{\tilde{W}}^{n+1,d_2}_{J,m+2,\mu}(x), \]
where $J''\in\mc{J}(\max(\rho^{-\frac{1}{n}},r^{-\frac{1+\d_2}{n}}))$. (Note that in other low-frequency dominating cases that appear elsewhere in this algorithm, we would invoke the hypothesized boundedness of $S_{\sigma'}^{m+k,m}(\cdot,\cdot,\cdot)$ instead of Proposition \ref{L2conekaklocgen} to obtain pointwise bounds.) The summary so far of this low-frequency case is that 
\[\int_{\R^{n+1}}|\sum_{J} |\sum_{\substack{I(\theta)\subset J\\\theta\in\Xi_{m}^{n+1}(R,\sigma)}}f_\theta|^2|^{\frac{p}{2}}W\lesssim C_{\d_1}r^{O(\d_1)} \int_{\R^{n+1}}|\sum_J\sum_{J''\subset J} |\sum_{\substack{I(\theta)\subset J''\\\theta\in\Xi_{m}^{n+1}(R,\sigma)}}f_\theta|^2*\tilde{W}_{J,m+2,r^{-\d_2}}^{n+1,d_2}|^{\frac{p}{2}}W. \]
By Lemma \ref{multilem3pf}, the integral on the right hand side is bounded above by 
\[ C_{\d_1}r^{\d_1} \int_{\R^{n+1}}|\sum_{J''} |\sum_{\substack{I(\theta)\subset J''\\\theta\in\Xi_{m}^{n+1}(R,\sigma)}}f_\theta|^2|^{\frac{p}{2}}W , \]
which we may then bound using $S_\sigma^{n,m}(\min(\rho,r^{1+\d_0^{n-2}}),\rho,R)$ to prove the lemma. 

\end{proof}

\begin{proof}[Proof of Lemma \ref{lem3k}] We perform a rescaling analogous to the proof of Proposition \ref{conemulti}. Fix $J\in\mc{J}(r^{-\frac{1}{n}})$. After rescaling, the integral 
\[ \int_{\R^{n+1}}|\sum_{\substack{I(\theta)\subset J\\ \theta\in\Xi_m^{n+1}(R,\sigma)}}f_\theta|^pW \]
becomes
\[ \int_{\R^{n+1}}|\sum_{\substack{I(\theta)\subset J\\ \theta\in\Xi_m^{n+1}(R,\sigma)}}\underline{f}_\theta|^p\underline{W} \]
where $\underline{f}_\theta$ has Fourier support in an element of $\Xi_m^{n+1}(R/r,\sigma(r))$. The weight function $W$ which is initially adapted to a ball of radius $R^{\frac{m}{n}}\rho$ is dilated by factors of $1,r^{-\frac{1}{n}},\ldots,r^{-1}$ in $n$ many directions determined by $J$. The rescaled weight function $\underline{W}$ may be approximated by a sum of weight functions which are localized to finitely overlapping $(R/r)^{\frac{m}{n}}(\rho/r)$-balls. It follows that 
\[ \int_{\R^{n+1}}|\sum_{\substack{I(\theta)\subset J\\ \theta\in\Xi_m^{n+1}(R,\sigma)}}\underline{f}_\theta|^p\underline{W}\lesssim S_{\sigma(r)}^{n,m}(1,\rho/r,R/r) \int_{\R^{n+1}}|\sum_{\substack{J'\subset J\\ J'\in\mc{J}(\rho^{-\frac{1}{n}}}} |\sum_{\substack{I(\theta)\subset J'\\ \theta\in\Xi_m^{n+1}(R,\sigma)}}\underline{f}_\theta|^p\underline{W}. \]
Undoing the change of variables proves the lemma. 
\end{proof}

\subsection{Proof of Proposition \ref{coneinduct} \label{truncsec}}

We prove Proposition \ref{coneinduct} by reducing to the set-up from $S_\sigma^{n,m}(1,R,R)$ for some $\sigma$. 

\begin{proof}[Proof of Proposition \ref{coneinduct}] We will show that $\C_m^{n+1}(R)$ satisfies 
\[ \C_m^{n+1}(R)\le (\log R)^C\sup_{\sigma\in\mc{S}^{n,m}_R}S_\sigma^{n,m}(1,R,R) . \]
Then Proposition \ref{coneinduct} follows from Proposition \ref{multiconek}. By definition, for each $\theta\in\Xi_m^{n+1}(R)$, $f_\theta$ has Fourier support contained in 
\begin{equation}\label{thetaapprox} 
\{\sum_{i=0}^m\lambda_i\phi_n^{(i)}(a_\theta): |\lambda_i|\le 2R^{-\frac{i}{n}}\,\,\forall i,\quad \max_{0\le i\le m}|\lambda_i|2^{m-i}R^{\frac{i}{n}}\ge \frac{1}{4}\} \end{equation}
for some $a_\theta=2^{-100n}R^{-\frac{1}{n}}\Z$. We use an analogous construction as in the proof of Proposition \ref{gencone}. For $\sigma\in 2^\Z\cap [2^{-1}R^{-\frac{1}{n}},1]$, let $\eta_{\theta,\sigma}^{m-1}$ be bump functions supported on 
\[ \{\sum_{i=0}^m\lambda_i\phi_n^{(i)}(a_\theta): |\lambda_i|\le \sigma^{m-i} (R/2)^{-\frac{i}{n}}\,\,\forall i\}\setminus \{\sum_{i=0}^m\lambda_i\phi_n^{(i)}(a_\theta): |\lambda_i|\le (\sigma/2)^{m-i}(R/2)^{-\frac{i}{n}}\,\,\forall i\}, \]
and $\eta_{\theta,2^{-1}R^{-1/n}}$ supported in $\{\sum_{i=0}^m\lambda_i\phi_n^{(i)}(a_\theta): |\lambda_i|\le R^{-\frac{m}{n}}\}$. By the triangle inequality, there is some tuple $\sigma$ which satisfies 
\[  \int_{\R^{n+1}}|\sum_\theta f_\theta|^{p}\le (\log R)^C \int_{\R^{n+1}}|\sum_{\substack{\theta}} f_{\theta}*\widecheck{\eta}_{\theta,\sigma_0}^{0}\cdots*\widecheck{\eta}_{\theta,\sigma_{m-1}}^{m-1}|^{p}. \]
The summands $f_{\theta}*\widecheck{\eta}_{\theta,\sigma_0}^{0}\cdots*\widecheck{\eta}_{\theta,\sigma_{m-1}}^{m-1}$ are supported on elements of $\Xi_m^{n+1}(R,\sigma)$, where $\sigma=(\sigma_0,\ldots,\sigma_{m-1},1)$. In this case, we have 
\[ \int_{\R^{n+1}}|\sum_\theta f_\theta|^{p}\le (\log R)^CS_\sigma^{n,m}(1,R,R) \int_{\R^{n+1}}(\sum_{\substack{\theta}}| f_{\theta}*\widecheck{\eta}_{\theta,\sigma_0}^{0}\cdots*\widecheck{\eta}_{\theta,\sigma_{m-1}}^{m-1}|^2)^{\frac{p}{2}} ,\]
so it suffices to prove that the integral on the right hand side is bounded above by $C_\d r^\d \int_{\R^{n+1}}|\sum_{\theta}|f_\theta|^2|^{\frac{p}{2}}$. 

The auxiliary functions satisfy $|\widecheck{\eta}_{\theta,\sigma_i}^i|\lesssim \tilde{W}_{I(\theta),i+1,\sigma_i}^{n,1}$ for each $i$. By Cauchy-Schwarz, it suffices to show that 
\[ \int_{\R^{n+1}}|\sum_{\substack{\theta}} |f_{\theta}|^2*\tilde{W}_{I(\theta),1,\sigma_0}^{n,1}*\cdots*\tilde{W}_{I(\theta),m,\sigma_{m-1}}^{n,1}*\tilde{W}_{I(\theta),m+1}^{n,1}|^{\frac{p}{2}} \lesssim_\d R^\d \int_{\R^{n+1}}|\sum_{\substack{\theta}} |f_{\theta}|^2|^{\frac{p}{2}}. \]
This follows directly from $m$ applications of Lemma \ref{multilem3pf}, so we are done.

\end{proof}


\section{High-low analysis set-up for Proposition \ref{momcurveinduct} \label{tools}}

Recall from \textsection\ref{intro1} that there are three key ingredients to a general multi-scale analysis proving a square function estimate. To see the role of the first ingredient (a \emph{base case}), consider the following special case of the iteration. Our goal is to prove
\[ \int_{\R^n}|\sum_{\theta\in\Theta^n(R)}f_\theta|^{p_n}\le C_\e R^\e\int_{\R^n}|\sum_{\theta\in\Theta^n(R)}|f_\theta|^2|^{\frac{p_n}{2}} \]
for any Schwartz functions with $\supp\widehat{f_\theta}\subset\theta$. Temporarily use $\lessapprox$ to hide implicit constants that we are not worrying about here. One iteration of the key ingredients from \textsection\ref{intro1} uses $\M^n(K)$, a special version of type 2(b) progress, and rescaling, yielding
\begin{align*}
(\text{def. of $\M^n(K)$}) \qquad \qquad \int_{\R^n}|\sum_{\theta\in\Theta^n(R)}f_\theta|^{p_n}&\le \M^n(K)\int_{\R^n}(\sum_{\tau\in\Theta^n(K)}|\sum_{\substack{\theta\in\Theta^n(R)\\\theta\subset\tau}}f_\theta|^2)^{\frac{p_n}{2}}\\
(\text{type 2(b) progress})\qquad\qquad\qquad    
    &\lessapprox \M^n(K)\sum_{\tau\in\Theta^n(K)}\int_{\R^n}|\sum_{\substack{\theta\in\Theta^n(R)\\\theta\subset\tau}}f_\theta|^{p_n} \\
(\text{rescaling $+$ $\M^n(K)$})\qquad\qquad\qquad    
    &\lessapprox [\M^n(K)]^2\sum_{\tau\in\Theta^n(K)}\int_{\R^n}(\sum_{\substack{\tau'\in\Theta^n(K^2)\\\tau'\subset\tau}}|\sum_{\substack{\theta\in\Theta^n(R)\\\theta\subset\tau'}}f_\theta|^2)^{\frac{p_n}{2}} .
\end{align*}    
If these steps are iterated, then we obtain
\begin{align*}
\int_{\R^n}|\sum_{\theta\in\Theta^n(R)}f_\theta|^{p_n} \lessapprox [\M^n(K)]^s \sum_{\tau\in\Theta^n(K^s)} \int_{\R^n}|\sum_{\substack{\theta\in\Theta^n(R)\\\theta\subset\tau}}f_\theta|^{p_n}. 
\end{align*}
If $K^s=R$ (and $p\ge 2$), then 
\[ \sum_{\tau\in\Theta^n(K^s)} \int_{\R^n}|\sum_{\substack{\theta\in\Theta^n(R)\\\theta\subset\tau}}f_\theta|^{p_n}=\sum_{\theta\in\Theta^n(R)} \int_{\R^n}|f_\theta|^{p_n} \le \int_{\R^n}(\sum_{\theta\in\Theta^n(R)}|f_\theta|^2)^{\frac{p_n}{2}}.\]
We have accumulated $\log R/\log K$ many factors of the factor $\M^n(K)$. There is no \emph{a priori} bound for $\M^n(K)$ since $p_n>p_{n-1}$. We could try induction on scales. Our goal is to show that $\M^n(R)\le \tilde{C}_\e R^\e$. By induction, suppose that $\M^n(K)\le \tilde{C}_\e K^\e$. Then the above algorithm accumulates a bound of the form $\M^n(R)\lessapprox (\tilde{C}_\e K^\e)^{\frac{\log R}{\log K}}=\tilde{C}_\e^{\frac{\log R}{\log K}} R^\e$ and the induction does not close since $\tilde{C}_\e ^{\frac{\log R}{\log K}}\gg \tilde{C}_\e$. The high-low frequency argument will give us a way to bound $\M^n(R)$ by an expression that does not have too many factors of $\M^n(K)$, which will allow us to close the induction. 

A highly abbreviated version of the high-low frequency analysis is as follows. First we approximate the $L^p$ integral with a level set
\[ \int_{\R^n}|\sum_{\theta\in\Theta^n(R)}f_\theta|^{p_n}\approx \a^{p_n}|U_\a|,\]
where $U_\a=\{x\in\R^n:|\sum_{\theta\in\Theta^N(R)}f_\theta|\sim \a\}$. Then there is a stopping time algorithm which produces a scale $1<r<R$ satisfying
\begin{enumerate}
    \item $\sum_{\tau\in\Theta^n(r)}|\sum_{\substack{\theta\in\Theta^n(R)\\\theta\subset\tau}} f_\theta|^2\sim \sum_{\theta\in\Theta^n(R)}|f_\theta|^2$ on most of $U_\a$, and 
    \item the inequality
    \[ \a^{p_n}|U_\a|\lessapprox \int_{\R^n}|\sum_{\tau\in\Theta^n(r)}|\sum_{\substack{\theta\in\Theta^n(R)\\\theta\subset\tau}}f_\theta|^2|^{\frac{p_n}{2}}. \]
\end{enumerate}
If $r$ is small, say $r\le R^\d$ for some $\d>0$, then we have a good bound just using (1) and Cauchy-Schwarz:
\[ \a^{p_n}|U_\a|\lesssim R^{\d\frac{p_n}{n}}\int_{U_\a}(\sum_{\tau\in\Theta^n(r)}|\sum_{\substack{\theta\in\Theta^n(R)\\\theta\subset\tau}}f_\theta|^2)^{\frac{p_n}{2}} \lesssim R^{\d\frac{p_n}{n}}\int_{U_\a}(\sum_{\theta\in\Theta^n(R)}|f_\theta|^2)^{\frac{p_n}{2}}. \] 
If $r$ is large, so $r>R^\d$, we use the inequality from (2). If we consider the special iteration discussed above, then the right hand side of the inequality in (2) is bounded by 
\begin{align*}
(\text{type 2(b) progress}) \qquad \qquad \int_{\R^n}|\sum_{\tau\in\Theta^n(r)}|\sum_{\substack{\theta\in\Theta^n(R)\\\theta\subset\tau}}f_\theta|^2|^{\frac{p_n}{2}}&\lessapprox \sum_{\tau\in\Theta^n(r)}\int_{\R^n}|\sum_{\substack{\theta\in\Theta^n(R)\\\theta\subset\tau}}f_\theta|^{p_n}\\
(\text{rescaling $+$ $\M^n(K)$})\qquad\qquad\qquad\qquad\qquad\qquad\qquad\qquad    
    &\lessapprox \M^n(K)\sum_{\tau\in\Theta^n(r)}\int_{\R^n}(\sum_{\substack{\tau'\in\Theta^n(rK)\\\tau'\subset\tau}}|\sum_{\substack{\theta\in\Theta^n(R)\\\theta\subset\tau'}}f_\theta|^2)^{\frac{p_n}{2}}     \\
    &\lessapprox \cdots \lessapprox [\M^n(K)]^s \sum_{\tau\in\Theta^n(rK^s)} \int_{\R^n}|\sum_{\substack{\theta\in\Theta^n(R)\\\theta\subset\tau}}f_\theta|^{p_n}. 
\end{align*}
Assuming by induction that $\M^n(K)\le \tilde{C}_\e K^\e$, the above argument gives the inequality $\M^n(R)\lessapprox (\tilde{C}_\e K^\e)^{\frac{\log(R/r)}{\log K}} =(\tilde{C}_\e)^{\frac{\log (R/r)}{\log K}}r^{-\e}R^\e$. If $r$ is large enough, then $(\tilde{C}_\e)^{\frac{\log (R/r)}{\log K}}r^{-\e}\le \tilde{C}_\e$. Finally, we choose $\d$ and $K$ so that both cases of the high-low analysis allow us to close the induction. This is carried out carefully in the proof of Proposition \ref{S1bd}. 

Now we begin laying out the set-up for the high-low analysis, beginning with a detailed definition of the anisotropic neighborhoods. The set-up is analogous to the one used in \cite{gmw}. The following alternative definition of canonical moment curve blocks has the advantage that the blocks at various scales are nested. 

\begin{definition}[Canonical moment curve blocks]
For $R\in 2^{n\N}$, define ${\bf{S}}_n(R^{-\frac{1}{n}})$ to be the following collection of canonical moment curve blocks at scale $R$ which partition $\mc{M}^n(R)$:
\[ \bigsqcup\limits_{l=0}^{R^{\frac{1}{n}}-1}\mc{M}^n(R)\cap\{\xi\in\R^N:lR^{-\frac{1}{n}}\le \xi_1<(l+1)R^{-\frac{1}{n}}\}.  \]
\end{definition}

If $\tau\in{\bf{S}}_n(R^{-\frac{1}{n}})$ is the $\ell$th moment curve block, then $\tau$ is comparable to the set
\[ \{\g_n(lR^{-\frac{1}{n}})+\sum_{i=1}^n\lambda_i\g_n^{(i)}(lR^{-\frac{1}{n}}): |\lambda_i|\le R^{-\frac{i}{n}}\} . \]
By comparable, we mean that there is an absolute constant $C>0$ for which $C^{-1}\tau$ is contained in the displayed set and $C\tau$ contains the displayed set, where the dilations are taken with respect to the centroid of $\tau$. Since the moment curve blocks in $\Theta^n(R)$ are comparable to those in ${\bf{S}}_n(R^{-\frac{1}{n}})$, it suffices to prove a square function estimate with either definition of moment curve blocks.

For $i\ge 1$, let $\tilde{\g}_n^{(i)}(t)$ denote the projection of $\g_n^{(i)}(t)$ onto the orthogonal complement of $\text{Span}(\g_n(t),\ldots,\g_n^{(i-1)}(t))$. Define the dual set $\tau^*$ by 
\begin{equation}\label{dualdef}
\tau^* = \{\sum_{i=1}^n\lambda_i\tilde{\g}_n^{(i)}(t):|\lambda_i|\le R^{\frac{i}{n}}\}.  
\end{equation}
We sometimes refer to the set $\tau^*$ as well as its translates as wave packets.

Next, we fix some notation for the scales. Let $\e>0$. To prove Proposition \ref{momcurveinduct}, it suffices to assume that $R$ is larger than a constant which depends on $\e$. Consider scales $R_k\in 8^\N$ closest to $R^{k\e}$, for $k=1,\ldots,N$ and $R_N\le R\le R^\e R_N$. Since $R$ differs from $R_N$ at most by a factor of $R^\e$, we will assume that $R=R_N$. The relationship between the parameters is
\[ 1=R_0\le R_k\le R_{k+1}\le R_N=R. \]

Fix notation for moment curve blocks of various sizes. 
\begin{enumerate}
    \item Let $\theta$ denote elements of ${\bf{S}}_n(R^{-1/n})$. 
    \item Let $\tau_k$ denote elements of ${\bf{S}}_n(R_k^{-1/n})$. 
\end{enumerate}
The definitions of $\theta,\tau_k$ provide the additional property that if $\tau_k\cap\tau_{k+m}\not=\emptyset$, then $\tau_{k+m}\subset\tau_k$.

We will use two square function constants: $\text{T}_n(r,R)$ and (the weighted version) $\text{T}^w_{n,d}(R)$ \label{const}, where the weights are defined in Definition \ref{M3ballweight}.

\begin{definition} \label{TnRw} Let $n,d\in\N_{>0}$. Let $1\le r\le R$. Let $\emph{T}^w_{n,d}(R)$ be the infimum of $A>0$ such that
\[ \int_{\R^n}|f|^{p} \le A\int_{\R^n}|\sum_{\theta\in{\bf{S}}_n(R)}|f_\theta|^2*\w_{\theta,d}|^{\frac{p}{2}} \]
for any $2\le p\le p_n$ and any Schwartz function $f:\R^n\to\C$ with Fourier transform supported in $\mc{M}^n(R)$. 
\end{definition}

\begin{definition} \label{TnR} Let $R\ge 1$. Let $\emph{T}_n(R)$ be the infimum of $B>0$ such that
\[ \int_{\R^n}|f|^{p}\le B\int_{\R^n}|\sum_{\theta\in {\bf{S}}_n(R)} |f_\theta|^2|^{\frac{p}{2}} \]
for any $2\le p\le p_n$ and any Schwartz function $f:\R^n\to\C$ with Fourier transform supported in $\mc{M}^n(R)$.  

\end{definition}

Fix a ball $B_R\subset\R^n$ of radius $R$ as well as a Schwartz function $f:\R^n\to\C$ with Fourier transform supported in $\mc{M}^n(R)$. Also fix $D\in\N_{>0}$. The parameters $\a,\b>0$ describe the set 
\[ U_{\a,\b}=\{x\in B_{R}:|f(x)|\ge \a,\quad\frac{\b}{2}\le \sum_{\theta\in{\bf{S}}_n(R^{-1/n})}|f_\theta|^2*\w_{\theta,D}(x)\le \b\}.\]
The weight function $\w_{\theta,D}$ is defined in Definition \ref{M3ballweight} below. We assume throughout this section (and until \textsection\ref{M3pigeon}) that the $f_\theta$ satisfy the extra condition that
\begin{equation}\label{unihyp} 
\frac{1}{2}\le \|f_\theta\|_{L^\infty(\R^n)}\le 2\qquad\text{or}\qquad \|f_\theta\|_{L^\infty(\R^n)}=0. \end{equation}

\subsection{A pruning step \label{prusec}}

We define wave packets associated to $f_{\tau_k}$ and sort them according to an amplitude condition which depends on the parameters $\a$ and $\b$. 

For each $\tau_k$, let $\T_{\tau_k}$ be the collection of $\tau_k^*$ its translates $T_{\tau_k}$ which form a tiling of $\R^n$. Fix an auxiliary function $\p(\xi)$ which is a bump function supported in $[-\frac{1}{4},\frac{1}{4}]^n$. For each $m\in\Z^n$, let 
\[ \s_m(x)=c\int_{[-\frac{1}{2},\frac{1}{2}]^n}|\widecheck{\p}|^2(x-y-m)dy, \]
where $c$ is chosen so that $\sum_{m\in\Z^n}\s_m(x)=c\int_{\R^n}|\widecheck{\p}|^2=1$. Since $|\widecheck{\p}|$ is a rapidly decaying function, for any $k\in\N$, there exists $C_k>0$ such that
\[ \s_m(x)\le c\int_{[-\frac{1}{2},\frac{1}{2}]^3}\frac{C_k}{(1+|x-y-m|^2)^n}dy \le \frac{\tilde{C}_k}{(1+|x-m|^2)^k}. \]
Define the partition of unity $\s_{T_{\tau_k}}$ associated to ${\tau_k}$ to be $\s_{T_{\tau_k}}(x)=\s_m\circ A_{\tau_k}$, where $A_{\tau_k}$ is a linear transformations taking $\tau_k^*$ to $[-\frac{1}{2},\frac{1}{2}]^n$ and $A_{\tau_k}(T_{\tau_k})=m+[-\frac{1}{2},\frac{1}{2}]^n$. The important properties of $\s_{T_{\tau_k}}$ are (1) rapid decay off of $T_{\tau_k}$ and (2) Fourier support contained in $\tau_k$ translated to the origin. We sort the wave packets $\T_{\tau_k}=\T_{\tau_k}^g\sqcup\T_{\tau_k}^b$ into ``good" and ``bad" sets, and define corresponding versions of $f$, as follows. 


\begin{rmk} In the following definitions, let $K\ge 1$  be a large parameter which will be used to define the broad set in Proposition \ref{mainprop}. Also, $A=A(\e)\gg 1$ is a large enough constant (determined by Lemma \ref{ftofk}) which also satisfies $A\ge \tilde{D}$, where $\tilde{D}$ is from Lemma \ref{low}.
\end{rmk}

\begin{definition}[Pruning with respect to $\tau_k$]\label{taukprune} Let $f^{N}=f$, $f^{N}_{\tau_N}=f_{\theta}$. 
For each $1\le k\le N-1$, let 
\begin{align*} \T_{\tau_k}^{g}&=\{T_{\tau_k}\in\T_{\tau_{k}}:\|\s_{T_{\tau_{k}}}^{1/2}f_{\tau_{k}}^{k+1}\|_{L^\infty(R^3)}\le K^3A^{N-k+1}\frac{\b}{\a}\}, \\
f_{\tau_{k}}^{k}=\sum_{T_{\tau_k}\in\T_{\tau_k}^{g}}&\s_{T_{\tau_k}}f^{k+1}_{\tau_k}\qquad\text{and}\qquad  f_{\tau_{k-1}}^{k}=\sum_{\tau_k\subset\tau_{k-1}}f_{\tau_k}^k .
\end{align*}
\end{definition}
For each $k$, the $k$th version of $f$ is $f^k=\underset{\tau_k}{\sum} f_{\tau_k}^k$.

\begin{rmk}
We may assume that $\a\lesssim R^{C_0 \e}\b$. This will be discussed in Proposition \ref{wpd} and Corollary \ref{wpdcor}, which involve pigeonholing the wave packets of $f$.
\end{rmk}

\begin{lemma}[Properties of $f^k$] \label{pruneprop}
\begin{enumerate} 
\item\label{item1} $| f_{\tau_{k}}^k (x) | \le |f_{ \tau_{k}}^{k+1}(x)|\lesssim \#\theta\subset\tau_k.$
\item \label{item2} $\| f_{\tau_k}^k \|_{L^\infty(\R^n)} \le K^3A^{N-k+1}\frac{\b}{\a}$.
\item\label{item3} For $R$ sufficiently large depending on $\e$, $\text{supp} \widehat{f_{\tau_k}^{k}}\subset3\tau_k. $
\end{enumerate}
\end{lemma}
\begin{proof} For the first property, recall that $\sum_{T_{\tau_k} \in \T_{\tau_k}}\s_{T_{\tau_k}}$ is a partition of unity so we may iterate the inequalities 
\begin{align*}
|f_{\tau_k}^k|\le |f_{\tau_k}^{k+1}|&\le \sum_{\tau_{k+1}\subset\tau_k}|f_{\tau_{k+1}}^{k+1}|\le\cdots\le \sum_{\tau_N\subset\tau_k}|f_{\tau_N}^N|= \sum_{\theta \subset\tau_k}|f_{\theta}|.  
\end{align*}
The first property follows from our assumption \eqref{unihyp} that each $\|f_\theta\|_{L^\infty(\R^n)}\lesssim 1$. For the $L^\infty$ bound in the second property, write
\[ |f_{ \tau_k}^k(x)| = |\sum_{\substack{T_{\tau_k} \in \T_{\tau_k^h}}} \s_{T_{\tau_k}}(x) f_{ \tau_k}^{k+1}(x)|\le \sum_{\substack{T_{\tau_k} \in \T_{\tau_k^h}}} \s_{T_{\tau_k}}^{1/2}(x) \|\s_{T_{\tau_k}}^{1/2}f_{ \tau_k}^{k+1}\|_\infty\lesssim \|\s_{T_{\tau_k}}^{1/2}f_{ \tau_k}^{k+1}\|_\infty. \]
\noindent By the definition of $\T_{\tau_k}^h$, $\|\s_{T_{\tau_k}}^{1/2}f_{\tau_k}^{k+1}\|_\infty\le K^3 A^{N-k+1}\frac{\b}{\a}$.

The third property depends on the Fourier support of $\s_{T_{\tau_k}}$, which is contained in $\tau_k$ shifted to the origin. Note if each $f_{\tau_k}^{k+1}$ has Fourier support in $\cup_{\tau_{k+1}\subset\tau_k}3\tau_{k+1}$, then $\supp\widehat{f_{\g_k}^k}$ is contained in $3\tau_k$. 

\end{proof}

\begin{definition} \label{M3ballweight} Let $d\in\N_{>0}$ and let $W^{n,d}$ be the weight function defined in \textsection\ref{wtsec}. Since we consider one dimension $n$ at a time, we suppress the $n$ notation in the weights. Let $B_0\subset\R^n$ denote the unit ball centered at the origin. For any set $U=T(B_0)$ where $T$ is an affine transformation $T:\R^n\to\R^n$, define
\[ w_{U,d}(x)=|U|^{-1}W^{n,d}(T^{-1}(x)). \]
For each $\tau_k$, let $A_{\tau_k}$ be a linear transformation mapping $\tau_k^*$ to the unit cube and define $\w_{\tau_k,d}$ by
\[  \w_{\tau_k,d}(x)=|\tau_k^*|^{-1}W^{n,d}(A_{\tau_k}(x)). \]
Assume that $d=1$ whenever the $d$ is not specified in a weight function.
\end{definition}
Let the capital-W version of weight functions denote the $L^\infty$-normalized (as opposed to $L^1$-normalized) versions, so for example, for any ball $B_s$, $W_{B_s,d}(x)=|B_s|w_{B_s,d}(x)$. If a weight function replaces the convex set with only a scale, say $s$, then the functions $w_{s,d},W_{s,d}$ are weight functions localized to the $s$-ball centered at the origin. 

Next, we record the locally constant property. By locally constant property, we mean that if a function $f$ has Fourier transform supported in a convex set $A$, then $|f|$ is bounded above by an averaged version of $|f|$ over a dual set $A^*$. 

\begin{lemma}[Locally constant property]\label{locconst} For each $\tau_k$ and $T_{\tau_k}\in\T_{\tau_k}$, 
\begin{align*} 
\|f_{\tau_k}\|_{L^\infty(T_{\tau_k})}^2\lesssim_d |f_{\tau_k}|^2*\w_{\tau_k,d}(x)\qquad\text{for any}\quad x\in T_{\tau_k} .\end{align*}
Also, for any $R_k^{1/n}$-ball $B_{R_k^{1/n}}$, 
\begin{align*} 
\|\sum_{\tau_k}|f_{\tau_k}|^2\|_{L^\infty(B_{R_k^{1/n}})}\lesssim_d |f_{\tau_k}|^2*w_{B_{R_k^{1/n}},d}(x)\qquad\text{for any}\quad x\in B_{R_k^{1/n}} .\end{align*}
\end{lemma}
Because the pruned versions of $f$ and $f_{\tau_k}$ have essentially the same Fourier supports as the unpruned versions, the locally constant lemma applies to the pruned versions as well.

\begin{proof}[Proof of Lemma \ref{locconst}] For the first claim, we write the argument for $f_{\tau_k}$ in detail. Let $\rho_{\tau_k}$ be a bump function equal to $1$ on $\tau_k$ and supported in $2\tau_k$. Then using Fourier inversion and H\"{o}lder's inequality, 
\[ |f_{\tau_k}(y)|^2=|f_{\tau_k}*\widecheck{\rho_{\tau_k}}(y)|^2\le\|\widecheck{\rho_{\tau_k}}\|_1 |f_{\tau_k}|^2*|\widecheck{\rho_{\tau_k}}|(y). \]
Since $\rho_{\tau_k}$ may be taken to be an affine transformation of a standard bump function adapted to the unit ball, $\|\widecheck{\rho_{\tau_k}}\|_1$ is a constant. The function $\widecheck{\rho_{\tau_k}}$ decays rapidly off of $\tau_k^*$, so $|\widecheck{\rho_{\tau_k}}|\lesssim_d \w_{{\tau_k},d}$.
Since for any $T_{\tau_k}\in\T_{\tau_k}$, $\w_{\tau_k,d}(y)\sim\w_{\tau_k}(y')$ for all $y,y'\in T_{\tau_k}$, we have
\begin{align*} \sup_{x\in T_{\tau_k}}|f_{\tau_k}|^2*\w_{\tau_k,d}(x)&\le \int|f_{\tau_k}|^2(y)\sup_{x\in T_{\tau_k}}\w_{\tau_k,d}(x-y)dy\\
&\sim \int|f_{\tau_k}|^2(y)\w_{\tau_k,d}(x-y)dy\qquad \text{for all}\quad x\in T_{\tau_k}. 
\end{align*}

For the second part of the lemma, repeat analogous steps as above, except begin with $\rho_{\tau_k}$ which is identically $1$ on a ball of radius $2R_k^{-1/n}$ containing $\tau_k$. Then 
\[  \sum_{\tau_k}|f_{\tau_k}(y)|^2=\sum_{\tau_k}|f_{\tau_k}*\widecheck{\rho_{\tau_k}}(y)|^2\lesssim \sum_{\tau_k}|f_{\tau_k}|^2*|\widecheck{\rho_{R_k^{-1/n}}}|(y),\]
where we used that each $\rho_{\tau_k}$ is a translate of a single function $\rho_{R^{-1/n}}$. The rest of the argument is analogous to the first part. 
\end{proof}

The following local $L^2$-orthogonality lemma is Lemma 3 in \cite{M3smallcap}.  
\begin{lemma}[Local $L^2$ orthogonality]\label{L2orth} Let $U=T(B)$ where $B$ is the unit ball centered at the origin and $T:\R^n\to\R^n$ is an affine transformation. Let $h:\R^n\to\C$ be a Schwartz function with Fourier transform supported in a disjoint union $X=\sqcup_k X_k$, where $X_k\subset B$ are Lebesgue measurable. If the maximum overlap of the sets $X_k+U^*$ is $L$, then
\[ \int |h_X|^2w_{U,d}\lesssim L\sum_{X_k}\int|h_{X_k}|^2w_{U,d}, \]
where $h_{X_k}=\int_{X_k}\widehat{h}(\xi)e^{2\pi i x\cdot\xi}d\xi$.
\end{lemma}
Here, we may take $\{x:|x\cdot\xi|\le 1\quad\forall \xi\in U-U\}$ as the definition of $U^*$. We will include a sketch of the proof for future reference. 
\begin{proof} By Plancherel's theorem, we have
\begin{align*}
    \int|h_X|^2w_U&=\int h_X \overline{h_Xw_U}=\int \widehat{h_X}\overline{\widehat{h_X}*\widehat{w_U}}.
\end{align*}
Since $\widehat{h_X}=\sum_k\widehat{h_{X_k}}$, $\int \widehat{h_X}\overline{\widehat{h_X}*\widehat{w_U}}=\sum_{X_k}\sum_{X_k'}\int\widehat{h_{X_k}}\overline{\widehat{h_{X_k'}}*\widehat{w_U}}$. For each $X_k$, the integral on the right hand side vanishes except for $\lesssim L$ many choices of $X_k'$. 

\end{proof}

\subsection{High-low frequency decomposition of square functions}

\begin{definition}[Auxiliary functions] Let $\eta:\R^n\to[0,\infty)$ be a radial, smooth bump function satisfying $\eta(x)=1$ on $B_{1/2}$ and $\supp\eta\subset B_1$. Then for each $s>0$, let 
\[ \eta_{\le s}(\xi) =\eta(s^{-1}\xi) .\]
We will sometimes abuse notation by denoting $h*\widecheck{\eta}_{>s}=h-h*\widecheck{\eta}_{\le s}$, where $h$ is some Schwartz function. Also define $\eta_{s}(x)=\eta_{\le s}-\eta_{\le s/2}$. 
\end{definition}

\vspace{3mm}
Fix $N_0<N-1$ which will be specified in \textsection\ref{ind}. Also fix $d_0=d_0(n,\e,D)\in\N_{>0}$ which will be specified in Proposition \ref{algo}. Let $d_k=(4\e^{-1}-k)d_0$
\begin{definition} For $N_0\le k\le N-1$, let 
\[ g_k(x)=\sum_{\tau_k}|f_{\tau_k}^{k+1}|^2*\w_{\tau_k,d_k}, \qquad g_k^{\ell}(x)=g_k*\widecheck{\eta}_{\le R_{k+1}^{-1/n}}, \qquad\text{and}\qquad g_k^h=g_k-g_k^{\ell}. \]
\end{definition}
\vspace{3mm}

In the following definition, $A\gg 1$ is the same constant that goes into the pruning definition of $f^k$. 
\begin{definition} \label{impsets}Define the high set by 
\[ H=\{x\in U_{\a,\b}: A \b \le g_{N-1}(x)\}. \]
For each $k=N_0,\ldots,N-2$, let $H=\Omega_{N-1}$ and let
\[ \Omega_k=\{x\in U_{\a,\b}\setminus \cup_{l=k+1}^{N-1}\Omega_{l}: A^{N-k}\b\le g_k(x) \}. \] 
Define the low set to be
\[ L=U_{\a,\b}\setminus[\cup_{k=N_0}^{N-1}\Omega_k]. \]
\end{definition}
\vspace{3mm}

\begin{lemma}[Low lemma]\label{low} There is an absolute constant $\tilde{D}>0$ so that for each $x$, $|g_k^\ell(x)|\le \tilde{D} g_{k+1}(x)$. 
\end{lemma}
\begin{proof} We perform a pointwise version of the argument in the proof of local/global $L^2$-orthogonality (Lemma \ref{L2orth}). For each $\tau_k^{k+1}$, by Plancherel's theorem,
\begin{align}
|f_{\tau_k}^{k+1}|^2*\widecheck{\eta}_{<R_{k+1}^{-1/n}}(x)&= \int_{\R^n}|f_{\tau_k}^{k+1}|^2(x-y)\widecheck{\eta}_{<R_{k+1}^{-1/n}}(y)dy \nonumber \\
&=  \int_{\R^n}\widehat{f_{\tau_k}^{k+1}}*\widehat{\overline{f_{\tau_k}^{k+1}}}(\xi)e^{-2\pi i x\cdot\xi}\eta_{<R_{k+1}^{-1/n}}(\xi)d\xi \nonumber \\
&=  \sum_{\tau_{k+1},\tau_{k+1}'\subset\tau_k}\int_{\R^n}e^{-2\pi i x\cdot\xi}\widehat{f_{\tau_{k+1}}^{k+1}}*\widehat{\overline{f_{\tau_{k+1}'}^{k+1}}}(\xi)\eta_{<R_{k+1}^{-1/n}}(\xi)d\xi .\label{dis2}\nonumber
\end{align}
The integrand is supported in $(2\tau_{k+1}-2\tau_{k+1}')\cap B_{R_{k+1}^{-1/n}}$. This means that the integral vanishes unless $\tau_{k+1}$ is within $\sim  R_{k+1}^{-1/n}$ of $\tau_{k+1}'$, in which case we write $\tau_{k+1}\sim\tau_{k+1}'$. Then 
\[\sum_{\tau_{k+1},\tau_{k+1}'\subset\tau_k}\int_{\R^2}e^{-2\pi i x\cdot\xi}\widehat{f}_{\tau_{k+1}}^{k+1}*\widehat{\overline{f}_{\tau_{k+1}'}^{k+1}}(\xi)\eta_{<R_{k+1}^{-1/n}}(\xi)d\xi=\sum_{\substack{\tau_{k+1},\tau_{k+1}'\subset\tau_k\\
\tau_{k+1}\sim\tau_{k+1}'}}\int_{\R^2}e^{-2\pi i x\cdot\xi}\widehat{f}_{\tau_{k+1}}^{k+1}*\widehat{\overline{f}_{\tau_{k+1}'}^{k+1}}(\xi)\eta_{<R_{k+1}^{-1/n}}(\xi)d\xi. \]
Use Plancherel's theorem again to get back to a convolution in $x$ and conclude that $|g_k*\widecheck{\eta}_{<R_{k+1}^{-1/n}}(x)|$ equals
\begin{align*}
&\Big|\sum_{\substack{\tau_{k+1},\tau_{k+1}'\subset\tau_k\\
\tau_{k+1}\sim\tau_{k+1}'}}(f_{\tau_{k+1}}^{k+1}\overline{f_{\tau_{k+1}'}^{k+1}})*\w_{\tau_k,d_k}*\widecheck{\eta}_{<R_{k+1}^{-1/n}}(x) \Big|\lesssim \sum_{\tau_k} \sum_{\tau_{k+1}\subset\tau_k}|f_{\tau_{k+1}}^{k+1}|^2*\w_{\tau_k,d_k}*|\widecheck{\eta}_{<R_{k+1}^{-1/n}}|(x)
. 
\end{align*}
By the locally constant property (Lemma \ref{locconst}) and \eqref{item1} of Lemma \ref{pruneprop}, the right hand side above is  
\[  \lesssim \sum_{\tau_k} \sum_{\tau_{k+1}\subset\tau_k}|f_{\tau_{k+1}}^{k+2}|^2*\w_{\tau_{k+1},d_{k}}*\w_{\tau_k,d_k}*|\widecheck{\eta}_{<R_{k+1}^{-1/n}}|(x)\lesssim g_{k+1}(x). \]
It remains to note that
\[ w_{\tau_{k+1},d_{k+1}}*\w_{\tau_k,d_k}*|\widecheck{\eta}_{<R_{k+1}^{-1/n}}|(x)\lesssim w_{\tau_{k+1},d_{k+1}}(x) \]
since $\tau_k^*\subset\tau_{k+1}^*$ and $\widecheck{\eta}_{<R_{k+1}^{-1/n}}$ is an $L^1$-normalized function that is rapidly decaying away from $B_{R_{k+1}^{1/n}}(0)$. 

\end{proof}

\begin{corollary}[High-dominance on $\Omega_k$]\label{highdom} For $R$ large enough depending on $\e$, $g_k(x)\le 2|g_k^h(x)|$ for all $x\in\Omega_k$. 
\end{corollary}
\begin{proof}
This follows directly from Lemma \ref{low}. Indeed, since $g_k(x)=g_k^{\ell}(x)+g_k^h(x)$, the inequality $g_k(x)>2|g_k^h(x)|$ implies that $g_k(x)<2|g_k^{\ell}(x)|$. Then by Lemma \ref{low}, $|g_k(x)|<2\tilde{D} g_{k+1}(x)$. Since $x\in\Omega_k$, $g_{k+1}(x)\le A^{N-k-1}\b$, which altogether gives the upper bound
\[ g_k(x)\le 2\tilde{D} A^{N-k-1}\b. \]
The contradicts the property that on $\Omega_k$, $A^{N-k}\b\le g_k(x)$, for $A$ sufficiently larger than $\tilde{D}$, which finishes the proof. 

\end{proof}

\begin{lemma}[Pruning lemma]\label{ftofk} For any $s\ge R^{-\e/n}$ and $\tau\in{\bf{S}}_n(s)$, 
\begin{align*} 
|\sum_{\tau_k\subset\tau}f_{\tau_k}-\sum_{\tau_k\subset\tau}f_{\tau_k}^{k+1}(x)|&\le \frac{\a}{A^{1/2}K^3} \qquad\text{for all $x\in \Omega_k$}, \qquad N_0\le k\le N-1,\\
\text{and}\qquad |\sum_{\tau_B\subset\tau}f_{\tau_B}-\sum_{\tau_B\subset\tau}f_{\tau_B}^{B}(x)|&\le \frac{\a}{A^{1/2}K^3}\qquad \text{ for all $x\in L$}. \end{align*}
\end{lemma}

\begin{proof} 
Begin by proving the first claim about $\Omega_k$. By the definition of the pruning process, we have 
\begin{equation}\label{diffs} f_{\tau}=f^{N-1}_{\tau}+(f_{\tau}^N-f^{N-1}_{\tau})=\cdots=f^{k+1}_{\tau}(x)+\sum_{m=k+1}^{N-1}(f^{m+1}_{\tau}-f^{m}_{\tau})\end{equation}
where formally, the subscript $\tau$ means $f_\tau=\sum_{\theta\subset\tau}f_\theta$ and $f_{\tau}^m=\sum_{\tau_m\subset\tau}f_{\tau_m}^m$. We will show that each difference in the sum is much smaller than $\a$.
For each $N-1\ge m\ge k+1$ and $\tau_m$, 
\begin{align*}
    |f_{\tau_m}^m(x)-f_{\tau_m}^{m+1}(x)|&=|\sum_{T_{\tau_m}\in\T_{\tau_m}^{b}}\s_{T_{\tau_m}}(x)f_{\tau_m}^{m+1}(x)|  = \sum_{T_{\tau_m}\in T_{\tau_m}^b} |\s_{T_{\tau_m}}^{1/2}(x)f_{\tau_m}^{m+1}(x)|\s_{T_{\tau_m}}^{1/2}(x) \\
     & \le\sum_{T_{\tau_m}\in \T_{\tau_m}^b}  K^{-3}A^{-(N-m+1)}\frac{\a}{\b}  \| \s_{T_{\tau_m}}^{1/2}f_{{\tau_m}}^{m+1} \|_{L^\infty(\R^n)}^2  \s_{T_{\tau_m}}^{1/2}(x) \\
     & \lesssim K^{-3}A^{-(N-m+1)}\frac{\a}{\b}\sum_{T_{\tau_m}\in \T_{\tau_m}^b}
      \sum_{\tilde{T}_{{\tau_m}}\in\T_{\tau_m}} \| \s_{T_{\tau_m}}|f_{{\tau_m}}^{m+1}|^2 \|_{L^\infty(\tilde{T}_{{\tau_m}})} \s_{T_{\tau_m}}^{1/2}(x) \\
     & \lesssim K^{-3}A^{-(N-m+1)}\frac{\a}{\b} \sum_{T_{\tau_m},\tilde{T}_{\tau_m}\in \T_{\tau_m}} \| \s_{T_{\tau_m}}\|_{L^\infty(\tilde{T}_{\tau_m})}\||f_{{\tau_m}}^{m+1} |^2\|_{{L}^\infty(\tilde{T}_{{\tau_m}})} \s_{T_{\tau_m}}^{1/2}(x) .
\end{align*}
Let $c_{\tilde{T}_{\tau_m}}$ denote the center of $\tilde{T}_{\tau_m}$ and note the pointwise inequality
\[ \sum_{{T}_{\tau_m}}\|\s_{T_{\tau_m}}\|_{L^\infty(\tilde{T}_{\tau_m})}\s_{T_{\tau_m}}^{1/2}(x)\lesssim_{d_m} |\tau_m^*|\w_{\tau_m,d_m}(x-c_{\tilde{T}_{\tau_m}}) ,\]
which means that
\begin{align*}
|f_{\tau_m}^m(x)-f_{\tau_m}^{m+1}(x)| & \lesssim_{d_m} K^{-3}A^{-(N-m+1)}\frac{\a}{\b} |\tau_m^*|\sum_{\tilde{T}_{\tau_m}\in \T_{\tau_m}} \w_{\tau_m,{d_m}}(x-c_{\tilde{T}_{\tau_m}})\||f_{{\tau_m}}^{m+1} |^2\|_{{L}^\infty(\tilde{T}_{{\tau_m}})} \\
&\lesssim_{d_m} K^{-3}A^{-(N-m+1)}\frac{\a}{\b}|\tau_m^*| \sum_{\tilde{T}_{\tau_m}\in \T_{\tau_m}} \w_{\tau_m,{d_m}}(x-c_{\tilde{T}_{\tau_m}})|f_{{\tau_m}}^{m+1} |^2*\w_{\tau_m,{d_m}}(c_{\tilde{T}_{\tau_m}})\\
&\lesssim_{d_m} K^{-3}A^{-(N-m+1)}\frac{\a}{\b} |f_{{\tau_m}}^{m+1} |^2*\w_{\tau_m,{d_m}}(x)
\end{align*}
where we used the locally constant property in the second to last inequality. The last inequality is justified by the fact that $\w_{\tau_m,{d_m}}(x-c_{\tilde{T}_{\tau_m}})\sim \w_{\tau_m,{d_m}}(x-y)$ for any $y\in\tilde{T}_{\tau_m}$, and we have the pointwise relation $\w_{\tau_m,{d_m}}*\w_{\tau_m,{d_m}}\lesssim \w_{\tau_m,{d_m}}$. 
Then 
\[
    |\sum_{\tau_m\subset\tau}(f_{\tau_m}^m(x)-f_{\tau_m}^{m+1}(x))|\lesssim_{d_m} K^{-3}A^{-(N-m+1)}\frac{\a}{\b}\sum_{\tau_m\subset\tau}|f_{\tau_m}^{m+1}|^2*\w_{\tau_m,{d_m}}(x)\sim K^{-3}A^{-(N-m+1)}\frac{\a}{\b}g_m(x). \]
We choose $A$ sufficiently large, determined by the proof of Corollary \ref{highdom} and the proof of Proposition \ref{algo} (where we choose $d_0=d_0(\e)$), so that if $g_m(x)\le A^{N-m}\b$, then the above inequality implies that
\[ |\sum_{\tau_m\subset\tau}(f_{\tau_m}^m(x)-f_{\tau_m}^{m+1}(x))|\le  \e K^{-3}A^{-1/2}\a  .\]
This finishes the proof since the number of terms in \eqref{diffs} is bounded by $N\le \e^{-1}$. The argument for the pruning on $L$ is analogous.  
\end{proof}

\subsection{High-frequency analysis }

Now that we have identified moment curve blocks with cone planks and moment curve wave envelopes with cone wave envelopes, we are prepared to use a square function estimate for a cone.

\begin{lemma}[High lemma]\label{high1} Let $4\le p\le p_{n-1}$. For each $\d>0$, there is $B_\d\in(0,\infty)$ so that the following holds. We have 
\[ \int_{\R^n}|g_k^h|^{\frac{p}{2}}\lesssim_\e R^{2\e} \int_{\R^{n}}(\sum_{\tau_k}||f_{\tau_k}^{k+1}|^2*\w_{\tau_k,d_k}*\widecheck{\eta}_{>R_{k+1}^{-1/n}}|^2)^{\frac{p}{4}}.\]
\end{lemma}

\begin{proof} First describe the Fourier support of $g_k^h$. By \eqref{item3} of Lemma \ref{pruneprop}, the support of $\widehat{|f_{\tau_k}^{k+1}|^2}$ is $2(\tau_k-\tau_k)$. The high-frequency cutoff removes a ball of radius $R_{k+1}^{-1/n}$, so $g_k^h$ is Fourier supported within the annulus $R_{k+1}^{-1/n}\le |\xi|\le 10 R_k^{-1/n}$. By dyadic pigeonholing, there is some dyadic $s\in[R_{k+1}^{-1/n},2R_k^{-1/n}]$ for which 
\[ \int|g_k^h|^4\lesssim(\log R)  \int|g_k^h*\widecheck{\eta}_s|^4, \]
where $\eta_s:\R^n\to[0,\infty)$ is a smooth function supported in the annulus $s/4\le |\xi|\le s$. In the proof of Lemma \ref{low}, we showed the pointwise equality
\begin{equation}\label{ptwiselo}
g_k^h*\widecheck{\eta}_s(x)=\sum_{\tau_s}\sum_{\tau_s'\sim\tau_s}(f_{\tau_s}^{k+1}\overline{f_{\tau_s'}^{k+1}})*\w_{\tau_k,d_k}*\widecheck{\eta}_{>R_{k+1}^{-1/n}}*\widecheck{\eta}_s(x)
\end{equation}
where $\tau_s\in\Theta^n(s^{-n})$ and $\tau_s'\sim\tau_s$ means that $\tau_s'\in\Theta^n(s^{-n})$ and $\text{dist}(2\tau_s,2\tau_s')\le  2s$. For each $\tau_s$, the sub-sum on the right hand side has Fourier transform supported in $2(\tau_{2s}-\tau_{2s})\setminus B_s(0)$ where $\tau_{2s}\in\Theta^n((2s)^{-n})$ contains $\tau_s$. Now write 
\begin{equation}\label{formg} 
g_k^h*\widecheck{\eta}_s(x)=|\det T|^{-1}(\widehat{g_k^h}\eta_s\circ T^{-1})^{\widecheck{\,\,\,}}((T^{-1})^t x) , 
\end{equation}
where $T$ is an affine transformation mapping $\supp\widehat{g_k^h}{\eta}_s$ to $\Gamma_0^n(s^{-n})$. It suffices to bound 
\[ \int|(\widehat{g_k^h}\eta_s\circ T^{-1})^{\widecheck{\,\,\,}}(x)|^{\frac{p}{2}}dx. \]
We may view the sub-sums corresponding to each $\tau_k$ on the right hand side of \eqref{ptwiselo} as having Fourier support contained in an element of $\Xi_0^n(s^{-n})$, dilated by a factor of $s$. Therefore, since $\C_0^n(s^{-1})\lesssim_\e R^\e$, we have 
\begin{align*}
\int|(\widehat{g_k^h}\eta_s*\circ T^{-1})^{\widecheck{\,\,\,}}&(x)|^{\frac{p}{2}}dx\lesssim_\e R^{C\e}\int(\sum_{\tau_k}||\sum_{\tau_s\subset\tau_{2s}}\sum_{\tau_s'\sim\tau_s}(\widehat{|f_{\tau_k}^{k+1}|^2}\widehat{\w}_{\tau_k,d_k}{\eta}_{>R_{k+1}^{-1/n}}{\eta}_s\circ T^{-1})^{\widecheck{\,\,\,}}(x)|^2)^{\frac{p}{4}}. 
\end{align*}
After undoing the change of variables, the lemma is proved.

\end{proof}

\section{Key iterations that unwind the pruning process\label{keyalgo}}

In the process of unwinding the pruning process, we will encounter expressions of the form 
\[ \int_{\R^n}(\sum_{\tau\in\Theta^n(s^{-n})}|f_\tau|^{\tilde{p}_l}*\w_{\tau,d})^{\frac{p}{\tilde{p}_l}} \]
for various values of $d$ and $R^{-\frac{1}{n}}<s<1$, and where $\frac{p}{\tilde{p}_l}\ge 1$. To analyze this expression, we further decompose the Fourier support of each summand, which is contained in (a constant dilate of) $\tau-\tau$, or the set 
\[ \{\sum_{i=1}^{n}\lambda_i\g_n^{(i)}(a):|\lambda_i|\le s^{i}\,\,\forall i\}, \]
where $a$ is the initial point of $I(\tau')$. Fix $s<\sigma_0<1$. Decompose the above set according to the dyadic parameter $\sigma_0<\sigma<1$ into subsets 
\begin{equation}\label{summands} \{\sum_{i=1}^{n}\lambda_i \g_n^{(i)}(a): |\lambda_i|\le \min(1,\sigma^{l+1-i})s^{i}\quad\forall i\}\setminus\{\sum_{i=1}^{n}\lambda_i \g_n^{(i)}(a): |\lambda_i|\le \min(1,(\sigma/2)^{l+1-i})s^{i}\quad\forall i\}   \end{equation}
and $\{\sum_{i=1}^{n}\lambda_i \g_n^{(i)}(a): |\lambda_i|\le \min(1,\sigma_0^{l+1-i})s^{i}\quad\forall i\}$. We carry out this decomposition using smooth bump functions $\eta_{\tau,\sigma}^l$ which are part of a partition of unity and constructed analogously to those introduced before Lemma \ref{multilem1}. The function $\eta_{\tau,\sigma}^l$ is supported in 
\[ \Big(\{\sum_{i=1}^{l}\lambda_i \g_l^{(i)}(a): |\lambda_i|\le \sigma^{l+1-i}s^{i}\,\,\forall i\}\setminus\{\sum_{i=1}^{l}\lambda_i \g_l^{(i)}(a): |\lambda_i|\le (\sigma/2)^{l+1-i}s^{i}\,\,\forall i\}\Big)\times[-1,1]^{n-l}  \] 
if $\sigma_0< \sigma\le 1$ and supported in  
\[ \{\sum_{i=1}^{l}\lambda_i \g_l^{(i)}(a): |\lambda_i|\le \sigma_0^{l+1-i}s^{i}\quad\forall i\}\times[-1,1]^{n-l} \]
if $\sigma=\sigma_0$. Write $\tilde{W}_{\tau,\sigma}^l$ for the $L^1$-normalized weight function centered at the origin that is Fourier supported in $\{\sum_{i=1}^{l}\lambda_i \g_l^{(i)}(a): |\lambda_i|\le \sigma^{l+1-i}s^{i}\,\,\forall i\}\times[-1,1]^{n-l}$.

\subsection{Incorporating square function estimates for lower dimensional moment curves. }

We use the following cylindrical, pointwise version of square function estimates for lower dimensional moment curves. 



\begin{lemma}\label{algo1} Suppose that $\mb{M}^{m_1}(R)\lesssim_\d R^\d$ for all $R\ge 2$. Let $R_k^{-\frac{1}{n}}<s\le R_{k-1}^{-\frac{1}{n}}$. For each $0<\sigma<1$, and $\tau'\in{\bf{S}}_n(s)$, let $\tilde{W}_{\tau',\sigma}^{m_1}$ be a weight function that is Fourier supported in $\{\sum_{i=1}^{m_1}\lambda_i\g_{m_1}^{(i)}(a):|\lambda_i|\le \sigma^{m_1+1-i}\quad\forall i\}\times[-1,1]^{n-m_1}$, where $a$ is the initial point of $I(\tau')$. For any $s\le r\le 1$, $\tau\in{\bf{S}}(r)$ with $\tau'\subset\tau$, and $2\le p\le p_{m_1}\le p_n$, we have 
\begin{align*}
|f_{\tau'}^{k}|^{p}*\w_{\tau,d}*\tilde{W}_{\tau',\sigma}^{m_1} \lesssim_{\d,\e} R^\d  (\sum_{\substack{\tau''\subset\tau'\\\tau''\in{\bf{S}}_n(\sigma^{\frac{1}{m_1}}s)}}|f_{\tau''}^{k_m}|^2*\w_{\tau,2p_n^{-1}d})^{\frac{p}{2}}* \tilde{W}_{\tau,\sigma}^{m_1} 
\end{align*}
where $k_m\ge k$ satisfies $R_{k_m}^{-\frac{1}{n}}< \sigma^{\frac{1}{m_1}}s\le R_{k_m-1}^{-\frac{1}{n}}$. 
\end{lemma}

\begin{proof}[Proof of Lemma \ref{algo1}]
The $k=k_m$ case follows easily from the argument when $k<k_m$, so assume that $k<k_m$. By the same argument as was used in the proof of Lemma \ref{ptwise} (using moment curve rescaling in place of cone rescaling and invoking the definition of $\mb{M}^{m_1}(\cdot)$ in place of $\C_m^{n+k}(\cdot)$), we have
\[ |f_{\tau'}^{k}|^{p}*\w_{\tau,d}*\tilde{W}_{\tau',\sigma}^{m_1} \lesssim_{\d,\e}R^{\d\e}  (\sum_{\substack{\tau''\subset\tau'\\\tau''\in{\bf{S}}_n(R_k^{-\frac{1}{n}})}}|f_{\tau''}^{k}|^2)^{\frac{p}{2}}*\w_{\tau,d} *\tilde{W}_{\tau',\sigma}^{m_1} 
. \]
By Lemma \ref{pruneprop}, the right hand side is bounded above by $C_{\d,\e}R^{\d\e}  (\sum_{\substack{\tau''\subset\tau'\\\tau''\in{\bf{S}}_n(R_k^{-\frac{1}{n}})}}|f_{\tau''}^{k+1}|^2)^{\frac{p}{2}}*\w_{\tau,d} *\tilde{W}_{\tau',\sigma}^{m_1}$. By Khintchine's inequality, there is a choice of signs $e_{\tau''}\in\{\pm1\}$ (permitted to depend on the point we are evaluating the two-fold convolution at) which satisfy 
\[ (\sum_{\substack{\tau''\subset\tau'\\\tau''\in{\bf{S}}_n(R_k^{-\frac{1}{n}})}}|f_{\tau''}^{k+1}|^2)^{\frac{p}{2}}*\w_{\tau,d} *\tilde{W}_{\tau',\sigma}^{m_1} 
\sim |\sum_{\substack{\tau''\subset\tau'\\\tau''\in{\bf{S}}_n(R_k^{-\frac{1}{n}})}}e_{\tau''}f_{\tau''}^{k+1}|^{p}*\w_{\tau,d} *\tilde{W}_{\tau',\sigma}^{m_1} .
\]
Then use a cylindrical version of $\mb{M}^{m_1}(\cdot)$ again to bound the right hand side above by 
\[ C_{\d,\e}R^{\d\e}(\sum_{\substack{\tau''\subset\tau'\\\tau''\in{\bf{S}}_n(R_{k+1}^{-\frac{1}{n}})}}|f_{\tau''}^{k+1}|^2)^{\frac{p}{2}}*\w_{\tau,d} *\tilde{W}_{\tau',\sigma}^{m_1} .
\]
Iterate this process (which is at most $\e^{-1}$ many steps) until we obtain the inequality
\[ |f_{\tau'}^{k}|^{p}*\w_{\tau,d}*\tilde{W}_{\tau',\sigma}^{m_1} \lesssim_{\d,\e}R^{\d}  (\sum_{\substack{\tau''\subset\tau'\\\tau''\in{\bf{S}}_n(\sigma^{\frac{1}{m_1}}s)}}|f_{\tau''}^{k_m}|^2)^{\frac{p}{2}}*\w_{\tau,d} *\tilde{W}_{\tau',\sigma}^{m_1} 
. \]

Then, by the locally constant property, we have
\[ (\sum_{\substack{\tau''\subset\tau'\\\tau''\in{\bf{S}}_n(\sigma^{\frac{1}{m_1}}s)}}|f_{\tau''}^{k_m}|^2)^{\frac{p}{2}}*\w_{\tau,d} *\tilde{W}_{\tau',\sigma}^{m_1}\lesssim (\sum_{\substack{\tau''\subset\tau'\\\tau''\in{\bf{S}}_n(\sigma^{\frac{1}{m_1}}s)}}|f_{\tau''}^{k_m}|^2*|\widecheck{\rho}_\tau|)^{\frac{p}{2}}*\w_{\tau,d} *\tilde{W}_{\tau,\sigma}^{m_1}  \]
where $\rho_\tau$ is a bump function localized to $\tau$, which contains all of the $\tau''\subset\tau'$. It suffices to show that 
\[ (\sum_{\substack{\tau''\subset\tau'\\\tau''\in{\bf{S}}_n(\sigma^{\frac{1}{m_1}}s)}}|f_{\tau''}^{k_m}|^2*|\widecheck{\rho}_\tau|)^{\frac{p}{2}}*\w_{\tau,d} \lesssim (\sum_{\substack{\tau''\subset\tau'\\\tau''\in{\bf{S}}_n(\sigma^{\frac{1}{m_1}}s)}}|f_{\tau''}^{k_m}|^2*\w_{\tau,p_n^{-1}d})^{\frac{p}{2}} .\] 
Let $T\|\tau^*$ denote a tiling of $\R^n$ by translates of $\tau^*$. Then, using $\frac{p}{2}\ge 1$ and $\|\cdot\|_{\ell^{p/2}}\le\|\cdot\|_1$, we have
\begin{align*}
(\sum_{\substack{\tau''\subset\tau'\\\tau''\in{\bf{S}}_n(\sigma^{\frac{1}{m_1}}s)}}|f_{\tau''}^{k_m}|^2&*|\widecheck{\rho}_\tau|)^{\frac{p}{2}}*\w_{\tau,d}(x)\le \sum_{T\|\tau^*}\|\sum_{\substack{\tau''\subset\tau'\\\tau''\in{\bf{S}}_n(\sigma^{\frac{1}{m_1}}s)}}|f_{\tau''}^{k_m}|^2*|\widecheck{\rho}_\tau|(x-y)\|_{L^\infty_y(T)}^{\frac{p}{2}} \|\w_{\tau,d}\|_{L^\infty(T)}|T| \\
    &\le \Big(\int_{\R^n}\sum_{\substack{\tau''\subset\tau'\\\tau''\in{\bf{S}}_n(\sigma^{\frac{1}{m_1}}s)}}|f_{\tau''}^{k_m}|^2(z)\sum_{T\|\tau^*}\||\widecheck{\rho}_\tau|(x-y-z)\|_{L^\infty_y(T)} \|\w_{\tau,d}\|_{L^\infty(T)}^{\frac{2}{p}}|T|^{\frac{2}{p}} dz\Big)^{\frac{p}{2}} \\
    &\lesssim_d \Big(\int_{\R^n}\sum_{\substack{\tau''\subset\tau'\\\tau''\in{\bf{S}}_n(\sigma^{\frac{1}{m_1}}s)}}|f_{\tau''}^{k_m}|^2(z)\w_{\tau,2p^{-1}d}*\w_{\tau,2p^{-1}d}(x-z)  dz\Big)^{\frac{p}{2}}\\
    &\lesssim_d \Big(\int_{\R^n}\sum_{\substack{\tau''\subset\tau'\\\tau''\in{\bf{S}}_n(\sigma^{\frac{1}{m_1}}s)}}|f_{\tau''}^{k_m}|^2*\w_{\tau,2p_n^{-1}d}(x)\Big)^{\frac{p}{2}}, 
\end{align*}
as desired.

\end{proof}

\subsection{Auxiliary estimates related to Taylor cones. }

\begin{lemma}\label{algo2} Assume that $\M^{n'}(R)\lesssim_\d R^\d$ and $\C^{n'+1}_{m'}(R)\lesssim_\d R^\d$ for all $n'<n$, $0\le m'\le n'-1$, $R\ge 1$, and $\d>0$. Suppose that $2\le \tilde{p}_l\le p\le p_n$ and $2\le \frac{p}{\tilde{p}_l}$. For $R_{k}^{-\frac{1}{n}}< s\le R_{k-1}^{-\frac{1}{n}}$, there exists some dyadic $s'$, $R^{-\frac{1}{n}}\le s'\le s$ such that 
\begin{align}
\int_{\R^n}(&\sum_{\substack{\tau'\in{\bf{S}}_n(s)}}|f_{\tau'}^{k}|^{\tilde{p}_l}*\w_{\tau',d})^{\frac{p}{\tilde{p}_l}} \le (C\log R)^{C}B_{\d}R^{\d}
\label{LHS}\\
&\times \left[  \int_{\R^n}(\sum_{\substack{\tau'\in{\bf{S}}_n(s)}}\sum_{\substack{\tau\subset\tau'\\\tau\in{\bf{S}}_n(\max(R^{-\frac{1}{n}},R^{-\frac{\e^3}{n}}s))}} |f_{\tau}^{k_m}|^2*\w_{\tau',p_n^{-1}d})^{\frac{p}{2}} +R^{C\e^3}\int_{\R^n}(\sum_{\substack{\tau'\in{\bf{S}}_n(s)}}|f_{\tau'}^{k}|^{\tilde{p}_{l+1}}*\w_{\tau',d})^{\frac{p}{\tilde{p}_{l+1}}} \right] \nonumber
\end{align}
where $k_m\ge k$ satisfies $R_{k_m}^{-\frac{1}{n}}< \max(R^{-\frac{1}{n}},R^{-\frac{\e^3}{n}}s)\le R_{k_m-1}^{-\frac{1}{n}}$.
\end{lemma}
\begin{proof}[Proof of Lemma \ref{algo2}]  We use the bump functions 
 $\eta_{\tau',\sigma}^{n-1}$ introduced at the beginning of \textsection\ref{keyalgo} to decompose the Fourier support of the integrand. Using dyadic pigeonholing, let $\sigma$ be a dyadic parameter, $R^{-\e^3}\le \sigma\le 1$, which satisfies  
\[ \int_{\R^n}(\sum_{\substack{\tau'\in{\bf{S}}_n(s)}}|f_{\tau'}^{k}|^{\tilde{p}_{l}}*\w_{\tau',d})^{\frac{p}{\tilde{p}_{l}}}\lesssim (\log R)^C\int_{\R^n}|\sum_{\substack{\tau'\in{\bf{S}}_n(s)}}|f_{\tau'}^k|^{\tilde{p}_{l}}*\w_{\tau,d}*\widecheck{\eta}_{\tau',\sigma}^{l}|^{\frac{p}{\tilde{p}_{l}}}. \]
If $\sigma= R^{-\e^3}$, then by Lemma \ref{algo1}, 
\[ \int_{\R^n}|\sum_{\substack{\tau'\in{\bf{S}}_n(s)}}|f_{\tau'}^{k}|^{\tilde{p}_{l}}*\w_{\tau',d}*\widecheck{\eta}_{\tau',\sigma}^{l}|^{\frac{p}{\tilde{p}_{l}}}\lesssim_\d R^{\d} \int_{\R^n}|\sum_{\substack{\tau'\in{\bf{S}}_n(s)}}|\sum_{\substack{\tau\subset\tau'\\ \tau\in{\bf{S}}_n(R^{-\frac{\e^3}{n}}s)}} |f_{\tau}^{k_m}|^2*\w_{\tau',p_n^{-1}d}|^{\frac{\tilde{p}_{l}}{2}}*\tilde{W}_{\tau',\sigma}^{l}|^{\frac{p}{\tilde{p}_{l}}}.  \]
Apply Proposition \ref{multilem3pf} to eliminate the weights $\tilde{W}_{\tau,\sigma}^{n-1}$, producing the upper bound
\[ D_\d R^\d\int_{\R^n}|\sum_{\substack{\tau'\in{\bf{S}}_n(s)}}|\sum_{\substack{\tau\subset\tau'\\ \tau\in{\bf{S}}_n(R^{-\frac{\e^3}{n}}s)}} |f_{\tau}^{k_m}|^2*\w_{\tau',p_n^{-1}d}|^{\frac{\tilde{p}_{l}}{2}}|^{\frac{p}{\tilde{p}_{l}}},  \]
which, since $\|\cdot\|_{\ell^{\tilde{p}_l/2}}\lesssim 1$, is bounded by 
\[ D_\d R^\d\int_{\R^n}\big(\sum_{\substack{\tau'\in{\bf{S}}_n(s)}}\sum_{\substack{\tau\subset \tau'\\ \tau\in{\bf{S}}_n(R^{-\frac{\e^3}{n}}s)}} |f_{\tau}^{k_m}|^2*\w_{\tau',p_n^{-1}d} \big)^{\frac{p}{2}}.  \]

It remains to consider the case that $\sigma>R^{-\e^3}$. The Fourier supports of the summands are contained in sets \eqref{summands}, which are a fixed dilate of elements of $\Xi_{l-1}^{n}(\sigma^{-\frac{n-1}{n}}s^{-\frac{n-1}{n}})$. Since we assumed that $2\le \frac{p}{\tilde{p}_l}$, by Lemma \ref{pprops} we have $2\le \frac{p}{\tilde{p}_l}\le p_{n-l}$. Therefore, by the hypothesized boundedness of $\C_{l-1}^n(\cdot)$, we have
\[ \int_{\R^n}|\sum_{\substack{\tau'\in{\bf{S}}_n(s)}}|f_{\tau'}^{k}|^{\tilde{p}_{l}}*\w_{\tau',d}*\widecheck{\eta}_{\tau',\sigma}^{l}|^{\frac{p}{\tilde{p}_{l}}}\lesssim_\d R^\d \int_{\R^n}|\sum \sum_{\tau\in{\bf{S}}_n(\sigma^{-1}s)}|\sum_{\substack{\tau'\subset\tau \\
\tau' \in{\bf{S}}_n(s)}}|f_{\tau'}^{k}|^{\tilde{p}_{l}}*\w_{\tau',d}*\widecheck{\eta}_{\tau',\sigma}^{l}|^2|^{\frac{p}{2\tilde{p}_{l}}}.  \]
Since $\sigma>R^{-\e^3}$, we have the pointwise inequality $\w_{\tau',d}*|\widecheck{\eta}_{\tau',\sigma}^l|\lesssim_\e R^{C\e^3}\w_{\tau',d}$. Then by Cauchy-Schwarz, the right hand side above is bounded by 
\[ D_\d R^\d R^{C\e^3}\int_{\R^n}\big(\sum_{\substack{
\tau' \in{\bf{S}}_n(s)}}||f_{\tau'}^{k}|^{\tilde{p}_{l}}*\w_{\tau',d}|^2\big)^{\frac{p}{2\tilde{p}_{l}}}.   \]
Use that $1\le \frac{\tilde{p}_{l+1}}{\tilde{p}_l}\le 2$ and the pointwise inequality $||f_{\tau'}^k|^{\tilde{p}_l}*\w_{\tau',d}|^{\frac{\tilde{p}_{l+1}}{\tilde{p}_l}}\lesssim |f_{\tau'}^k|^{\tilde{p}_l}*\w_{\tau',d}$ to bound the right hand side above by 
\[D_\d R^\d R^{C\e^3}\int_{\R^n}\big(\sum_{\substack{
\tau' \in{\bf{S}}_n(s)}}|f_{\tau'}^{k}|^{\tilde{p}_{l+1}}*\w_{\tau',d}\big)^{\frac{p}{\tilde{p}_{l+1}}}, \] 
as desired. 

\end{proof}

\subsection{Algorithm to fully unwind the pruning process \label{hisec}}


\begin{proposition}\label{algo} Let $D\in\N_{>0}$. There exists $d_0=d_0(n,\e,D)$ such that the following holds. Suppose that $\M^{m_1}(r)\lesssim_\d r^\d$ for all $1\le m_1<n$. For each $k$ and each $2\le p\le p_n$, we have 
\begin{align}
\label{algline1}
&\int_{\R^n}|\sum_{\tau_k}|f_{\tau_k}^{k+1}|^{2}*\w_{\tau_k,d_k}|^{\frac{p}{2}}\le C_\e R^{3\e^2}\emph{T}_n(R^\e)^{N-k+1}\int_{\R^n}
|\sum_\theta|f_\theta|^2*\w_{\theta,D}|^{\frac{p}{2}}.
\end{align}
\end{proposition}

Propositon \ref{algo} will follow from an algorithm which uses Lemmas \ref{algo1} and \ref{algo2} as building blocks. 

\begin{proof}[Proof of Proposition \ref{algo}] We will define an algorithm which at intermediate step $m$, produces an inequality 
\begin{align}\label{stepm} 
(\text{L.H.S. of \eqref{algline1}})\le (C_\e\log R)^{4\e^{-1}m}&(R^{\e^7})^{m}(R^{C\e^3}\text{T}_n(R^\e))^{k_m-k}  \int_{\R^n}(\sum_{\substack{\tau\in{\bf{S}}_n(s_m)}}|f_{\tau}^{k_m}|^{2}*\w_{\tau,p_n^{-m}d_k})^{\frac{p}{2}} 
\end{align}
in which $0\le a\le m$ and $R^{-\frac{1}{n}}\le s_m\le \max(R_k^{-\frac{1}{n}}R^{-\frac{\e^3}{n}a},R_{k_{m-1}}^{-\frac{1}{n}})$ and $N\ge k_m\ge k+1$ satisfies $R_{k_m}^{-\frac{1}{n}}< s_m\le R_{k_m-1}^{-\frac{1}{n}}$. Notice that \eqref{stepm} clearly holds with $m=0$, taking $k_m=k+1$ and $s_m=R_k^{-\frac{1}{n}}$. Assuming \eqref{stepm} holds for $m-1$, we will show that either the algorithm terminates and the proposition is proved or \eqref{stepm} holds for $m\ge 1$. 

\noindent\fbox{Step m:} Suppose that \eqref{stepm} holds with $m-1$, so the left hand side of \eqref{algline1} is bounded by
\begin{align*} 
(C_\e\log &R)^{4\e^{-1}(m-1)}(R^{\e^7})^{m-1}(R^{\e^3}\text{T}_n(R^\e))^{k_{m-1}-k}  \int_{\R^n}(\sum_{\substack{\tau\in{\bf{S}}_n(s_{m-1})}}|f_{\tau}^{k_{m-1}}|^{2}*\w_{\tau,p_n^{-(m-1)}d_k})^{\frac{p}{2}} \end{align*}
in which $0\le a\le m-1$ and $R^{-\frac{1}{n}}\le s_{m-1}\le \max(R_k^{-\frac{1}{n}}R^{-\frac{\e^3}{n}a},R_{k_{m-1}}^{-\frac{1}{n}})$ and $N\ge k_{m-1}\ge k+1$ satisfies $R_{k_{m-1}}^{-\frac{1}{n}}< s_{m-1}\le R_{k_{m-1}-1}^{-\frac{1}{n}}$. Apply Lemma \ref{algo2} to the integral, yielding the inequality
\begin{align}\label{stephyp} 
&(\text{L.H.S. of \eqref{algline1}})\le (C_\e\log R)^{4\e^{-1}m}(R^{\e^7})^{m-1+\frac{1}{n}}(R^{C\e^3}\text{T}_n(R^\e))^{k_{m-1}-k} \\ 
&\times \left[  \int_{\R^n}(\sum_{\substack{\tau'\in{\bf{S}}_n(R^{-\frac{\e^3}{n}}s_{m-1})}}|f_{\tau'}^{k_m}|^2*\w_{\tau',p_n^{-m}d_k})^{\frac{p}{2}} +R^{C\e^3}\int_{\R^n}(\sum_{\substack{\tau\in{\bf{S}}_n(s_{m-1})}} |f_{\tau}^{k_{m-1}}|^{\tilde{p}_{2}}*\w_{\tau,p_n^{-(m-1)}d_k})^{\frac{p}{\tilde{p}_{2}}}\right] \nonumber
\end{align}
where $R_{k_m}^{-\frac{1}{n}}<R^{-\frac{\e^3}{n}}s_{m-1}\le R_{k_m-1}^{-\frac{1}{n}}$, $N\ge k_m\ge k+1$. If the first term on the right hand side dominates, then the inner loop terminates, producing the step $m$ inequality 
\[ (\text{L.H.S. of \eqref{algline1}})\le (C_\e\log R)^{4\e^{-1}m}(R^{\e^7})^{m}(R^{\e^3}\text{\text{T}}_n(R^\e))^{k_{m-1}-k}  \int_{\R^n}(\sum_{\substack{\tau\in{\bf{S}}_n(R^{-\frac{\e^3}{n}}s_{m-1})}}|f_{\tau}^{k_m}|^2*\w_{\tau,p_n^{-m}d_k})^{\frac{p}{2}} . \]
If the second term dominates, we have the following upper bound for the left hand side of \eqref{algline1}:
\[(C_\e\log R)^{4\e^{-1}m}(R^{\e^7})^{m-1+\frac{1}{n}}(R^{C\e^3}\text{T}_n(R^\e))^{k_{m-1}-k} R^{n^{-1}C\e^3}\int_{\R^n}(\sum_{\substack{\tau\in{\bf{S}}_n(s_{m-1})}}|f_{\tau}^{k_{m-1}}|^{\tilde{p}_{2}}*\w_{\tau,p_n^{-(m-1)}d_k})^{\frac{p}{\tilde{p}_{2}}} . \]
Note that we may again apply Lemma \ref{algo2} to the integral above. We iterate this process until either the inner loop terminates with proving step $m$ or we have shown that the left hand side of \eqref{algline1} is bounded by 
\[ (C_\e\log R)^{4\e^{-1}m}(R^{\e^7})^{m-1+\frac{l}{n}}(R^{C\e^3}\text{T}_n(R^\e))^{k_{m-1}-k} R^{ln^{-1}C\e^3}\int_{\R^n}(\sum_{\substack{\tau\in{\bf{S}}_n(s_{m-1})}}|f_{\tau}^{k_{m-1}}|^{\tilde{p}_{l}}*\w_{\tau,p_n^{-(m-1)}d_k})^{\frac{p}{\tilde{p}_{l}}} \]
with $1\le \frac{p}{\tilde{p}_l}\le 2$ and $2\le \frac{p}{\tilde{p}_{l-1}}$, which means that $l\le n$. As in the proof of Lemma \ref{algo2}, use bump functions $\eta_{\tau,\sigma}^{n-1}$ and suppose that 
\[ \int_{\R^n}(\sum_{\substack{\tau\in{\bf{S}}_n(s_{m-1})}}|f_{\tau}^{k_{m-1}}|^{\tilde{p}_{l}}*\w_{\tau,p_n^{-(m-1)}d_k})^{\frac{p}{\tilde{p}_{l}}}\lesssim (\log R)^C\int_{\R^n}|\sum_{\substack{\tau\in{\bf{S}}_n(s_{m-1})}}|f_{\tau}^{k_{m-1}}|^{\tilde{p}_{l}}*\w_{\tau,p_n^{-(m-1)}d_k}*\widecheck{\eta}_{\tau,\sigma}^{n-1}|^{\frac{p}{\tilde{p}_{l}}}. \]
If $\sigma= R^{-\e^3}s_{m-1}$, then by Lemma \ref{algo1}, 
\begin{align*} 
\int_{\R^n}|\sum_{\substack{\tau\in{\bf{S}}_n(s_{m-1})}}&|f_{\tau}^{k_{m-1}}|^{\tilde{p}_{l}}*\w_{\tau,p_n^{-(m-1)}d_k}*\widecheck{\eta}_{\tau,\sigma}^{n-1}|^{\frac{p}{\tilde{p}_{l}}}\\
&\lesssim C_\e R^{n^{-1}\e^7} \int_{\R^n}|\sum_{\substack{\tau\in{\bf{S}}_n(s_{m-1})}} |\sum_{\substack{\tau'\subset\tau\\ \tau'\in{\bf{S}}_n(R^{-\frac{\e^3}{n}}s_{m-1})}} |f_{\tau'}^{k_m}|^2*\w_{\tau,p_n^{-m}d_k}|^{\frac{\tilde{p}_{l}}{2}}*\tilde{W}_{\tau,\sigma}^{n-1}|^{\frac{p}{\tilde{p}_{l}}}.  \end{align*}
By the locally constant property and $\w_{\tau',p_n^{-m}d_k}*\w_{\tau,2p_n^{-m}d_k}\lesssim\w_{\tau',p_n^{-m}d_k}$ whenever $\tau'\subset \tau$, so we have $|f_{\tau'}^{k_m}|^2*\w_{\tau,2p_n^{-1}d_k}\lesssim |f_{\tau'}^{k_m}|^2*\w_{\tau',p_n^{-m}d_k}$ for each $\tau'$. Finally, apply Proposition \ref{multilem3pf} to eliminate $\tilde{W}_{\tau,\sigma}^{n-1}$, concluding step $m$.

The other case is that $\sigma>R^{-\e^3}s_{m-1}$. Using the boundedness of $\C^n_{n-2}(\cdot)$ and Cauchy-Schwarz, we have
\[ \int_{\R^n}|\sum_{\substack{\tau\in\Theta^n(s_{m-1}^{-n})}}|f_{\tau}^{k_{m-1}}|^{\tilde{p}_{l}}*\w_{\tau,p_n^{-(m-1)}d_k}*\widecheck{\eta}_{\tau,\sigma}^{n-1}|^{\frac{p}{\tilde{p}_{l}}}\le C_\e R^{n^{-1}\e^7}\int_{\R^n}\sum_{\tau\in{\bf{S}}_n(s_{m-1})} ||f_{\tau}^{k_{m-1}}|^{\tilde{p}_{l}}*\w_{\tau,p_n^{-(m-1)}d_k}*\widecheck{\eta}_{\tau,\sigma}^{n-1}|^{\frac{p}{\tilde{p}_{l}}} .\]
By Young's convolution inequality, the right hand side is bounded by 
\[ C_\e R^{n^{-1}\e^7}\int_{\R^n}\sum_{\tau\in{\bf{S}}_n(s_{m-1})} |f_{\tau}^{k_{m-1}}|^{p} ,\]
which, by the definition of $\text{T}_n(\cdot)$ and rescaling, is itself bounded by 
\[ C_\e R^{n^{-1}\e^7}\text{T}_n(R^\e)\int_{\R^n}\sum_{\tau\in{\bf{S}}_b(s_{m-1})} (\sum_{\tau_{k_{m-1}}\subset\tau}|f_{\tau_{k_{m-1}}}^{k_{m-1}}|^2 )^{\frac{p}{2}} . \]
By Lemma \ref{pruneprop}, $\|\cdot\|_{\ell^{p/2}}\le \|\cdot\|_{\ell^1}$, and the locally constant property, this expression is bounded by 
\begin{align*} 
C_\e R^{n^{-1}\e^7}\text{T}_n(R^\e)\int_{\R^n}\sum_{\tau\in\Theta^n(s_{m-1}^{-n})} &(\sum_{\tau_{k_{m-1}}\subset\tau}|f_{\tau_{k_{m-1}}}^{k_{m-1}+1}|^2)^{\frac{p}{2}}\\
&\le C_\e R^{n^{-1}\e^7}\text{T}_n(R^\e)\int_{\R^n}(\sum_{\tau\in\Theta^n(s_{m-1}^{-n})} |f_{\tau_{k_{m-1}}}^{k_{m-1}+1}|^2*\w_{\tau_{k_{m-1}},p_n^{-m}d_k})^{\frac{p}{2}}.  
\end{align*}
Taking $k_m=k_{m-1}+1$, this concludes the justification of step $m$.

\noindent\fbox{Termination criteria.} By taking $m$ large enough so that $s_m=R^{-\frac{1}{n}}$, the algorithm terminates with the inequality
\begin{equation}\label{term1} (\text{L.H.S. of \eqref{algline1}})\le (C_\e\log R)^{4\e^{-1}m}(R^{\e^7})^{m+1}(R^{C\e^3}\text{T}_n(R^\e))^{k_m-k} \int_{\R^n}(\sum_{\theta\in{\bf{S}}_n(R^{-\frac{1}{n}})}|f_\theta|^2*\w_{\theta,p_n^{-m}d_k})^{\frac{p}{2}} \end{equation}
where $0\le a<m$ and $R^{-\frac{1}{n}}= s_m\le \max(R_k^{-\frac{1}{n}}R^{-\frac{\e^3}{n}a},R_{k_{m-1}}^{-\frac{1}{n}})$. Then $R^{-\frac{1}{n}}\le R_k^{-\frac{1}{n}}R^{-\frac{\e^3}{n}a}$ implies that $a\le\e^{-3}(N-k)$ and $N\ge k_m\ge k+1$ implies that $(R^{C\e^3}\text{T}_n(R^\e))^{k_m-k}\le(R^{C\e^3}\text{T}_n(R^\e))^{N-k}$. The total number of steps $m$ is bounded by $2\e^{-3}$ since each step either refines $s_m$ by a factor of $R^{-\frac{\e^3}{n}}$ or replaces $s_m$ by $R_{k_m}^{-\frac{1}{n}}$, with $k_m>k_{m-1}$. Thus the constants in the upper bound from \eqref{term1} are bounded by 
\begin{align*} 
(C_e\log R)^{4\e^{-1}m+2\e^{-1}+1}&(R^{\e^7})^{m-1}(R^{C\e^3}\text{T}_n(R^\e))^{N-k} \\
&\le (C_\e R^{\e^7})^{\e^{-5}}R^{\e^2}\text{T}_n(R^\e)^{N-k}\le C_\e R^{3\e^2}\text{T}_n(R^\e)^{N-k+1} .
\end{align*}
Finally, since $m\le 2\e^{-3}$, it suffices to choose $d_0$ (from the definition of $d_k$) satisfying $p_n^{-2\e^{-3}}d_0\ge D$, so the proposition is proved.

\end{proof}

\section{Proof of Proposition \ref{momcurveinduct} \label{mainsec}}

We use the set-up for the high-low method to prove a broad estimate, Proposition \ref{mainprop}. Then, we prove Proposition \ref{S1bd}, which says that $\text{T}_{n,D}^w(R)\lesssim_\e R^\e$. This follows using various reductions from pigeonholing, a broad-narrow argument, and Proposition \ref{mainprop}. Finally, in \textsection\ref{S2}, we use induction to show that Proposition \ref{S1bd} implies that $\text{T}_n(R)\lesssim_\e R^\e$, which is equivalent to Proposition \ref{momcurveinduct}.

\subsection{Bounding the broad part of $U_{\a,\b}$ \label{broad} }

For $n$ canonical blocks $\tau^1,\ldots,\tau^n$ (with dimensions $\sim R^{-\e/n}\times R^{-2\e/n}\times \cdots\times R^{-\e}$) which are pairwise $\ge R^{-\e/n}$-separated, define the broad part of $U_{\a,\b}$ to be
\[ \text{Br}_{\a,\b}^K=\{x\in U_{\a,\b}: \a\le K|\prod_{i=1}^nf_{\tau^i}(x)|^{\frac{1}{n}},\quad\max_{\tau^i}|f_{\tau^i}(x)|\le \a\}. \]

We bound the broad part of $U_{\a,\b}$ in the following proposition. Recall that the parameter $N_0$ was used in the definition of the sets $\Omega_k$ and $L$. 

\begin{proposition}\label{mainprop} Let $R,K\ge1$ and $2\le p\le p_n$. Suppose that $\|f_\theta\|_{L^\infty(\R^n)}\le 2$ for all $\theta\in{\bf{S}}_n(R^{-\frac{1}{n}})$. Then  
\begin{equation*} 
\a^{p}|\text{\emph{Br}}_{\a,\b}^{K}|\le \big[CR^{10 \e N_0}A^{\e^{-1}} + K^{50}R^{4\e^2+10\e}A^{\e^{-1}}  \emph{T}_n(R^\e)^{\e^{-1}-N_0}\big]\int\Big|\sum_\theta|f_\theta|^2*\w_{\theta,D}\Big|^{\frac{p}{2}} .
\end{equation*} 
\end{proposition}

We will use the following version of a (well-known) local multilinear restriction inequality for the moment curve. The weight function $W_{B_r,d}$ is defined in Definition \ref{M3ballweight}. 
\begin{thm}\label{trirestprop}
Let $s\ge 10r\ge10$ and let $f:\R^n\to\C$ be a Schwartz function with Fourier transform supported in $\mc{N}_{r^{-1}}(\mc{M}^n)$, the $r^{-1}$-neighborhood of the moment curve in $\R^n$. Suppose that $\tau^1,\tau^2,\ldots,\tau^n\in{\bf{S}}_n(R^{-\e/n})$ satisfy  $\text{dist}(\tau^i,\tau^j)\ge s^{-1}$ for $i\not=j$. Then 
\[  \a^p |\emph{Br}_{\a,\b}^K\cap B_r| \lesssim_d  s^n|B_r| \prod_{i=1}^n\big(|B_r|^{-1}\int|f_{\tau^i}|^2W_{B_r,d}\big)^{\frac{p}{2n}} \]
for any $2\le p\le 2n$ and any Schwartz function $f:\R^n\to\C$ with Fourier transform supported in $\mc{M}^n(r)$. 
\end{thm}

\begin{proof} In the case that $p=2n$, 
\[\a^{2n}|\emph{Br}_{\a,\b}^K\cap B_r|\lesssim_d  s^n|B_r|^{-(n-1)}\prod_{i=1}^n\big(\int|f_{\tau^i}|^2W_{B_r,d}\big) \]
follows from a straightforward adaptation of the proof of Proposition 6 from \cite{M3smallcap} to $n$-dimensions. The $2\le p\le 2n$ case follows from rearranging the above inequality:
\begin{equation}\label{intp} \a^{p}|\emph{Br}_{\a,\b}^K\cap B_r|\lesssim_d  s^n|B_r|\left(\prod_{i=1}^n\big(|B_r|^{-1}\int|f_{\tau^i}|^2W_{B_r,d}\right)^{\frac{p}{2n}}\left[\frac{1}{\a}\prod_{i=1}^n\big(|B_r|^{-1}\int|f_{\tau^i}|^2W_{B_r,d}\big)^{\frac{1}{2n}}\right]^{2n-p}. \end{equation}
If 
\[ \prod_{i=1}^n\big(|B_r|^{-1}\int|f_{\tau^i}|^2W_{B_r,d}\big)^{\frac{1}{2n}}\le \a, \]
then the theorem follows from \eqref{intp}. If the above inequality does not hold, then we have
\begin{align*}
    \a^p|\text{Br}_{\a,\b}^K\cap B_r|\le |B_r| \prod_{i=1}^n\big(|B_r|^{-1}\int|f_{\tau^i}|^2W_{B_r,d}\big)^{\frac{p}{2n}},
\end{align*}
which proves the theorem directly. 

\end{proof}

\begin{proof}[Proof of Proposition \ref{mainprop}]
Fix $2\le p\le p_n$ and the decay rate $D\in\N_{>0}$ for the weights. Note that
\[ \text{Br}_{\a,\b}^K=(L\cap \text{Br}_{\a,\b}^K)\cup( \sqcup_{k=B}^{N-1}\Omega_k\cap \text{Br}_{\a,\b}^K)\]
We bound each of the sets $\text{Br}_{\a,\b}^K\cap\Omega_k$ and $\text{Br}_{\a,\b}^K\cap L$ in separate cases. It suffices to consider the case that $R$ is at least some constant depending on $\e$ since if $R\le C_\e$, we may prove the proposition using trivial inequalities. 

\vspace{2mm}
\noindent\underline{Case 1: bounding $| \text{Br}_{\a,\b}^{K}\cap\Omega_k|$. } 
By Lemma \ref{ftofk}, 
\[ |\text{Br}_{\a,\b}^K\cap\Omega_k|\le|\{x\in  U_{\a,\b}\cap\Omega_k:\a\lesssim K|\prod_{i=1}^nf_{\tau^i}^{k+1}(x)|^{\frac{1}{n}},\quad\max_{\tau^i}|f_{\tau^i}(x)|\le \a\}|.\]
By Lemma \ref{pruneprop}, the Fourier supports of the $f_{\tau^i}^{k+1}$ are contained in $2\tau^i$, which are pairwise $\ge5 R^{-\e}$-separated blocks of the moment curve. Let $\{B_{R_k^{\frac{1}{n}}}\}$ be a finitely overlapping cover of $\text{Br}_{\a,\b}^{K}\cap\Omega_k$ by $R_k^{\frac{1}{n}}$-balls. For $R$ large enough depending on $\e$, apply Theorem \ref{trirestprop} with $r:= \min(2n,p)$ to get
\begin{align*}
    \a^{r}|\text{Br}_{\a,\b}^K\cap B_{R_k^{\frac{1}{n}}}|&\lesssim_\e R^{\e} |B_{R_k^{\frac{1}{n}}}|\prod_{i=1}^n\Big(|B_{R_k^{\frac{1}{n}}}|^{-1}\int|f_{\tau^i}^{k+1}|^2W_{B_{R_k^{\frac{1} {n}}},d_k+1}\Big)^{\frac{r}{2n}}.
\end{align*}
Using local $L^2$-orthogonality (Lemma \ref{L2orth}), each integral on the right hand side above is bounded by 
\[ \lesssim \int\sum_{\tau_k}|f_{\tau_k}^{k+1}|^2W_{B_{R_k^{\frac{1}{n}},d_k+1}}. \]
If $x\in \text{Br}_{\a,\b}^{K}\cap\Omega_k\cap B_{R_k^{\frac{1}{n}}}$, then the above integral is bounded by 
\[ \lesssim  \int \sum_{\tau_k}|f_{\tau_k}^{k+1}|^2*\w_{\tau_k,d_k+1}W_{B_{R_k^{\frac{1}{n}}},d_k+1}\lesssim C |B_{R_k^{\frac{1}{n}}}| \sum_{\tau_k}|f_{\tau_k}^{k+1}|^2*\w_{\tau_k,d_k}(x) \]
by the locally constant property (Lemma \ref{locconst}) and properties of the weight functions. The summary of the inequalities so far is that 
\[ \a^{r}|\text{Br}_{\a,\b}^{K}\cap\Omega_k\cap B_{R_k^{\frac{1}{n}}}|\lesssim_\e R^\e K^{2n} |B_{R_k^{\frac{1}{n}}}|g_k(x)^{\frac{r}{2}} \]
where $x\in \text{Br}_{\a,\b}^{K}\cap\Omega_k\cap B_{R_k^{\frac{1}{n}}}$. 

Recall that since $x\in\Omega_k$, we have the lower bound $A^{M-k}\b\le g_k(x)$ (where $A$ is from Definition \ref{impsets}), which leads to the inequality
\[ \a^{r}|\text{Br}_{\a,\b}^{K}\cap\Omega_k\cap B_{R_k^{\frac{1}{n}}}|\lesssim_\e K^{2n} R^{\e} \frac{1}{[A^{M-k}\b]^{q-\frac{r}{2}}}|B_{R_k^{\frac{1}{n}}}|g_k(x)^{q} ,\]
where $q$ is defined by $\frac{r}{2}+q=p$. There are two possibilities for $q$, depending on whether $r=p$ or $r=2n$. If $q=\frac{p}{2}$, then $r=p$ and since $g_k\lesssim R^{C\e}g_{k+1}\lesssim_\e R^{C\e}\b$ on $\Omega_k$, we have 
\[ \sum_{B_{R_k^{\frac{1}{n}}}} \a^p|\text{Br}_{\a,\b}^K\cap B_{R_k^{\frac{1}{n}}}|\lesssim_\e R^{C\e}K^{2n} |\mc{N}_{R_k^{\frac{1}{n}}}(\text{Br}_{\a,\b}^K)|\b^{\frac{p}{2}}\lesssim_\e R^{C\e}K^{2n} \int_{\R^n}|\sum_\theta|f_\theta|^2*\w_{\theta,D}|^{\frac{p}{2}} ,\]
where we used the locally constant property in the final inequality. 

It remains to treat the case that $r=2n$, meaning that $n+q=p\ge 2n$ and $q$ therefore satisfies $2\le q\le p_{n-1}$. By Corollary \ref{highdom}, we also have the upper bound $|g_k(x)|\le 2|g_k^h(x)|$, so that 
\[ \a^{2n}|\text{Br}_{\a,\b}^{K}\cap\Omega_k\cap B_{R_k^{\frac{1}{n}}}|\lesssim_\e K^{2n} R^{\e} \frac{1}{[A^{M-k}\b]^{q-n}} |B_{R_k^{\frac{1}{n}}}||g_k^h(x)|^{q} .\]
By the locally constant property applied to $g_k^h$, we have $|g_k^h|^{q}\lesssim_\e |g_k^h*w_{ B_{R_k^{\frac{1}{n}}}}|^{q}$. By Cauchy-Schwarz, we also have $|g_k^h*w_{ B_{R_k^{\frac{1}{n}}}}|^{q}\lesssim |g_k^h|^{q}*w_{B_{R_k^{\frac{1}{n}}}}$. Combine this with the previous displayed inequality to get 
\[ \a^{2n}|\text{Br}_{\a,\b}^{K}\cap\Omega_k\cap B_{R_k^{\frac{1}{n}}}|\lesssim_\e K^{2n} R^{\e} \frac{1}{[A^{M-k}\b]^{q-n}}\int|g_k^h|^{q}W_{ B_{R_k^{\frac{1}{n}}}} .\]
Summing over the balls $B_{R_k^{\frac{1}{n}}}$ in our finitely-overlapping cover of $\text{Br}_{\a,\b}^{K}\cap\Omega_k$, we conclude that
\begin{equation}\label{peqn} \a^{2n}|\text{Br}_{\a,\b}^{K}\cap\Omega_k|\lesssim_\e K^6 R^{\e} \frac{1}{[A^{M-k}\b]^{q-n}}\int_{\R^n}|g_k^h|^{q} .\end{equation}
We are done using the properties of the set $\text{Br}_{\a,\b}^{K}\cap\Omega_k$, which is why we now integrate over all of $\R^n$ on the right hand side. We will now use Lemma \ref{high1} to analyze the high part $g_k^h$. In particular, Lemma \ref{high1} gives 
\begin{equation}\label{highapp} 
\int|g_k^h|^{q} \lesssim_\e R^{\e} \int_{\R^n}\big|\sum_{\tau_k}||f_{\tau_k}^{k+1}|^2*\w_{\tau_k,d_k}*\widecheck{\eta}_{R_{k+1}^{-\frac{1}{n}}}|^2\big|^{\frac{q}{2}} . 
\end{equation}
By Young's convolution inequality, since $\widecheck{\eta}_{>R_{k+1}^{-\frac{1}{n}}}$ is $L^1$-normalized, the right hand side above is bounded by 
\[ C_\e R^{\e} \int_{\R^n}\big|\sum_{\tau_k}||f_{\tau_k}^{k+1}|^2*\w_{\tau_k,d_k}|^2\big|^{\frac{q}{2}}\]
Next use \eqref{item2} from Lemma \ref{pruneprop} to note that $\|f_{\tau_k}^{k+1}\|_\infty\le \sum_{\tau_{k+1}\subset\tau_k}\|f_{\tau_{k+1}}^{k+1}\|_\infty\lesssim R^\e A^{\e^{-1}}K^n\frac{\b}{\a}$: 
\begin{align*}
&\int_{\R^n}\big|\sum_{\tau_k}||f_{\tau_k}^{k+1}|^2*\w_{\tau_k,d_k}|^2\big|^{\frac{q}{2}} \lesssim_\e[(A^{\e^{-1}}R^{\e}K^n\frac{\b}{\a})]^{q-n}\int_{\R^n}\big|\sum_{\tau_k}||f_{\tau_k}^{k+1}|^2*\w_{\tau_k,d_k}|^{2-\frac{q-n}{q}}\big|^{\frac{q}{2}}. 
\end{align*} 
Since the exponent $2-\frac{q-n}{q}=\frac{p}{q}\ge 1$, the integral on the right hand side above is bounded by 
\[ \int_{\R^n}|\sum_{\tau_k}|f_{\tau_k}^{k+1}|^2*\w_{\tau_k,d_k}|^{\frac{p}{2}}. \]
Combining this with \eqref{peqn}, \eqref{highapp}, and Proposition \ref{algo}, the summary of the argument from this case is 
\[\a^{p}|U_{\a,\b}|\lesssim_\e K^{2n} R^{2\e}(A^{\e^{-1}}R^{\e}K^n)\big(\text{T}_n(R^\e)\big)^{N-k+1}\int_{\R^n}\big|\sum_{\theta\in{\bf{S}}_n(R^{-1/n})}|f_\theta|^2*\w_{\theta,D}\big|^{\frac{p}{2}}. \]
Since $k>N_0$, this upper bound has the desired form.

\noindent\underline{Case 2: bounding $|U_{\a,\b}\cap L|$.} 
Begin by using Lemma \ref{ftofk} to bound
\[ \a^{p}|\text{Br}_{\a,\b}^K\cap L|\lesssim K^{p} \int_{U_{\a,\b}\cap L}|f^{N_0+1}|^{p}. \]
Then use Cauchy-Schwarz and the locally constant property for $g_{N_0}$ :
\[ \int_{U_{\a,\b}\cap L}|f^{N_0+1}|^{p}\lesssim (R^{ N_0\e/n})^{p/2} \int_{U_{\a,\b}\cap L}(\sum_{\tau_{N_0}}|f_{\tau_{N_0}}^{N_0+1}|^2)^{p/2}\lesssim (R^{ N_0\e/n})^{p/2} \int_{U_{\a,\b}\cap L}(g_{N_0})^{p/2}.\]
Using the definition the definition of $L$, we bound the factors of $g_{N_0}$ by
\[ \int_{U_{\a,\b}\cap L} (A^{\e^{-1}}\b)^{p/2}. \]
Finally, by the definition of $U_{\a,\b}$, conclude that
\[    \a^{p}|\text{Br}_{\a,\b}^K\cap L|\lesssim_\e K^{p}R^{2N_0\e}A^{\e^{-1}}\int_{\R^n}|\sum_\theta|f_\theta|^2*\w_{\theta,D}|^{p/2}. \]

\end{proof}

\subsection{Proof of Proposition \ref{momcurveinduct} from Proposition \ref{mainprop} }

\begin{proposition}\label{S1bd} For any $\e>0$, $R\ge 1$, and $D\in\N_{>0}$, 
\[ \emph{T}_{n,D}^w(R)\lesssim_{\e,D} R^\e. \]
\end{proposition}
In order to make use of Proposition \ref{mainprop}, we need to reduce to the case that our function $f$ is localized to a ball, its wave packets have been pigeonholed so that $\|f_\theta\|_\infty\lesssim 1$ for all $\theta\in{\bf{S}}_n(R^{-\frac{1}{n}})$, and we have approximated $\|f\|_p$ by an expression involving a superlevel set. This is the content of the following subsection. 

\subsection{Wave packet decomposition and pigeonholing \label{M3pigeon}}

We describe standard reductions from dyadic pigeonholing. The following lemmas have direct analogues in \cite{maldagueM3}, so we omit their proofs. 

The following spatial localization lemma is analogous to Lemma 6.2 in \cite{maldagueM3}. 
\begin{lemma}\label{loc} For any $R$-ball $B_R\subset\R^n$, suppose that 
\[     \|f\|_{L^p(B_R)}^p\lesssim_{\e,D}  R^\e\int\big|\sum_{\theta\in{\bf{S}}_n(R^{-\frac{1}{n}})}|f_\theta|^2*\w_{\theta,D}\big|^{\frac{p}{2}} \]
for any $2\le p\le p_n$ and any Schwartz function $f:\R^n\to\C$ with Fourier transform supported in $\mc{M}^n(R)$.  Then Proposition \ref{S1bd} is true. 
\end{lemma}

The following weak, level-set version of Proposition \ref{S1bd} is analogous to Lemma 6.3 in \cite{maldagueM3}.  
\begin{lemma}\label{alph} For each $B_R$ and Schwartz function $f:\R^n\to\C$ with Fourier transform supported in $\mc{M}^n(R)$, and each $2\le p\le p_n$, there exists $\a>0$ such that
\[ \|f\|_{L^p(B_R)}^p\lesssim (\log R)\a^p|\{x\in B_R:\a\le |f(x)|\}|+R^{-500n}\int|\sum_{\theta\in{\bf{S}}_n(R^{-\frac{1}{n}})} |f_\theta|^2*\w_{\theta,D}|^{p/2}. \]

\end{lemma}

Continue to use the notation 
\[ U_{\a}=\{x\in B_R:\a\le |f(x)|\}. \]
We will show that to estimate the size of $U_{\a}$, it suffices to replace $f$ with a version whose wave packets at scale $\theta$ have been pigeonholed. Write 
\begin{align}\label{sum} f=\sum_\theta\sum_{T\in\T_\theta}\s_Tf_\theta \end{align}
where for each $\theta\in{\bf{S}}_n(R^{-\frac{1}{n}})$, $\{\s_T\}_{T\in\T_\theta}$ is the partition of unity from \textsection\ref{prusec}. If $\a\le C_\e R^{-100n}\max_\theta\|f_\theta\|_{\infty}$, then using a similar argument that bounds the second expression in the proof of Lemma \ref{alph}, the inequality
\[ \a^p|U_{\a}|\lesssim_\e R^\e\int|\sum_\theta|f_\theta|^2*\w_{\theta,D}|^{p/2} \]
is trivial. 

The following proposition about wave packet decomposition is analogous to Lemma 6.4 in \cite{maldagueM3}. 
\begin{proposition}[Wave packet decomposition] \label{wpd} Let ${\a}>C_\e R^{-100n}\max_\theta\|f_\theta\|_{L^\infty(\R^n)}$. There exist subsets $\tilde{\T}_\theta\subset\T_\theta$, as well as a constant $A>0$ with the following properties:
\begin{align} 
|U_{\a}|\lesssim (\log R)|\{x\in U_{\a}:\,\,{{\a}}&\lesssim |\sum_{\theta\in{\bf{S}}_n(R^{-\frac{1}{n}})}\sum_{T\in\tilde{\T}_\theta}\s_T(x)f_\theta (x)|\,\,\}|, \\
R^\e T\cap U_{\a}\not=\emptyset\qquad&\text{for all}\quad\theta\in{\bf{S}}_n(R^{-\frac{1}{n}}),\quad T\in\tilde{\T}_\theta\\
A\lesssim \|\sum_{T\in\tilde{\T}_\theta}\s_Tf_\theta\|_{L^\infty(\R^n)}&\lesssim R^{3\e} A\qquad\text{for all}\quad  \theta\in{\bf{S}}_n(R^{-\frac{1}{n}})\label{propM}\\
\|\s_Tf_\theta\|_{L^\infty(\R^n)}&\sim A\qquad\text{for all}\quad  \theta\in{\bf{S}}_n(R^{-\frac{1}{n}}),\quad T\in\tilde{\T}_\theta. \label{prop'M}
\end{align}

\end{proposition}

The following is analogous to Corollary 6.5 in \cite{maldagueM3}. 
\begin{corollary}\label{wpdcor} Let $f$, $\a$, $\tilde{T}_\theta$ and $A>0$ be as in Proposition \ref{wpd}. Then for each $x\in U_\a$,
\[ \a\le R^{103n\e}\frac{1}{A} \sum_{\theta\in\mc{S}}|\sum_{T\in\tilde{\T}_\theta}\s_T f_\theta|^2*\w_{\theta,D}(x). \]
\end{corollary}

Using the locally constant property for $\sum_\theta|f_\theta|^2$ and dyadic pigeonholing, we have the following, which is analogous to Lemma 6.6 in \cite{maldagueM3}. 
\begin{lemma}\label{bet} For each $2\le p\le p_n$, $\a>0$, $B_R$, and Schwartz function $f:\R^n\to\C$ with Fourier transform supported in $\mc{M}^n(R)$, there exists $\b>0$ such that $\a^p|\{x\in B_R:\a\le|f(x)|\}|$ is bounded by 
\[C(\log R)\a^{p}|\{x\in B_R:\a\le|f(x)|,\quad\b/2\le\sum_{\theta}|f_\theta|^2*\w_{\theta,D}(x)\le \b\}|+R^{-500n}\int|\sum_{\theta}|f_\theta|^2*\w_{\theta,D}|^{p/2} .\]
\end{lemma}

\subsection{A multi-scale inequality for $\text{T}_{n,D}^w(R)$ implying Proposition \ref{S1bd} \label{ind} }

First we use a broad/narrow analysis to prove a multi-scale inequality for $\text{T}_n^w(R)$. 

\begin{lemma}\label{multi-scale'} For any $D\in\N_{>0}$, $1\le K^n\le R$, and $1\le N_0\le \e^{-1}$,
\[ \emph{T}_{n,D}^w(R)\lesssim(\log R)^2\left( K^{53}\big[R^{10 \e N_0}A^{\e^{-1}} + R^{4\e^2+200\e}A^{\e^{-1}}  \emph{T}_{n,D}^w(R^\e)^{\e^{-1}-N_0}\big]+\emph{T}_{n,D}^w(R/K^3)\right) .\]
\end{lemma}

\begin{proof}[Proposition \ref{mainprop} implies Lemma \ref{multi-scale'}]
Let $f:\R^n\to\C$ be a Schwartz function with Fourier transform supported in $\mc{M}^n(R)$. Let $2\le p\le p_n$. By Lemma \ref{loc}, it suffices to bound $\|f\|_{L^p(B_R)}^p$ instead of $\|f\|_{L^p(\R^n)}^p$. By Lemma \ref{alph}, we may fix $\a>0$ so that $\|f\|_{L^p(B_R)}^p\lesssim(\log R)^2\a^p|U_{\a}|$. By Proposition \ref{wpd}, we may replace $\a$ by $\a/A$ and replace $f$ by $\tilde{f}=\frac{1}{A}\sum_{\theta\in{\bf{S}}_n(R^{-\frac{1}{n}})}\sum_{T\in\tilde{\T}_\theta}\s_Tf_\theta$ where $\tilde{\T}_\theta$ satisfies the properties in that proposition. From here, we will take $f$ to mean $\tilde{f}$. By Lemma \ref{bet}, we may fix $\b>0$ so that $\a^p|U_\a|\lesssim (\log R)\a^p|U_{\a,\b}|$. Finally, 
by Corollary \ref{wpdcor}, we have $\a\lesssim R^{103n\e}$. 

Write $f=\sum_{\tau\in{\bf{S}}_n(K^{-1})}f_\tau$. The broad-narrow inequality is
\begin{align}\label{brnar}
    |f(x)|&\le 2n\max_{\tau\in{\bf{S}}_n(K^{-1})}|f_{\tau}(x)|+K^3\max_{\substack{d(\tau^i,\tau^j)\ge K^{-1}\\\tau^i\in{\bf{S}}_n(K^{-1})}}|\prod_{i=1}^nf_{\tau^i}(x)|^{\frac{1}{n}}  .
\end{align}
Indeed, suppose that the set $\{\tau\in{\bf{S}}_n(K^{-1}):|f_{\tau}(x)|\ge K^{-1}\max_{\tau'\in{\bf{S}}_n(K^{-1})}|f_{\tau'}(x)|\}$ has at least $2n-1$ elements. Then we can find $n$ many $\tau^i$ which are pairwise $\ge K^{-1}$-separated and satisfy $|f(x)|\le K^3|\prod_{i=1}^nf_{\tau^i}(x)|^{\frac{1}{n}}$. If there are fewer than $2n-1$ elements, then $|f(x)|\le 2n \max_{\tau\in{\bf{S}}_n(K^{-1})}|f_{\tau}(x)|$.

The broad-narrow inequality leads to two possibilities. In one case, we have  
\begin{equation}\label{case1brn} |U_{\a,\b}|\lesssim |\{x\in U_{\a,\b}:|f(x)|\le 2n\max_{\tau\in{\bf{S}}_n(K^{-1})}|f_{\tau}(x)|\}| . \end{equation}
Then the summary of inequalities from this case is
\begin{align*}
    \int_{B_R}|f|^p\lesssim(\log R)^2\a^p|U_{\a,\b}|\lesssim(\log R)^2\sum_{\tau\in{\bf{S}}_n(K^{-1})}\int_{\R^n}|f_{\tau\in{\bf{S}}_n(K^{-1})}|^p .
\end{align*}
By rescaling for the moment curve and the definition of $\text{T}_{n,D}^w(\cdot)$, we may bound each integral in the final upper bound by 
\[ \int_{\R^n}|f_{\tau}|^p\le \text{T}_{n,D}^w(R/K^n)\int_{\R^n}|\sum_{\theta\subset\tau}|f_\theta|^2*\w_{\theta,D}|^{p/2}. \]
Noting that $\sum_\tau\int_{\R^n}|\sum_{\theta\subset\tau}|f_\theta|^2*\w_{\theta,D}|^{7/2}\le\int_{\R^n}|\sum_{\theta}|f_\theta|^2*\w_{\theta,D}|^{7/2} $ finishes this case.

The remaining case from the broad-narrow inequality is that
\[ |U_{\a,\b}|\lesssim |\{x\in U_{\a,\b}:|f(x)|\le K^3 \max_{\substack{
d(\tau^i,\tau^j)\ge K^{-1}\\\tau^i\in{\bf{S}}_n(K^{-1})}}|\prod_{i=1}^nf_{\tau^i}(x)|^{\frac{1}{n}}\}| .\]
We may further assume that 
\[ |U_{\a,\b}|\lesssim |\{x\in U_{\a,\b}:|f(x)|\le K^3 \max_{\substack{
d(\tau^i,\tau^j)\ge K^{-1}\\\tau^i\in{\bf{S}}_n(K^{-1})}}|\prod_{i=1}^nf_{\tau^i}(x)|^{\frac{1}{n}},\quad \max_{\tau\in{\bf{S}}_n(K^{-1})}|f_\tau(x)|\le \a\}| \]
since otherwise, we would be in the first case \eqref{case1brn}. The size of the set above is now bounded by a sum over pairwise $K^{-1}$-separated $n$-tuples $(\tau^1,\ldots,\tau^n)$ which form $|\text{Br}_{\a,\b}^{K^3}|$ from \textsection\ref{broad}. Using Proposition \ref{mainprop} to bound $|\text{Br}_{\a,\b}^{K^3}|$, the summary of the inequalities from this case is 
\[\int_{B_R}|f|^p\lesssim(\log R)^2  K^{53}\big[R^{10 \e N_0}A^{\e^{-1}} + R^{4\e^2}A^{\e^{-1}}  \text{T}_{n,D}^w(R^\e)^{\e^{-1}-N_0}\big]\int|\sum_\theta|f_\theta|^2*\w_{\theta,D}|^{p/2}, \]
where we used that $\text{T}_n(r)\lesssim_D\text{T}_{n,D}^w(r)$, which follows from the locally constant property.

\end{proof}

With Lemma \ref{multi-scale'} in hand, we may now prove Proposition \ref{S1bd}. 
\begin{proof}[Proof of Proposition \ref{S1bd}] 
Let $\eta$ be the infimum of the set
\[ \mc{S}=\{\d\ge 0:\sup_{R\ge 1}\frac{\text{T}_{n,D}^w(R)}{R^\d}<\infty\}. \]
Suppose that $\eta>0$. Let $\e_1$, $\eta>\e_1>0$, be a parameter we will specify later. By Lemma \ref{multi-scale'}, we have
\begin{align*} 
\sup_{R\ge 1}\frac{\text{T}_{n,D}^w(R)}{R^{\eta-\e_1}}&\lesssim_\e\sup_{R\ge 1}\frac{1}{R^{\eta-\e_1}}\Big[(\log R)^2\left( K^{53}\big[R^{10 \e N_0} + R^{4\e^2+10\e}  \text{T}_{n,D}^w(R^\e)^{\e^{-1}-N_0}\big]+\text{T}_{n,D}^w(R/K^n)\right)  \Big] 
\end{align*}
where we are free to choose $\e>0$, $1\le N_0\le\e^{-1}$, and $1\le K^3\le R$. Continue to bound the expression on the right hand side by
\begin{align*}
\sup_{R\ge 1}(\log R)^2\Big( K^{53}\frac{R^{10 \e N_0}}{R^{\eta-\e_1}} +& K^{53}\frac{R^{4\e^2+10\e}}{R^{N_0\e(\eta+\e_1)-2\e_1}}  \big[\frac{\text{T}_{n,D}^w(R^\e)}{R^{\e(\eta+\e_1)}}\big]^{\e^{-1}-N_0}+\frac{1}{K^{3(\eta+\e_1)}R^{-2\e_1}}\frac{{\text{T}_{n,D}^w(R/K^n)}}{(R/K^n)^{\eta+\e_1}}  \Big).
\end{align*}
By definition of $\eta$, 
\[ \sup_{R\ge 1}\frac{\text{T}_{n,D}^w(R^\e)}{R^{\e(\eta+\e_1)}}+\sup_{R\ge 1}\frac{{\text{T}_{n,D}^w(R/K^n)}}{(R/K^n)^{\eta+\e_1}} <\infty, \]
so it suffices to check that 
\begin{align*}
\sup_{R\ge 1}(\log R)^2\Big( K^{53}\frac{R^{10 \e N_0}}{R^{\eta-\e_1}} +& K^{53}\frac{R^{4\e^2+10\e}}{R^{N_0\e(\eta+\e_1)-2\e_1}}  +\frac{1}{K^{n(\eta+\e_1)}R^{-2\e_1}}  \Big)<\infty
\end{align*}
to obtain a contradiction. From here, choose $N_0=\e^{-1/2}$ and $K=R^{\e_1}$ so that it suffices to check 
\begin{align*}
\sup_{R\ge 1}(\log R)^2\Big( \frac{1}{R^{\eta-10\e^{1/2}-54\e_1}} +& \frac{1}{R^{\e^{1/2}\eta-4\e^2-10\e-55\e_1}}  +\frac{1}{R^{(n-2)\e_1}}  \Big)<\infty .
\end{align*}
This is clearly true if we choose $\e>0$ to satisfy $\min(\eta-10\e^{1/2},\eta-4\e^{3/2}-10\e^{1/2})>\eta/2$ and then choose $\e_1$ to be smaller than $\frac{1}{55}\e^{1/2}\eta/4$. Our reasoning has shown that $\eta-\e_1\in\mc{S}$, which is a contradiction. Conclude that $\eta=0$, as desired.

\end{proof}

\subsection{Proof of Proposition \ref{momcurveinduct} \label{S2}}
We will show that $\text{T}_{n,D}^w(R)\lesssim_\e R^\e$ implies $\text{T}_n(R)\lesssim_\e R^\e$, which proves Proposition \ref{momcurveinduct}. See Definitions \ref{TnRw} and \ref{TnR} at the beginning of \textsection\ref{tools} for the definitions of $\text{T}_n(R)$ and $\text{T}_n^w(R)$. The following is a multi-scale inequality bounding $\text{T}_n(R)$ by $\text{T}_{n,D}^w(\cdot)$, and $\text{T}_n(\cdot)$ evaluated at parameters smaller than $R$. 
\begin{proposition}\label{multi-scaleS2} For $R\ge 10$ and $D\ge p_n^{n\e^{-n}}$,  
\[\emph{T}_n(R)\lesssim_\e R^\e \emph{T}_{n,D}^w(R^{\frac{1}{n}})\max_{1\le \lambda\le R^{1-\frac{1}{n}}}\emph{T}_n(\lambda). \]
\end{proposition}
Proposition \ref{momcurveinduct} follows directly from Proposition \ref{multi-scaleS2} by the proof of Theorem 1 in \cite{maldagueM3}. The proof of Proposition \ref{multi-scaleS2} has the same overall structure as the proof of Proposition 6.8 in \cite{maldagueM3}. 
\begin{proof}[Proof of Proposition \ref{multi-scaleS2}] 
Let $2\le p\le p_n$. Let $f:\R^n\to\C$ be a Schwartz function with Fourier transform supported in $\mc{M}^n(R)$. By the defining inequality for $\text{T}_{n,D}^w(R^{\frac{1}{n}})$, 
\begin{equation}\label{S1app}
    \int_{\R^n}|f|^p\le \text{T}_n^w(R^{\frac{1}{n}})\int_{\R^n}|\sum_{\tau\in{\bf{S}}_n(R^{-\frac{1}{n^2}})}|f_\tau|^2*\w_{\tau,D}|^{p/2}. 
\end{equation}
We choose the scale $R^{\frac{1}{n}}$ because each $\w_{\tau}$ is localized to an $R^{\frac{1}{n}^2}\times R^{2/n^2}\times\cdots\times R^{n/n^2}$ plank, which is contained in an $R^{\frac{1}{n}}$ ball. The square function we are aiming for, $\sum_{\theta\in{\bf{S}}_n(R^{-\frac{1}{n}})}|f_\theta|^2$, is locally constant on $R^{\frac{1}{n}}$ balls, so we will be able to eliminate the weights and therefore obtain a bound for $\text{T}_n(R)$. The idea for going from the right hand side of \eqref{S1app} to our desired right hand side is to perform a simplified version of the proof of Proposition \ref{algo}. 

Begin with the assumption that 
\begin{equation}\label{assS2}
\int_{\R^n}|\sum_{\tau\in{\bf{S}}_n(R^{-\frac{1}{n^2}})}|f_\tau|^2*\w_{\tau,D}|^{p/2}\lesssim  \int_{\R^n}|\sum_{\tau\in{\bf{S}}_n(R^{-\frac{1}{n^2}})}|f_\tau|^2*\w_{\tau,D}*\widecheck{\eta}_{>R^{-\frac{1}{n}}}|^{p/2}.\end{equation}
Indeed, if this does not hold, then 
\[\int_{\R^n}|\sum_{\tau\in{\bf{S}}_n(R^{-\frac{1}{n^2}})}|f_\tau|^2*\w_{\tau,D}|^{p/2}\lesssim  \int_{\R^n}|\sum_{\tau\in{\bf{S}}_n(R^{-\frac{1}{n^2}})}|f_\tau|^2*\w_{\tau,D}*\widecheck{\eta}_{\le R^{-\frac{1}{n}}}|^{p/2}.\]
This is a termination criterion case since by local $L^2$-orthogonality, for each $\tau\in{\bf{S}}_n(R^{-\frac{1}{n^2}})$,
\[ ||f_\tau|^2*\w_{\tau,D}*\widecheck{\eta}_{\le R^{-\frac{1}{n}}}(x)|\lesssim \sum_{\theta\subset\tau}|f_\theta|^2*\w_{\tau,D}*|\widecheck{\eta}_{R^{-\frac{1}{n}}}|(x).  \]
Finally, simply note that $\w_{\tau,D}*|\widecheck{\eta}_{R^{-\frac{1}{n}}}|\lesssim w_{R^{\frac{1}{n}}}$ and by Young's convolution inequality, 
\[ \int_{\R^n}|\sum_\theta|f_\theta|^2*w_{R^{\frac{1}{n}}}|^{p/2}\lesssim \int_{\R^n}|\sum_\theta|f_\theta|^2|^{p/2}. \]
From here on, assume that \eqref{assS2} holds.

Now we describe the simplified algorithm. Let $\d>0$ be a constant that we specify later in the proof. At intermediate step $m$, we have the inequality 
\begin{align}\label{stepmS2} 
\text{(R.H.S. of \eqref{assS2})}
\le (C_\d R^{C\d})^m s_m^{-C\e}\int_{\R^n}(\sum_{\tau\in{\bf{S}}_n(R^{-\frac{1}{n^2}})} \sum_{\substack{\tau'\subset\tau \\\tau'\in{\bf{S}}_n(s_m)}}|f_{\tau'}|^2*\w_{\tau,p_n^{-m}D} \Big)^{\frac{p}{2}}
\end{align}
in which $R^{-\frac{1}{n}}\le s_m\le R^{-\frac{1}{n^2}}$ and $s_m\le R^{-\e^n/n}s_{m-1}$. Note that \eqref{stepmS2} clearly holds with $m=0$ and $s_m=R^{-\frac{1}{n^2}}$. We will show that assuming \eqref{stepmS2} for $m-1$, either \eqref{stepmS2} holds for $m$ or the iteration terminates. 

\noindent\fbox{Step m:} Suppose that \eqref{stepmS2} holds with $m-1$, so 
\begin{align*} 
(\text{L.H.S. of \eqref{assS2}})\le (C_\d R^{C\d})^{m-1}s_{m-1}^{-C\e}  &  \int_{\R^n}(\sum_{\substack{\tau\in{\bf{S}}_n(R^{-\frac{1}{n^2}})}}\sum_{\substack{\tau'\subset\tau\\\tau'\in{\bf{S}}_n(s_{m-1})}} |f_{\tau'}|^{2}*\w_{\tau,p_n^{-(m-1)}D})^{\frac{p}{2}} .\end{align*}
Apply Lemma \ref{algo2} to the integral, yielding 
\begin{align}\label{stephypS2} 
(\text{L.H.S. of \eqref{assS2}})&\le (C_\d R^{2\d})^{m-1}R^{\d/n} s_{m-1}^{-C\e} \left[  \int_{\R^n}(\sum_{\tau\in{\bf{S}}_n(R^{-\frac{1}{n^2}})}\sum_{\substack{\tau'\subset\tau\\ \tau'\in{\bf{S}}_n(R^{-\frac{\e^n}{n}}s_{m-1})}}|f_{\tau'}|^2*\w_{\tau,p_n^{-m}D})^{\frac{p}{2}}\right.\\
    & \left.+R^{C\e^n/n}\int_{\R^n}(\sum_{\tau\in{\bf{S}}_n(R^{-\frac{1}{n^2}})}\sum_{\substack{\tau'\subset\tau\\ \tau'\in{\bf{S}}_n(s_{m-1})}}|f_{\tau'}|^{\tilde{p}_2}*\w_{\tau,p_n^{-(m-1)}D})^{\frac{p}{\tilde{p}_{2}}}\right] .\nonumber
\end{align}
If the first term on the right hand side dominates, then we have produced the step $m$ inequality 
\[ (\text{L.H.S. of \eqref{assS2}})\le (C_\d R^{2\d})^{m} s_m^{-\e}\int_{\R^n}(\sum_{\substack{\tau\in{\bf{S}}_n(R^{-\frac{1}{n}})}}\sum_{\substack{\tau'\subset\tau\\\tau'\in{\bf{S}}_n(R^{-\frac{\e^n}{n}}s_{m-1})}}|f_{\tau'}|^2*\w_{\tau,p_n^{-m}D})^{\frac{p}{2}} . \]
If the second term dominates, we have
\[ (\text{L.H.S. of \eqref{assS2}})\le (C_\d R^{2\d})^{m-1}R^{\d/n}s_{m-1}^{-C\e}R^{C\e^n/n}\int_{\R^n}(\sum_{\tau\in{\bf{S}}_n(R^{-\frac{1}{n^2}})}\sum_{\substack{\tau'\subset\tau \\\tau'\in{\bf{S}}_n(s_{m-1})}}|f_{\tau'}|^{\tilde{p}_{2}}*\w_{\tau,p_n^{-(m-1)}D})^{\frac{p}{\tilde{p}_{2}}} . \]
Note that we may again apply Lemma \ref{algo2}, with $\e^n$ replaced by $\e^{n-1}$, to the integral on the right hand side. We iterate this process until we have one of two outcomes. The first outcome is that for some $j<n$, 
\[ (\text{L.H.S. of \eqref{assS2}})\le (C_\d R^{2\d})^{m-1}R^{j\d/n}s_{m-1}^{-C\e}R^{C\e^{n-j}/n} \int_{\R^n}(\sum_{\tau\in{\bf{S}}_n(R^{-\frac{1}{n^2}})}\sum_{\substack{\tau'\subset\tau\\ \tau'\in{\bf{S}}_n(s_{m})}}|f_{\tau'}|^2*\w_{\tau,p_n^{-m}D})^{\frac{p}{2}} \]
where $R^{C\e^{n-j}/n}=(s_{m-1}/s_m)^{C\e}$. Step $m$ is proved in this case. 

The second outcome is that 
\[ (\text{L.H.S. of \eqref{assS2}})\le (C_\d R^{2\d})^{m-1}R^{j\d/n}s_{m-1}^{-C\e}R^{C\e^{n+2-l}} \int_{\R^n}(\sum_{\tau\in{\bf{S}}_n(R^{-\frac{1}{n^2}})}\sum_{\substack{\tau'\subset\tau\\ \tau'\in{\bf{S}}_n(s_{m-1})}}|f_{\tau'}|^{\tilde{p}_{l}}*\w_{\tau,D})^{\frac{p}{\tilde{p}_{l}}} \]
with $1\le \frac{p}{\tilde{p}_l}\le 2$ and $2\le \frac{p}{\tilde{p}_{l-1}}$, which means that $l< n$. As in the proof of Lemma \ref{algo2}, use bump functions $\eta_{\tau',\sigma}^{n-1}$ and suppose that 
\begin{align*} 
\int_{\R^n}(\sum_{\tau\in{\bf{S}}_n(R^{-\frac{1}{n^2}})}&\sum_{\substack{\tau'\subset\tau\\ \tau'\in{\bf{S}}_n(s_{m-1})}}|f_{\tau'}|^{\tilde{p}_{l}}*\w_{\tau,p_n^{-(m-1)}D})^{\frac{p}{\tilde{p}_{l}}}\lesssim (\log R)^C\\
    &\times \int_{\R^n}|\sum_{\tau\in{\bf{S}}_n(R^{-\frac{1}{n^2}})}\sum_{\substack{\tau'\subset\tau\\ \tau'\in{\bf{S}}_n(s_{m-1})}}|f_{\tau'}|^{\tilde{p}_{l}}*\w_{\tau,p_n^{-(m-1)}D}*\widecheck{\eta}_{\tau',\sigma}^{n-1}|^{\frac{p}{\tilde{p}_{l}}}. 
\end{align*}
If $\sigma\le  R^{-\e^{n+1-l}}s_{m-1}$, then by Lemma \ref{algo1}, 
\begin{align*} 
\int_{\R^n}|&\sum_{\tau\in{\bf{S}}_n(R^{-\frac{1}{n^2}})}\sum_{\substack{\tau'\subset\tau\\ \tau'\in{\bf{S}}_n(s_{m-1})}}|f_{\tau'}|^{\tilde{p}_{l}}*\w_{\tau,p_n^{-(m-1)}D}*\widecheck{\eta}_{\tau',\sigma}^{n-1}|^{\frac{p}{\tilde{p}_{l}}} \\
    &\le  C_\d R^{\d/n} \int_{\R^n}|\sum_{\tau\in{\bf{S}}_n(R^{-\frac{1}{n^2}})}\sum_{\substack{\tau'\subset\tau \\\tau'\in{\bf{S}}_n(s_{m-1})}}|\sum_{\substack{\tau''\subset\tau'\\ \tau''\in{\bf{S}}_n(R^{-\frac{\e^{n+1-l}}{n}}s_{m-1})}} |f_{\tau''}|^2*\w_{\tau,p_n^{-m}D}|^{\frac{\tilde{p}_{l}}{2}}*\tilde{W}_{\tau',\sigma}^{n-1}|^{\frac{p}{\tilde{p}_{l}}}.  
\end{align*}
Apply Proposition \ref{multilem3pf} to get rid of $\tilde{W}_{\tau,\sigma}^{n-1}$, concluding step $m$.

The other case is that $\sigma>R^{-\e^{n+1-l}}s_{m-1}$. Using the boundedness of $\C^n_{n-2}(\cdot)$ and Cauchy-Schwarz, we have
\begin{align*} 
\int_{\R^n}|\sum_{\tau\in{\bf{S}}_n(R^{-\frac{1}{n^2}})}\sum_{\substack{\tau'\subset\tau\\\tau'\in{\bf{S}}_n(s_{m-1})}}&|f_{\tau'}|^{\tilde{p}_{l}}*\w_{\tau,D}*\widecheck{\eta}_{\tau',\sigma}^{n-1}|^{\frac{p}{\tilde{p}_{l}}} \\
&\lesssim_\e R^{C\e^2}\int_{\R^n}\sum_{\tau\in{\bf{S}}_n(R^{-\frac{1}{n^2}})} \sum_{\substack{\tau'\subset\tau \\ \tau'\in{\bf{S}}_n(s_{m-1})}} ||f_{\tau'}|^{\tilde{p}_{l}}*\w_{\tau,D}*\widecheck{\eta}_{\tau,\sigma}^{n-1}|^{\frac{p}{\tilde{p}_{l}}} .
\end{align*}
By Young's convolution inequality, the right hand side is bounded by 
\[ C_\e R^{C\e^2}\int_{\R^n}\sum_{\tau\in{\bf{S}}_n(R^{-\frac{1}{n^2}})}\sum_{\substack{\tau'\subset\tau\\\tau'\in{\bf{S}}_n(s_{m-1})}} |f_{\tau'}|^{p} ,\]
which, by the definition of $\text{T}_n(\cdot)$ and rescaling, is itself bounded by 
\[ C_\e R^{C\e^2}\text{T}_n(s_{m-1}^nR)\int_{\R^n}\sum_{\tau\in{\bf{S}}_n(s_{m-1})} (\sum_{\substack{\theta\subset\tau\\\theta\in{\bf{S}}_n(R^{-\frac{1}{n}})}}|f_{\theta}|^2 )^{\frac{p}{2}} . \]
By Lemma \ref{pruneprop}, $\|\cdot\|_{\ell^{p/2}}\le \|\cdot\|_{\ell^1}$, this expression is bounded by 
\begin{align*} 
C_\e R^{C\e^2}\text{T}_n(s_{m-1}^n R)\int_{\R^n}(\sum_{\theta\in{\bf{S}}_n(R^{-\frac{1}{n}})} |f_{\theta}|^2)^{\frac{p}{2}},  
\end{align*}
which terminates the algorithm. This concludes the justification of step $m$.

\noindent\fbox{Termination criteria.} There are two outcomes of the algorithm. The first, as we have just seen, is that
\[\text{(R.H.S. of \eqref{assS2})}
\le (C_\d R^{2\d})^{m}s_{m-1}^{-\e} C_\e R^{C\e^2}\text{T}_n(s_{m-1}^nR) \int_{\R^n}(\sum_{\theta\in{\bf{S}}_n(R^{-\frac{1}{n}})} |f_{\theta}|^2)^{\frac{p}{2}}
 \]
in which $mn\e^{-n}$ and $s_{m-1}^{-\e}\le R^\e$. Choose $\d<\e^n$ guarantees an upper bound of the form $C_\e R^\e \text{T}_n(s_{m-1}^nR)$. The second outcomes is that in fewer than $n\e^{-n}$ many steps, the algorithm terminates with the inequality 
\[\text{(R.H.S. of \eqref{assS2})}
\le (C_\d R^{C\d})^m R^\e\int_{\R^n}\Big(\sum_{\tau\in{\bf{S}}_n(R^{-\frac{1}{n^2}})} \sum_{\substack{\theta\subset\tau \\\theta\in{\bf{S}}_n(R^{-\frac{1}{n}})}}|f_{\theta}|^2*\w_{\tau,p_n^{-n\e^{-n}}D} \Big)^{\frac{p}{2}}. 
 \]
By the locally constant property, we have
\[ \int_{\R^n}(\sum_{\tau\in{\bf{S}}_n(R^{-\frac{1}{n^2}})} \sum_{\substack{\theta\subset\tau \\\theta\in{\bf{S}}_n(R^{-\frac{1}{n}})}}|f_{\theta}|^2*\w_{\tau,p_n^{-n\e^{-n}}D} \Big)^{\frac{p}{2}}\lesssim \int_{\R^n}|\sum_{\tau\in{\bf{S}}_n(R^{-\frac{1}{n^2}})} \sum_{\substack{\theta\subset\tau \\\theta\in{\bf{S}}_n(R^{-\frac{1}{n}})}}|f_{\theta}|^2*\w_{\tau,p_n^{-n\e^{-n}}D}*\widecheck{\eta}_{\le R^{-\frac{1}{n}}} |^{\frac{p}{2}}. \]
As before, $\w_{\tau,p_n^{-n\e^{-n}}D}*|\widecheck{\eta}_{\le R^{-\frac{1}{n}}}|\lesssim w_{R^{-\frac{1}{n}}}$ for each $\tau$. Therefore, by Young's convolution inequality, 
\[ \int_{\R^n}|\sum_{\tau\in{\bf{S}}_n(R^{-\frac{1}{n^2}})} \sum_{\substack{\theta\subset\tau \\\theta\in{\bf{S}}_n(R^{-\frac{1}{n}})}}|f_{\theta}|^2*\w_{\tau,p_n^{-n\e^{-n}}D}*\widecheck{\eta}_{\le R^{-\frac{1}{n}}} |^{\frac{p}{2}}\lesssim \int_{\R^n}| \sum_{\substack{\theta\in{\bf{S}}_n(R^{-\frac{1}{n}})}}|f_{\theta}|^2|^{\frac{p}{2}}, \]
which finishes the proof. 
\end{proof}

\bibliographystyle{alpha}
\bibliography{arxivdraft}

\end{document}